\documentclass[11pt]{article}
\usepackage[utf8]{inputenc}
\usepackage[english]{babel}
\usepackage{authblk}
\usepackage{graphicx}        % standard LaTeX graphics tool
\usepackage{subfigure}                            % when including figure files
\usepackage{multicol}        % used for the two-column index
\usepackage[bottom]{footmisc}% places footnotes at page bottom
\usepackage{placeins}
\usepackage{changes,todonotes}
\usepackage{pgf}
\usepackage{psfrag}
\usepackage{pgfplots}
\pgfplotsset{compat=1.18}

\usepackage{hyperref}
\usepackage{appendix}
\usepackage{amsmath,amssymb,amsfonts}
\usepackage{amsthm}
\usepackage{bm}
\usepackage{mathtools}
\usepackage{color}
\usepackage{stmaryrd}
\usepackage{blindtext}
\usepackage{caption}
\usepackage{multirow}
\usepackage{float}
\usepackage[margin=2cm]{geometry}
\usepackage{mathrsfs}%

\usepackage{xcolor}%
\usepackage{textcomp}%
\usepackage{manyfoot}%
\usepackage{booktabs}%
\usepackage{tikz}

\newtheorem{theorem}{Theorem}%  meant for continuous numbers
\newtheorem{corollary}[theorem]{Corollary}%  meant for continuous numbers

\newtheorem{example}{Example}%
\newtheorem{remark}{Remark}%
\newtheorem{lemma}{Lemma}%

\newcommand{\trinorm}[1]{{\vert\kern-0.25ex\vert\kern-0.25ex\vert #1 \vert\kern-0.25ex\vert\kern-0.25ex\vert}}
\makeindex             % used for the subject index
\begin{document}

\title{Discontinuous Galerkin time integration for second-order differential problems: formulations, analysis, and analogies}

\author[$\star$]{Gabriele Ciaramella}
\author[$\diamond$]{Martin J. Gander}
\author[$\star$]{Ilario Mazzieri}

\affil[$\star$]{MOX-Laboratory for Modeling and Scientific Computing, Department of Mathematics, Politecnico di Milano, P.za L. da Vinci 32, 20133, Milan, Italy.}

\affil[ ]{\texttt{gabriele.ciaramella@polimi.it}, \quad \texttt{ilario.mazzieri@polimi.it}}

\affil[$\diamond$]{Section de Mathématiques, Université de Génève, 2-4 rue du Lièvre, 1211, Genève 4,   Switzerland.}
\affil[ ]{\texttt{martin.gander@unige.ch}}

\maketitle


\maketitle	
	
\abstract{We thoroughly investigate 
Discontinuous Galerkin (DG) discretizations as time integrators for second-order oscillatory systems, considering both second-order and first-order formulations of the original problem. Key contributions include new convergence analyses for the second-order formulation and equivalence proofs between DG and classical time-stepping schemes (such as Newmark schemes and general linear methods). In addition, the chapter provides a detailed review and convergence analysis for the first-order formulation, alongside comparisons of the proposed schemes in terms of accuracy, consistency, and computational cost.}

\section{Introduction}\label{sec:intro}
We review in this chapter time-discontinuous Galerkin finite element methods for the numerical
approximation of linear systems of second-order ordinary differential equations (ODEs). Such systems of ODEs typically come from space discretizations of wave propagation problems in acoustics, elastodynamics, or electromagnetism.
Conventional techniques for numerically integrating second-order ODEs usually use implicit or explicit finite differences, Runge-Kutta, and Newmark methods (see \cite{LeVeque2007,quarteroni2008numerical,Butcher2008} for a comprehensive overview). In various engineering scenarios, explicit methods are generally favored over implicit ones due to their computational efficiency. Despite implicit methods being (in general) unconditionally stable, explicit methods are computationally cheaper. However, explicit methods are limited by the time step constraints imposed by the Courant–Friedrichs–Lewy (CFL) condition. This constraint, determined by the spatial discretization parameters and media properties, can significantly impact the computational effectiveness.
One potential approach to mitigate the CFL constraint is implementing suitable local time stepping (LTS) algorithms \cite{Collino2003,Diaz2009,Grote2013270}, where a smaller time step is used only when necessary according to the local CFL condition. 
Alternatively, explicit LTS methods could be adopted, by using either high-order derivative discontinuous Galerkin approaches \cite{Taube2009}, or Tent Pitching algorithms \cite{gopalakrishnan2017mapped,unmapped}.
%\marginpar{\MG{'latter scheme' refers to tent pitching, which seems not correct}}
%The latter scheme begins with an initial condition on the computational grid. Then, it proceeds to move forward in time using a chosen time-stepping scheme, such as Euler or Runge-Kutta methods. At each time step, the solution is updated based on the values of neighboring grid points, typically using finite difference or finite volume stencils to represent the local discretization of the partial differential equation in space.

In contrast to the aforementioned methods, we analyze implicit time-integration schemes proposed in \cite{DalSanto18,AntoniettiMiglioriniMazzieri2021} that achieve arbitrarily high accuracy, based on a discontinuous Galerkin (DG) approach.
DG methods were initially introduced to approximate hyperbolic problems in space \cite{ReedHill1973} and subsequently extended to handle elliptic and parabolic equations \cite{wheeler78,arnold1982}. 
The DG method was also employed to tackle initial-value problems, cf. the pioneering work in \cite{lesaint1974finite}. In time-dependent scenarios, information flows forward in time, and solutions exhibit a causal behavior, relying on past events but not future ones. Unlike finite difference time integration schemes, where the current solution is influenced by previous steps, time discontinuous Galerkin methods, operating over time intervals $I_n=[t_n, t_{n+1}]$, establish a causal system where the solution in $I_n$ depends exclusively on the solution in $I_{n-1}$.  
The use of DG schemes as time integrators is motivated by the fact that they can be easily combined with DG (or other) space discretization methods to obtain a high-order space-time finite element formulation. 
Space-time finite elements for hyperbolic problems are typically constructed based on a system of first-order differential equations, as shown for example in  \cite{hughes1988space,Banks2014,johnson1993discontinuous,lowrie1998space,kocher2014variational,dorfler2016space,gopalakrishnan2017mapped,perugia2020tent}. On the other hand, only a limited number of recent studies have addressed finite element approximations of second-order differential systems \cite{hulbert1990space,thompson1996space,Adjerid2011837,Yang2012,Walkington2014,steinbach2019stabilized,steinbach2022generalized}. 

We review here the main aspects of DG discretizations as time integrators of second-order oscillatory systems. In particular, we consider two different formulations that arise by considering the original problem in its second-order (i) or first-order formulation (ii).
%\marginpar{\MG{we should talk about chapter, not paper}}
The main novelties of the chapter can be summarized as follows. For (i), we present a new framework for convergence analysis, addressing issues related to consistency and accuracy. We show the equivalence between the DG formulation with linear finite elements and the Newmark scheme, and propose optimization strategies for the DG method. Moreover, due to the unconventional algebraic structure of DG for second-order in time problems, the spectral properties of the time-stepping matrix have to be taken into account. This leads to a new analysis for this class of problems. Finally, by introducing a new theoretical framework for Generalized Linear Methods (GLMs) we show the equivalence between general order DG schemes and GLMs.
For (ii) we present an in-depth review from \cite{lesaint1974finite}, supplemented with a new algebraic proof. Moreover, we introduce the Lobatto IIIC scheme \cite{Jay2015} and present a related convergence analysis.  Finally, we refer the reader to \cite{AntoniettiArtoniCiaramellaMazzieri2025} for a review of DG time-stepping methods applied to different wave propagation problems.

Our chapter is organized as follows. In Section~\ref{sec:review} we review some classical integrators, e.g., the Newmark and the Runge-Kutta schemes, and briefly recall the classical GLM framework. Next, we analyze the time DG method for the second-order equation in Section~\ref{sec:dG:second}. We discuss its property of accuracy and consistency and show the analogy with the Newmark or the GLM schemes. 
In Section~\ref{sec:dG:first}, we consider the DG schemes as time integrators of first-order systems and show their equivalence with special implicit Runge-Kutta schemes. Finally, in Section~\ref{sec:comparison} we compare the proposed schemes from the point of view of accuracy, consistency, and computational cost, and in Section~\ref{sec:conclusions} we draw some conclusions. 

%%%%%%%%%%%%%%%%%%%%%%%%%%%%%%%%%%%%%%%%
%\subsection{Notation}\label{sec:notation}
%%%%%%%%%%%%%%%%%%%%%%%%%%%%%%%%%%%%%%%%
Throughout the chapter, we denote by $\|\bm a\|$ the Euclidean norm of a vector $\bm a \in \mathbb{R}^d$, $d\ge 1$ and by $\|A\|$ the $\ell^{2}$-norm of a matrix $A\in\mathbb{R}^{m\times n}$, $m,n\ge1$. For a given $I\subset\mathbb{R}$ and $v: I\rightarrow\mathbb{R}$ we denote by $L^p(I)$ and $H^p(I)$, $p\in\mathbb{N}_0$, the classical Lebesgue and Hilbert spaces, and endow them with the usual norms, see \cite{AdamsFournier2003}. Finally, we indicate the Lebesgue and Hilbert spaces for vector-valued functions as $\bm L^p(I) = [L^p(I)]^d$ and $\bm H^p(I) = [H^p(I)]^d$, $d\ge1$.

%%%%%%%%%%%%%%%%%%%%%%%%%%%%%%%%%%%%%%%%
\section{A brief review of some integrators}\label{sec:review}
%%%%%%%%%%%%%%%%%%%%%%%%%%%%%%%%%%%%%%%%

We now briefly review three different classes of numerical time-stepping methods: the Newmark methods in section \ref{sec:review:Newmark}, the implicit Runge-Kutta methods in section \ref{sec:review:RK}, and the general linear methods in section \ref{sec:review:GLM}. For this purpose, we consider the second-order model problem
\begin{equation}\label{eq:review:second}
	\begin{cases}
		\ddot{u}(t) = f(u(t)) \qquad \forall\, t \in (0,T],\\
		u(0) = \widehat{u}_0, \\
		\dot{u}(0) = \widehat{u}_1,
	\end{cases}
\end{equation}
which can be written as a first-order system,
\begin{equation}\label{eq:review:first}
	\begin{cases}
		\dot{v}(t) = f(u(t))\qquad \forall\, t \in (0,T],\\
        \dot{u}(t) = v(t)\qquad \forall\, t \in (0,T],\\
		u(0) = \widehat{u}_0, \\
		v(0) = \widehat{u}_1.
	\end{cases}
\end{equation}
We consider a discrete time grid of $N+1$ points $t_n=n \Delta t$, $n=0,\dots,N$, $\Delta t = \frac{T}{N}$, and denote by $u_n$ and $v_n$ the discrete approximations to $u(t_n)$ and $v(t_n)$.
Finally, to keep the discussion simple, we assume that $f$ is sufficiently regular.

\subsection{The Newmark method}\label{sec:review:Newmark}
The derivation of the Newmark scheme is based on Taylor's theorem,
\begin{equation}\label{eq:Taylor}
\begin{split}
u(t_n+\Delta t) &= u(t_n) + \Delta t \ \dot{u}(t_n) + \int_{t_n}^{t_{n+1}} \ddot{u}(\tau) (t_{n+1}-\tau) \ d\tau, \\ 
\dot{u}(t_n+\Delta t) &= \dot{u}(t_n) + \int_{t_n}^{t_{n+1}} \ddot{u}(\tau) \ d\tau.
\end{split}
\end{equation}
The Newmark scheme is obtained by a discretization of the integrals (remainder terms) in \eqref{eq:Taylor}. This is achieved by the approximation
$\ddot{u}(\tau) \approx (1-\gamma) \ddot{u}(t_n) + \gamma \ddot{u}(t_{n+1})$,
which inserted in the integrals above leads to
\begin{equation}\label{eq:Taylor_remind}
\begin{split}
&\int_{t_n}^{t_{n+1}} \ddot{u}(\tau) (t_{n+1}-\tau) \ d\tau 
\approx 
\Delta t^2 \ \Bigl[ \Bigl(\frac{1}{2}-\frac{\gamma}{2} \Bigr) \ddot{u}(t_n) + \frac{\gamma}{2} \ddot{u}(t_{n+1}) \Bigr], \\ 
&\int_{t_n}^{t_{n+1}} \ddot{u}(\tau) \ d\tau 
\approx 
\Delta t \ \bigl[(1-\gamma) \ddot{u}(t_n) + \gamma \ddot{u}(t_{n+1})\bigr] . \\ 
\end{split}
\end{equation}
Even though a natural choice for $\gamma$ would be $\gamma=\frac{1}{2}$ for both integrals (corresponding to the trapezoidal rule) in the Newmark framework, the term $\frac{\gamma}{2}$ is usually replaced by a second parameter $\beta$ that can be chosen independently of $\gamma$. 
Now, recalling that $\ddot{u}(t) = f(u(t))$, we define $f_n := f(u_n)$.
Thus, by inserting the approximation \eqref{eq:Taylor_remind} into \eqref{eq:Taylor}, we get
\begin{equation}\label{eq:Newmark}
\begin{split}
u_{n+1} &= u_{n} + \Delta t \ v_n + \Delta t^2 \ \Bigl[ \Bigl(\frac{1}{2}-\beta \Bigr) f_n + \beta f_{n+1} \Bigr], \\ 
v_{n+1} &= v_n + \Delta t \ [(1-\gamma) f_n + \gamma f_{n+1}].
\end{split}
\end{equation}
Equation \eqref{eq:Newmark} represents the Newmark family of numerical integrators. As we are going to see, the choice of $\gamma$ and $\beta$ affects the properties of the numerical integrator. Thus, each pair $(\gamma,\beta)$ gives rise to a different method.

%\marginpar{\MG{What do you mean by this statement? This is not consistent with Hairer/Wanner, see e.g. Definition 10.2 page 468: consistent of order p means if convergent then the convergence order is p}}
The Newmark scheme is first-order accurate (an analysis is provided in Section \ref{sec:P1}) if $\gamma\neq \frac{1}{2}$. The only choice that makes it second-order accurate is $\gamma=\frac{1}{2}$, corresponding to a trapezoidal rule for the approximation of the integral in the second equation in \eqref{eq:Taylor}.
Notice that the choice of $\beta$ does not affect the order of consistency, but the stability features of the method. The choice $\beta=0$ leads to the leapfrog scheme, which is an explicit scheme and hence conditionally stable \cite{raviart1983introduction}. If $\beta>0$ is chosen, then the Newmark scheme \eqref{eq:Newmark} becomes implicit, and the particular choice $\beta=\frac{1}{4}$ makes the method conservative:
if $f(u) = \lambda$, with $\lambda<0$, then it is possible to show that $|\dot{u}(t)|^2+\lambda|u(t)|^2$ is constant for all $t\geq 0$, and the choice $\beta=\frac{1}{4}$ guarantees that $|v_n|^2+\lambda|u_n|^2$ remains also constant for all $n$.
Notice that the choice $\gamma=\frac{1}{2}$ and $\beta=\frac{1}{4}$ corresponds to the trapezoidal rule for the discretization of both integrals in \eqref{eq:Taylor_remind}. 
It is the only choice making the method conservative, and one can show that it leads exactly to the Crank-Nicolson scheme for \eqref{eq:review:first}. 
A general treatment for different values of $\gamma$ and $\beta$ can be found in \cite{raviart1983introduction}.

\subsection{Implicit Runge-Kutta methods}\label{sec:review:RK}
An $s-stage$ Runge-Kutta scheme applied to system \eqref{eq:review:first} is characterized by the Butcher tableau 
\begin{equation}\label{eq:Butcher_tableau}
\renewcommand\arraystretch{1.2}
\begin{array}{c|c}
\bm c & A\\
\hline
& \bm b
\end{array}
\end{equation}
with matrix $A = \{a_{ij} \} \in \mathbb{R}^{s\times s}$, weight vector $\bm b = (b_1, \dots , b_s)$, and
quadrature nodes $\bm c = (c_1, \dots , c_s)^\top$.
Runge-Kutta methods update the solution using a sum over stage vectors as follows:
\begin{equation}\label{eq:RK_extended_0}
    \begin{cases}
        \bm z_{n+1} = \bm z_n + \Delta t \sum_{i=1}^s b_i \bm k_{i,n},\\
        \bm k_{i,n}  = f \left( \bm z_n + \Delta t \sum_{j=1}^s a_{ij} \bm k_{j,n}\right).
    \end{cases}
\end{equation}
If the coefficients $\{a_{ij}\}$ in $A$ are nonzero for $j \geq i$, for $i = 1, 2, \dots , s$, then 
every $\bm k_{i,n}$ can be obtained explicitly as a function of the $i-1$ stages $\bm k_{1,n},...,\bm k_{i-1,n}$ computed previously.
In this case, the scheme is called \textit{explicit}. Otherwise, the Runge-Kutta scheme is called \textit{implcit} and a non-linear system has to be solved  to compute 
$\bm k_{i,n}$. 
When $f$ is a linear operator represented by a matrix $L$, the stages $\bm k_{i,n}$, $i=1,\dots,s$, can be expressed as the solution of the linear system 
\begin{equation}\label{eq:RK_extended}
    \left(\begin{bmatrix}
        I & & \bm 0\\
          & \ddots & \\
          \bm 0 &  & I
    \end{bmatrix}
  -\Delta t \begin{bmatrix}
        a_{11} L& \ldots & a_{1s} L\\
         \vdots & \ddots & \vdots \\
          a_{s1}L  & \ldots & a_{ss}L
    \end{bmatrix}
   \right)
   \begin{bmatrix}
        \bm k_{1,n} \\
         \vdots \\
         \bm k_{s,n}
    \end{bmatrix}
    = 
    \begin{bmatrix}
        \bm f_1  \\
         \vdots \\
         \bm f_s
    \end{bmatrix},     
   \end{equation}
with $\bm f_i = L\bm z_n$,
%$\bm f_i = L(t_n +\Delta t c_i )\bm z_n$, 
for $i=1,..,s$.
%Since $L$ is the same at all stages (independent
%of time), 
System \eqref{eq:RK_extended}  
can be expressed in the equivalent compact Kronecker product form
\begin{equation}\label{eq:rk_kroneker}
    (I_s \otimes I_2 - \Delta t A \otimes L) \bm k_n = \bm f.
\end{equation}
%which is closely related to \eqref{eq:dg_kroneker}.
%
Lobatto methods are fully implicit RK methods for which the coefficients $b_j$ and $c_j$ in \eqref{eq:Butcher_tableau} are determined based on the Lobatto quadrature formula 
%Indeed, by integrating \eqref{eq:modelP_system_matrix} between $[t_n, t_{n+1}]$ and using on the right-hand side the Lobatto quadrature rule, we can obtain the following scheme
%\begin{equation}
%    \bm z_{n+1} - \bm z_n = \Delta t \sum_{i=1}^s b_i L \bm z(t_n + c_i \Delta t),
%\end{equation}
 with $s$ node coefficients $c_1,\dots,c_s$, and $s$ weight coefficients $b_1, \dots, b_s$.
 
 The $s$ nodes $c_j$ are the roots of the polynomial of degree $s$
 \begin{equation*}
     \dfrac{d^{s-2}}{dt^{s-2}} (t^{s-1}(1-t)^{s-1}),
 \end{equation*}
and satisfy $c_1 < c_2 < \dots < c_s$.
The weights $b_j$ and nodes $c_j$ satisfy the order conditions $B(2s-2)$, which are for generic argument $p$ given by
\begin{align*}
    B(p) & : \quad\sum_{j=1}^s b_j c_j^{k-1} = \dfrac{1}{k}, \quad k=1,\dots,p,
\end{align*}
implying that the quadrature formula is of order $2s -2$.
%\IM{Questo lo vedi facilmente perchè $\frac1k= \int_0^1 x^{k-1}$, mentre a sinistra hai la formula di quadratura per $x^{k-1}$}. 
The families of Lobatto RK methods differ only in the values of their coefficients $a_{ij}$. Here, we recall the following order conditions that can be used to determine their expressions:
\begin{align}
    C(q) & :\quad \sum_{j=1}^s a_{ij}c_j^{k-1} = \dfrac{c_i^k}{k}, \quad i=1,\dots,s \; {\rm and} \; k= 1,\dots,q,\\
    D(r) & :\quad \sum_{i=1}^s b_i c_i^{k-1} a_{ij} = \dfrac{b_j}{k}(1-c_i^k), \quad j=1,\dots,s \; {\rm and} \; k= 1,\dots,r.
\end{align}
The importance of these order conditions comes from a fundamental result due to Butcher, cf. also \cite[Therem 5.1]{wanner1996solving}: 
\textit{if the coefficients $b_i$, $c_i$, $a_{ij}$ of a Runge-Kutta method satisfy $B(p)$, $C(q)$, and $D(r)$, with $p\leq q+r+1$ and $p\leq 2q + 2$, then the method is of order $p$}.

Four main families belong to the set of Lobatto RK methods, namely Lobatto IIIA, Lobatto IIIB, Lobatto IIIC, and Lobatto IIIC$^*$ methods, cf. \cite{wanner1996solving}. We briefly recall some results for Lobatto IIIC: their coefficients $a_{ij}$ are defined by $a_{i1} = b_1$ for $i=1,\dots,s$ and the order conditions $C(s-1)$. They satisfy the order conditions $D(s-1)$ and $a_{sj} = b_j$, for $j=1,\dots,s$. Lobatto IIIC methods are of order $2s-2$. The order of convergence is evaluated for the last stage (not intermediate stages for which the order is $s$; see \cite{wanner1996solving}).  They are not symmetric. Their stability function $R(z)$ is given by the $(s-2, s)$-Padé approximation to
$e^z$, e.g.,
\begin{align*}
    R(z) & = \dfrac{1}{1-z+\frac{z^2}{2}}, \quad { s=2}, \\
    R(z) & = \dfrac{1+\frac{1}{4}z}{1-\frac{1}{4}z+ \frac{1}{2}\frac{z^2}{2\!} - \frac{1}{4}\frac{z^3}{3\!} }, \quad {s=3}.    
\end{align*}
They are L-stable, i.e., they are A-stable and $R(z) \rightarrow 0$ as $z \rightarrow \infty$. They are algebraically stable, i.e.,
\begin{equation*}
    B := {\rm diag}(b_1,\dots,b_s), \; {\rm and } \;  M := BA+A^\top B- \bm b \bm b^\top
\end{equation*}
are positive semi-definite, \cite{Jay2015}.
The above condition implies that the Lobatto IIIC methods are B-stable. Thus, they are excellent methods for stiff problems.

\subsection{General linear methods}\label{sec:review:GLM}
General linear methods (GLMs) were introduced by Butcher in \cite{Butcher1996} and further developed by several authors; see, e.g., \cite{Butcher2006} for a general survey. 
For our short review, we start from the point of view given by \cite{DAmbrosio2014}, which is a direct application of Butcher's original idea to second-order problems.

The essential idea of GLMs is to consider a vector ${\bf y}^{[n]} \in \mathbb{R}^r$, $r \in \mathbb{N}^+$, which does not only contain the approximation $u_n$ to $u(t_n)$, but also additional information in components needed to build the integration step, and a stage-like vector (in the spirit of RK methods), denoted by $Y\in \mathbb{R}^s$, $s \in \mathbb{N}^+$, accounting for implicit parts of the integration step. 
Thus, a single step of a GLM reads as
\begin{equation}\label{eq:GLM}
    \begin{bmatrix}
    Y \\
    {\bf y}^{[n+1]} \\
    \end{bmatrix}
    =
    \begin{bmatrix}
        A & U \\
        B & V \\
    \end{bmatrix}
    \begin{bmatrix}
    \Delta t^2 f(Y) \\
    {\bf y}^{[n]} \\
    \end{bmatrix},
\end{equation}
where the entries of the matrices $A \in \mathbb{R}^{s \times s}$, $U \in \mathbb{R}^{s \times r}$, $B \in \mathbb{R}^{r \times s}$, and $V \in \mathbb{R}^{r \times r}$ are used to define the GLM.
The system \eqref{eq:GLM} can be written explicitly as
\begin{align}
    Y_i = \Delta t^2 \sum_{j=1}^s a_{i,j}f(Y_j) + \sum_{j=1}^r u_{i,j} y_j^{[n]}, &&\text{for $i=1,\dots,s$,} \label{eq:GLM2_A} \\
    y_i^{[n+1]} = \Delta t^2 \sum_{j=1}^s b_{i,j}f(Y_j) + \sum_{j=1}^r v_{i,j} y_j^{[n]}, &&\text{for $i=1,\dots,r$.} \label{eq:GLM2_B}
\end{align}
Notice that \eqref{eq:GLM} can be implicit in the stage vector $Y$, while it is explicit in~${\bf y}$.
Let us give a simple example to better clarify the idea behind a GLM.

\begin{example}[Newmark as a GLM]\label{example:GLM:1}
    Consider the Newmark scheme \eqref{eq:Newmark} for $\gamma=\frac{1}{2}$ and $\beta=\frac{1}{4}$:
    \begin{equation}\label{eq:Newmark2}
    \begin{split}
    u_{n+1} &= u_{n} + \Delta t \ v_n + \frac{\Delta t^2}{4} \ \Bigl[  f(u_{n+1}) + f(u_{n}) \Bigr], \\ 
    v_{n+1} &= v_n + \frac{\Delta t}{2} \ \Bigl[  f(u_{n+1}) + f(u_{n}) \Bigr].
    \end{split}
    \end{equation}
    To write this scheme in the form \eqref{eq:GLM}, we introduce the variables
    $a_{n} := \Delta t \ v_{n}$, $b_{n} := \Delta t^2 \ f(u_{n})$, and
    rewrite the first equation in \eqref{eq:Newmark2} as 
    \begin{equation}\label{eq:A}
    Y = u_{n} + a_n + \frac{\Delta t^2}{4} f(Y) + \frac{1}{4} b_{n},
    \end{equation}
    and the second equation in \eqref{eq:Newmark2} as
    \begin{equation}\label{eq:B}
    a_{n+1} = a_n + \frac{1}{2}b_n + \frac{\Delta t^2}{2} f(Y).
    \end{equation}
    Notice that $Y=u_{n+1}$. This also implies that 
    \begin{equation}\label{eq:C}
    b_{n+1} = \Delta t^2 \ f(u_{n+1}) = \Delta t^2 \ f(Y).
    \end{equation}
    Collecting \eqref{eq:A}, \eqref{eq:B}, \eqref{eq:C} and the first equation of \eqref{eq:Newmark2}, we obtain
    \begin{equation*}
        \begin{bmatrix}
            Y \\
            a_{n+1} \\
            b_{n+1} \\
            u_{n+1}
        \end{bmatrix}
        =
        \begin{bmatrix}
            1/4 & 1 & 1/4 & 1 \\
            1/2 & 1 & 1/2 & 0 \\
            1 & 0 & 0 & 0 \\
            1/4 & 1 & 1/4 & 1 \\
        \end{bmatrix}
        \begin{bmatrix}
            \Delta t^2 f(Y) \\
            a_{n} \\
            b_{n} \\
            u_{n}
        \end{bmatrix}.
    \end{equation*}
    This is exactly in the form \eqref{eq:GLM} with ${\bf y}^{[n]} = [a_{n},b_{n},u_{n}]^\top$.
    Notice that this system is only implicit in the variable $Y$, and $a_{n}$ and $b_{n}$ contain auxiliary information, and form the vector ${\bf y}^{[n]}$ together with the approximation $u_n$.
\end{example}
\textit{A GLM of the form \eqref{eq:GLM} is convergent if and only if it is consistent and zero stable, (see, e.g., \cite{Butcher2006,DAmbrosio2012}).}
Thus, we need to clarify the notions of convergence, consistency and zero stability for GLMs.

Neglecting the issue of a starting procedure (see \cite{Butcher2006,DAmbrosio2012}), we say that a GLM converges if ${\bf y}^{[n]} \rightarrow \widehat{{\bf y}}^{[n]} = {\bf q}_0 u(t_n)$, for all $n$ and some vector ${\bf q}_0$, as $\Delta t \rightarrow 0$.

The zero stability is the property of the method to well represent the solution of $\ddot{u}(t) =0$. For the case $f=0$, the GLM \eqref{eq:GLM} becomes ${\bf y}^{[n+1]} = V {\bf y}^{[n]}$.
Therefore, the zero stability is related to the spectral properties of the matrix $V$ 
\cite[Theorem~4.1]{DAmbrosio2012}: \textit{a GLM of the form \eqref{eq:GLM} is zero stable if the eigenvalues of $V$ lie within or on the unit circle %\marginpar{\MG{Need to explain why no linear growth with multiplicity 2}} 
and the multiplicity of those lying on the unit circle is at most two.}\footnote{Note that this definition of zero stability is the one given in \cite{DAmbrosio2012} (see also \cite[Definition 10.1]{wanner1996solving}) for second-order equations.
The meaning is that, when applied to \eqref{eq:review:second} with $f=0$, and $\hat{u}_1=0$, a zero-stable method preserves the second-order derivative.
A different definition is provided in \cite[Definition 3.2]{wanner1996solving} for first-order systems.}

The consistency of GLMs is usually accomplished by the so-called \textit{order conditions} (see, e.g., \cite{Butcher2006,DAmbrosio2012,DAmbrosio2012,Li2016}).
To obtain them, consider the solution $u(t)$ to \eqref{eq:review:second} and an input vector ${\bf y}^{[n]}$, whose components are written in a generalized Taylor form:
\begin{equation}\label{eq:GLM:input}
    y_i^{[n]} := \sum_{k=0}^p q_{i,k} \Delta t^k u^{(k)}(t_n) + O(\Delta t^{p+1}),
    \quad \text{for $i=1,\dots,r$},
\end{equation}
where $q_{i,k}$ are given coefficients, and
\begin{equation}\label{eq:GLM:output2}
    Y_i = u(t_n+c_i \Delta t) + O(\Delta t^{p+1}) \quad \text{for $i=1,\dots,s$},
\end{equation}
where $c_i$ are stage coefficients. 
Notice that this definition is consistent with the notion of convergence we stated above.
Now, the driving question is: under which conditions on the matrices $A$, $B$, $U$, and $V$ and on the coefficients $q_{i,k}$ and $c_i$ does one step of a GLM \eqref{eq:GLM} produce exactly the output 
\begin{equation}\label{eq:GLM:output}
    y_i^{[n+1]} := \sum_{k=0}^p q_{i,k} \Delta t^k u^{(k)}(t_n+\Delta t) + O(\Delta t^{p+1}),
    \quad \text{for $i=1,\dots,r$},
\end{equation}
if the input \eqref{eq:GLM:input} is considered? 
More precisely: which conditions on $A$, $B$, $U$, $V$, $q_{i,k}$, and $c_i$ guarantee that the output is exact in all the terms till order $p$?

The answer is provided by the \textit{order conditions} that are algebraic conditions on $A$, $B$, $U$, $V$, and $q_{i,k}$. 
Notice that we consider here the same order $p$ in \eqref{eq:GLM:output} and \eqref{eq:GLM:output2}.
This choice is made only for the sake of simplicity, and more general results can be found in the literature (see, e.g., \cite{DAmbrosio2012}).
The derivation proceeds by Taylor expansions, and we obtain the order conditions as necessary conditions for \eqref{eq:GLM:input}, \eqref{eq:GLM:output}, and \eqref{eq:GLM:output2}.
Recalling that $\ddot{u}(t) = f(u(t))$ and using \eqref{eq:GLM:output2} we get
\begin{equation}\label{eq:GLM:exp1}
    \Delta t^2 f(Y_i) 
    = \Delta t^2 \ddot{u}(t_n+c_i \Delta t) + O(\Delta t^{p+3}) 
    = \sum_{k=2}^p \Delta t^k u^{(k)}(t_n) \frac{c_i^{k-2}}{(k-2)!} + O(\Delta t^{p+1}).
\end{equation}
Now, using \eqref{eq:GLM2_A}, \eqref{eq:GLM:input}, \eqref{eq:GLM:output2}, and \eqref{eq:GLM:exp1}, we get
\begin{equation*}%\label{eq:GLM:output3}
\begin{split}
   u(t_n+c_i \Delta t)
   &\stackrel{\eqref{eq:GLM:output2}}{=} Y_i + O(\Delta t^{p+1})
   \stackrel{\eqref{eq:GLM2_A}}{=} \Delta t^2 \sum_{j=1}^s a_{i,j}f(Y_j) + \sum_{j=1}^r u_{i,j} y_j^{[n]} + O(\Delta t^{p+1}) \\
   &\stackrel{\eqref{eq:GLM:exp1}}{=}
   \sum_{j=1}^s a_{i,j} \sum_{k=2}^p \Delta t^k u^{(k)}(t_n) \frac{c_i^{k-2}}{(k-2)!}
   + \sum_{j=1}^r u_{i,j} y_j^{[n]} + O(\Delta t^{p+1})\\
   &\stackrel{\eqref{eq:GLM:input}}{=}
   \sum_{j=1}^s a_{i,j} \sum_{k=2}^p \Delta t^k u^{(k)}(t_n) \frac{c_i^{k-2}}{(k-2)!}
   + \sum_{j=1}^r u_{i,j} \sum_{k=0}^p q_{j,k} \Delta t^k u^{(k)}(t_n) \\ & \qquad \quad + O(\Delta t^{p+1}).\\
\end{split}
\end{equation*}
Rearranging the sums in this expression, we get
\begin{equation}\label{eq:GLM:output3}
\begin{split}
   u(t_n+c_i \Delta t)
   &= \sum_{j=1}^r u_{i,j} q_{j,0} u(t_n)
   + \sum_{j=1}^r u_{i,j} q_{j,1} \Delta t \ u^{(1)}
   (t_n) \\
   &+ \sum_{k=2}^p \Delta t^k \Biggl[ 
   \sum_{j=1}^r u_{i,j} q_{j,k} u^{(k)}(t_k)
   + \sum_{j=1}^s a_{i,j} \frac{c_j^{(k-2)}}{(k-2)!}
   \Biggr] + O(\Delta t^{p+1}).
\end{split}
\end{equation}
Now, we expand $u(t_n+c_i \Delta t)$ and write
\begin{equation}\label{eq:GLM:output4}
    u(t_n+c_i \Delta t) 
    = u(t_n) + c_i \Delta t \ u^{(1)}(t_n)
    + \sum_{k=2}^p \Delta t^k \frac{c_i^k}{k!} u^{(k)}(t_n) + O(\Delta t^{p+1}).
\end{equation}
Now, by a direct comparison of \eqref{eq:GLM:output3} and \eqref{eq:GLM:output4}, matching term by term for each $i$, we obtain the conditions
\begin{subequations}
\begin{align}
    U {\bf q}_0 &= {\bf 1} &&k=0 \text{ (preconsistency)}, \label{eq:OrdC:1}\\
    U {\bf q}_1 &= {\bf c} &&k=1 \text{ (preconsistency)},\label{eq:OrdC:2}\\
    U {\bf q}_2 + A{\bf 1} &= \frac{1}{2}{\bf c}^2 &&k=2 \text{ (stage consistency)},\label{eq:OrdC:3}\\
    U {\bf q}_k+ \frac{1}{(k-2)!}A{\bf c}^{k-2} &= \frac{1}{k!}{\bf c}^k &&k>2,\label{eq:OrdC:4}
\end{align}
\end{subequations}
where ${\bf 1} \in \mathbb{R}^s$ denotes a vector whose components are all equal to $1$, and ${\bf c}^k = [c_1^k,\dots,c_s^k]^\top$.
Now, using \eqref{eq:GLM2_B}, \eqref{eq:GLM:input}, \eqref{eq:GLM:output2}, \eqref{eq:GLM:output}, and \eqref{eq:GLM:exp1}, and proceeding as before, we get
\begin{subequations}
\begin{align}
    V {\bf q}_0 &= {\bf q}_0 &&k=0 \text{ (preconsistency)},\label{eq:OrdC:5}\\
    V {\bf q}_1 &= {\bf q}_0+{\bf q}_1 &&k=1 \text{ (preconsistency)},\label{eq:OrdC:6}\\
    V {\bf q}_2 + B{\bf 1} &= \frac{1}{2}{\bf q}_0+{\bf q}_1+{\bf q}_2 &&k=2,\label{eq:OrdC:7}\\
    V {\bf q}_k+ \frac{1}{(k-2)!}B{\bf c}^{k-2} &= \sum_{\ell=0}^k \frac{1}{\ell!}{\bf q}_{k-\ell} &&k>2.\label{eq:OrdC:8}
\end{align}
\end{subequations}
Equations \eqref{eq:OrdC:1}-\eqref{eq:OrdC:8} are the order conditions. Even though we derived them as necessary conditions for \eqref{eq:GLM:input}, \eqref{eq:GLM:output}, and \eqref{eq:GLM:output2}, they can be shown to be also sufficient conditions; see \cite{Butcher2006,DAmbrosio2012}.
The equations \eqref{eq:OrdC:1}, \eqref{eq:OrdC:2}, \eqref{eq:OrdC:5}, and \eqref{eq:OrdC:6} are also known as \textit{preconsistency conditions}. 
A GLM is said to be \textit{consistent} if it is preconsistent and \eqref{eq:OrdC:7} holds. Moreover, a GLM is \textit{stage consistent} if it satisfies equation \eqref{eq:OrdC:3}.
The conditions \eqref{eq:OrdC:5} and \eqref{eq:OrdC:6} can be interpreted as the property of the GLM to properly represent solutions to $\ddot{y}(t)=0$. If this equation is complemented by initial conditions $y(0) \neq 0$ and $\dot{y}(0)=0$, then $y(t)$ is constant. 
This is clearly represented by the condition ${\bf q}_0 = V{\bf q}_0$.
Moreover, if $\ddot{y}(t)=0$ is complemented by initial conditions $y(0) \neq 0$ and $\dot{y}(0)\neq 0$, then $y(t)$ is linear-affine.
This structure is represented by the condition ${\bf q}_0 + {\bf q}_1 = V{\bf q}_1$.

Now a GLM is convergent if and only if it is zero stable and it satisfies \eqref{eq:OrdC:1}, \eqref{eq:OrdC:2}, \eqref{eq:OrdC:5}, \eqref{eq:OrdC:6}, and \eqref{eq:OrdC:7}. Clearly, in this case the order of convergence is $1$.
If in addition, \eqref{eq:OrdC:3}, \eqref{eq:OrdC:4} and \eqref{eq:OrdC:8} are satisfied for $k=2,\dots,p$, then the GLM is convergent of order $p$.

\begin{example}[Order conditions of Newmark as a GLM]\label{example:GLM:2}
Recalling the Newmark scheme written as a GLM in Example \ref{example:GLM:1}, we have that $V = \begin{bmatrix}
            1 & 1/2 & 0 \\
            0 & 0 & 0 \\
            1 & 1/4 & 1 \\
\end{bmatrix}$, which has eigenvalues equal to 0 and 1.
Thus, the Newmark scheme as a GLM is zero stable.
A direct calculation shows that the vectors
\begin{equation*}
    {\bf q}_0 = \begin{bmatrix}
        0 \\ 0 \\ 1
    \end{bmatrix}, \quad
    {\bf q}_1 = \begin{bmatrix}
        1 \\ 0 \\ 0
    \end{bmatrix}, \quad
    {\bf q}_2 = \begin{bmatrix}
        0 \\ 1  \\ 0
    \end{bmatrix}, \quad
    {\bf q}_3 = \begin{bmatrix}
        -\frac{1}{12} \\ 0  \\ 0
    \end{bmatrix}
\end{equation*}
satisfy the order conditions for $k=0,1,2,3$. However, it is also possible to show that there does not exist a vector ${\bf q}_4$ such that the order conditions hold for $k=4$. Thus, $p=3$ and we can conclude that the Newmark scheme as a GLM is second-order convergent as expected.
%\marginpar{\MG{Again Hairer and Wanner do not use different order for consistency and convergence, see e.g. Definition 10.2 page 468}}
%fourth-order consistent, and second-order convergent as expected.
\end{example}

%%%%%%%%%%%%%%%%%%%%%%%%%%%%%%%%%%%%%%%%
\section{Time DG methods for second-order equations: DG2}\label{sec:dG:second}
%%%%%%%%%%%%%%%%%%%%%%%%%%%%%%%%%%%%%%%%
Let us consider $T>0$ and the model problem %: find $u(t) \in H^2(0,T]$ such that 
\begin{equation}
	\label{eq:modelProblem}
	\begin{cases}
		\ddot{u}(t) + \lambda u(t) = 0 \qquad \forall\, t \in (0,T],\\
		u(0) = \widehat{u}_0, \\
		\dot{u}(0) = \widehat{u}_1,
	\end{cases}
\end{equation}
where $\lambda \in \mathbb{R}^+$ is a positive coefficient and $\widehat{u}_0, \widehat{u}_1 \in \mathbb{R}$. 
%Problem \eqref{eq:modelProblem} is well posed and admits a unique solution $\textbf{u} \in H^2(0,T]$ in the interval $(0,T]$, see \cite{kroopnick}. 
We consider a partition for the interval $ I = (0, T] $ into $N $ sub-intervals (time slabs) $ I_n = (t_{n-1}, t_n] $ such that $\Delta t_n = t_n - t_{n-1}$, for $ n = 1,..,N $,  with $t_0 = 0$ and $t_N=T$, as shown in Figure~\ref{fig:time_domain}. We suppose in what follows that all time slabs have the same size, that is $\Delta t_n = \Delta t$ for all $n$.
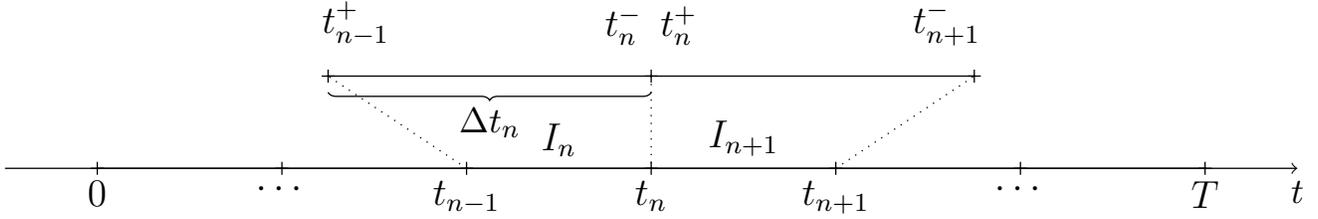
\begin{figure}[t]
		\centering
		{\resizebox{1\textwidth}{!}{\definecolor{mycolor1}{rgb}{0.00000,0.49804,0.00000}%
\begin{tikzpicture}

\draw [->]  (-1,0) -- (13,0) node [below] {$t$};
\draw plot [mark=+, smooth, line width=7pt] coordinates
{(0,0) (2,0) (4,0) (6,0) (8,0) (10,0) (12,0)  };

\node [font=\large] at (0,0) [below] {0};
\node [font=\large] at (2,0) [below] {$\cdots$};
\node [font=\large] at (4,0) [below] {$t_{n-1}$};
\node [font=\large] at (5,0) [above] {$I_{n}$};
\node [font=\large] at (6,0) [below] {$t_{n}$};
\node [font=\large] at (7,0) [above] {$I_{n+1}$};
\node [font=\large] at (8,0) [below] {$t_{n+1}$};
\node [font=\large] at (10,0) [below] {$\cdots$};
\node [font=\large] at (12,0) [below] {$T$};

\draw plot [mark=+, smooth, line width=9pt,mark options={solid,black}] coordinates
{(2.5,1) (6,1) (9.5,1)  };

%\draw [dashed] (2,1) -- (10,1);
\draw [dotted] (6,1) -- (6,0);
\draw [dotted] (4,0) -- (2.5,1);
\draw [dotted] (8,0) -- (9.5,1);

\node [font=\large] at (6.3,1.2) [above] {$t_{n}^{+}$};
\node [font=\large] at (5.7,1.2) [above] {$t_n^-$};

\node [font=\large] at (2.8,1.2) [above] {$t_{n-1}^+$};
\node [font=\large] at (9.2,1.2) [above] {$t_{n+1}^-$};

\draw[decoration={brace,mirror,raise=5pt},decorate]
  (2.5,1) -- node[below=6pt] {$\Delta t_n$} (6,1);

\end{tikzpicture}}}
		\caption{Example of a time domain partition and a zoom  with the values $t_n^+$ and $t_n^-$.}\label{fig:time_domain}
	\end{figure}
To build an approximation to the solution $u$ on each $I_n$, we multiply \eqref{eq:modelProblem} by $\dot{v}(t)$ and integrate over $I_n$,
\begin{equation}\label{eq:intPart}
(\ddot{u},\dot{v})_{I_n} + \lambda ( u, \dot{v} )_{I_n} = 0.
\end{equation}
Next, assuming sufficient regularity of $u$ (namely $u \in H^2(0,T) \Subset C^1([0,T])$), and observing that $ [u]_n = [\dot{u}]_n = 0$, for  $n=1, \dots, N$, we can rewrite \eqref{eq:intPart} by adding suitable consistent terms for any $s \geq 0$ as follows:
\begin{equation}\label{eq:intStrong}
(\ddot{u},\dot{v})_{I_n} + \lambda ( u, \dot{v} )_{I_n} +  [\dot{u}]_{n-1} \dot{v}_{n-1}^+ +  \lambda [u]_{n-1} v_{n-1}^+ - s \lambda\{ \dot{u} \}_{n-1} \dot{v}_{n-1}^+ = 0.
\end{equation}
 For $s=0$, this formulation is equivalent to the one shown in \cite{DalSanto18}. The term $-s \lambda\{ \dot{u} \}_{n-1} \dot{v}_{n-1}^+$ is introduced as a correction term to improve the accuracy of the scheme. Moreover, to guarantee that this is a consistent correction, the parameter $s$ will be chosen proportional to $\Delta t^p$, with $p$ to be determined.
Next, we introduce the local finite-dimensional space
$\mathbb{V}_n^r = \{v: I_n \rightarrow \mathbb{R}: v \in \mathbb{P}^r(I_n) \},
$
where $ \mathbb{P}^r(I_n) $ is the space of polynomials of degree $ r \geq 1 $ on $ I_n $ and define the DG space as
\begin{align*}
\mathbb{V}_{DG} = \{v \in L^2(0,T): v|_{I_n} \in \mathbb{V}_n^r \hspace{3mm} \forall \, n = 1, \dots, N \},
\end{align*}
having finite dimension $(r+1)N$.  
%\textcolor{red}{WHY $r_n$?? (SUBSCRIPT $n$)} 
Summing over all time slabs in \eqref{eq:intStrong}, we obtain the problem:
find $u_{DG} \in \mathbb{V}_{DG}$ such that 
\begin{align}\label{eq:discreteFormulation}
\mathcal{A}(u_{DG}, v) = F(v) \qquad \forall \, v\in \mathbb{V}_{DG}, 
\end{align}
where $\mathcal{A} : \mathbb{V}_{DG} \times  \mathbb{V}_{DG} \rightarrow \mathbb{R}$ is defined by
\begin{align}
\mathcal{A}(u, v) & := \sum_{n=1}^N \Big( (\ddot{u},\dot{v})_{I_n} + \lambda (u, \dot{v} )_{I_n}\Big) +   \sum_{n=1}^{N-1} \left( \dot{[u]}_{n} \dot{v}_{n}^+  +  \lambda [u]_{n}  v_{n}^+  
- s \lambda \{ \dot{u} \}_{n} \dot{v}_{n}^+ \right) \nonumber \\  & \quad   
 + \dot{u}_{0}^+ \dot{v}_{0}^+
+  \lambda u_{0}^+  v_{0}^+   - \frac{s}{2} \lambda \dot{u}_0^+ \dot{v}_0^+ ,  \label{eq:bilinearAn}
\end{align}
and the linear functional $F : \mathbb{V}_{DG} \rightarrow \mathbb{R} $ is defined as
\begin{align}\label{eq:linearF}
F (v) & := \widehat{u}_{1} \dot{v}_{0}^+ +  \lambda \widehat{u}_{0} v_{0}^+  + \frac{s}{2} \lambda \widehat{u}_{1} \dot{v}_{0}^+ .
\end{align}
For $ n = 1 $, we adopted the convention $ {u}_0^- = \widehat{u}_0 $ and $ \dot{u}_0^- = \widehat{u}_1 $. Moreover, we point out that the definition of the bilinear form $\mathcal{A}(\cdot, \cdot)$ makes sense 
whenever its arguments are, at least, $H^2(I_n)$ functions for any $n=1,\ldots, N$. The existence and uniqueness of the discrete solution as well as stability bounds in a suitable mesh-dependent norm can be found in \cite{DalSanto18} for the case $s=0$.

\subsection{Algebraic formulation for the DG2 discretization}\label{sec:algebraic}
We focus on the algebraic formulation of problem \eqref{eq:intStrong} by considering a generic time slab $I_{n}$, where a local polynomial degree $r$ is used. With the DG technique \eqref{eq:intStrong}, it is possible to compute the solution of the problem separately for one time slab at a time, by using the solution computed at the previous time slab. 
To do so, we introduce a  basis $\{\psi_{n,j}(t) \}_{j=1, \dots, r+1}$ for the polynomial space $\mathbb{P}^r(I_{n})$, $r+1$ being the dimension of the local finite-dimensional space $\mathbb{V}_{n}^r$. Next, we write the trial function $u_{DG}$ as a linear combination of the basis functions, i.e.,
\begin{align*}
u(t) = 
\sum \limits_{j=1}^{r+1} u_{n,j} \psi_{n,j}(t), \text{for $t \in I_{n}$} , 
\end{align*}
where $ u_{n,j} \in \mathbb{R}$, for $j = 1, \dots, r+1$. By writing equation \eqref{eq:intStrong} for any test function $\psi_{n,i}, i = 1, \dots, r+1$ and $n=0,\dots,N$, we obtain the system  
\begin{align}\label{eq:simple_local_system}
A \bm u = \bm b,
\end{align}
where
\begin{equation*}
    A = \begin{bmatrix}
        I \\
        A_- & A_+ \\
        & A_- & A_+ \\
        & & \ddots & \ddots \\
        & & & A_- & A_+ \\
    \end{bmatrix},
    \quad
    \bm u = \begin{bmatrix}
        \bm u_0 \\
        \bm u_1 \\
        \bm u_2 \\
        \vdots \\
        \bm u_N \\
    \end{bmatrix},
    \quad
    \bm b = \begin{bmatrix}
        \bm b_0 \\
        \bm 0 \\
        \bm 0 \\
        \vdots \\
        \bm 0 \\
    \end{bmatrix}.
\end{equation*}
Here, $\bm u \in \mathbb{R}^{(r+1)(N+1)}$ is the vector containing the slab coefficients $\bm u_n = (u_{n,1},\dots,u_{n,r+1})^\top\in \mathbb{R}^{(r+1)}$, $n=1,\dots,N$, $\bm u_0$ and $\bm b_0$ contain the initial conditions, and the local matrices $A_-$ and $A_+$ are given by
 \begin{align}
     A_{-,ij} & = -\langle \dot{\psi}_{n,j}, \dot{\psi}_{n+1,i} \rangle_{t_n}- \lambda \langle \psi_{n,j}, \psi_{n+1,i} \rangle_{t_n} - \frac{s}{2}\lambda \langle \dot{\psi}_{n,j}, \dot{\psi}_{n+1,i} \rangle_{t_n}, \label{eq:a_m}\\
     A_{+,ij} & = (\ddot{\psi}_{n+1,j},\dot{\psi}_{n+1,i})_{I_n} + \lambda ( \psi_{n+1,j}, \dot{\psi}_{n+1,i} )_{I_n} +\langle \dot{\psi}_{n+1,j}, \dot{\psi}_{n+1,i} \rangle_{t_n} \label{eq:a_p} \\ & \quad + \lambda \langle \psi_{n+1,j}, \psi_{n+1,i} \rangle_{t_n}  - \frac{s}{2}\lambda \langle \dot{\psi}_{n+1,j}, \dot{\psi}_{n+1,i} \rangle_{t_n} \nonumber
 \end{align} 
for $i,j=1,\dots,r+1$. 
Clearly, defining $G:=-A_+^{-1}A_-$, \eqref{eq:simple_local_system} can be written as
\begin{equation}\label{eq:simple_local_system2}
    A_+ \bm u_{n+1} + A_- \bm u_{n} = \bm 0
    \; \Longleftrightarrow \;
    \bm u_{n+1} = G \bm u_{n} \, \text{ for $n=0,1,\dots,N-1$}.
\end{equation}
To study the accuracy and consistency of a scheme of the form \eqref{eq:simple_local_system}-\eqref{eq:simple_local_system2}, we define error and truncation error as $\bm e_n := \bm u_n^{ex} - \bm u_n$ and $\bm \theta_n := \bm u_{n+1}^{ex} - G \bm u_{n}^{ex}$, for all $n$, where $\bm u_n^{ex}$ denotes
the vector containing the exact solution evaluated at the nodes of the slab $I_n$.
The definitions of $\bm e_n$ and $\bm \theta_n$ yield immediately
\begin{equation}\label{eq:error}
    \bm e_{n+1} = G \bm e_{n} + \bm \theta_n .
\end{equation}
We now have the following result.
\begin{lemma}[accuracy and consistency]\label{lemma:general}
Consider the scheme \eqref{eq:error} for $n=0,\dots,N-1$. 
Let $G=WDW^{-1}$ be the eigen-decomposition of $G$, with $D$ denoting the diagonal matrix whose entries are the eigenvalues of $G$. Assume that the spectral radius of $G$ satisfies $\rho(G)\leq 1$. 
If $\bm e_0 = \bm 0$, $W=O(\Delta t^q)$, and $W^{-1} \bm \theta_n = O(\Delta t^p)$, for some $q,p \in \mathbb{N}$, then $\bm e_n = O(\Delta t^{q+p-1})$.
\end{lemma}
\begin{proof}
Using \eqref{eq:error} recursively, we obtain
$\bm e_{n+1} = G^n \bm e_{0} + \sum_{\ell=0}^{n} G^{\ell} \bm \theta_{n-\ell}$.
Now, using that $\bm e_0 = \bm 0$ and the decomposition $G=WDW^{-1}$, we get
\begin{equation}\label{eq:relation}
    \bm e_{n+1} = W \sum_{\ell=0}^{n} D^{\ell} W^{-1}\bm \theta_{n-\ell}.
\end{equation}
Equation \eqref{eq:relation} allows us to estimate 
\begin{equation*}
    \begin{split}
        \| \bm e_{n+1} \| \leq \| W \| \sum_{\ell=0}^{n} \rho(G)^{\ell} \| W^{-1}\bm \theta_{n-\ell} \| \leq \| W \| N \max_{\ell=0,\dots,N-1} \| W^{-1}\bm \theta_{\ell} \|,
    \end{split}
\end{equation*}
where we used that $\rho(G) \leq 1$. The result follows by noticing that $N=O(\Delta t^{-1})$.
\end{proof}

According to Lemma \ref{lemma:general}, the order of the method can be obtained directly from the orders of $W$ and $W^{-1}\bm \theta_{\ell}$. However, one could think that the order of $\bm \theta_{\ell}$ alone is enough to determine the order of the method. Indeed, this is the case for, e.g., the explicit Euler method which can be written in the form $\bm u_{n+1} = G \bm u_{n}$ with $G=I+\Delta t G_1$.
This particular structure of $G$ allows one to follow a standard proof where, from
$\bm e_{n+1} = \sum_{\ell=0}^{n} G^{\ell} \bm \theta_{n-\ell}$, it is possible to obtain $\| \bm e_{n+1} \| \leq N (1+ \Delta t \|G_1\|)^N \max_{\ell=0,\dots,N-1} \| \bm \theta_{\ell} \|$. 
Then, one can estimate $(1+\Delta t\|G_1\|)^N \leq \exp(\Delta tN\|G_1\|) \leq \exp(T\|G_1\|)=: C$ to get $\| \bm e_{n+1} \| \leq C N \max_{\ell=0,\dots,N-1} \| \bm \theta_{\ell} \|$, and hence
$$
\bm \theta_{n} =O(\Delta t^p) \; \Rightarrow \; \bm e_{n} =O(\Delta t^{p-1}).
$$
Thus, the order of the method is obtained directly from the order of $\bm \theta_{\ell}$.
However, if $G$ has a more general structure, like e.g., $G=G_0+\Delta t G_1$ with $G_0 \neq I$, the above argument cannot be applied. In this case, Lemma \ref{lemma:general} says that the orders of $W^{-1}\bm \theta_{\ell}$ and $W$ are sufficient to obtain the order of the method. 
Moreover, as we are going to see in the next sections, the order of the method cannot be determined from the order of $\bm \theta_{\ell}$ only, and the effect of $W$ and $W^{-1}$ will become apparent.
Finally, we wish to remark that the orders $\Delta t^q$ and $\Delta t^p$ of $W$ and $W^{-1}$ are generally related. In fact, a direct calculation leads to $\| W^{-1} \| = \sup_{\|{\bf x}\|=1}\frac{1}{\| W {\bf x} \|}$. Thus, if $W = O(\Delta t^p)$ then $W^{-1} = O(\Delta t^q)$ with $q$ at least equal to $-p$. However, assuming that
$W = \Delta t^p \sum_j \Delta t^j W_j$ with 
$\ker W_0 \cap \ker W_1 \cap \cdots \cap \ker W_{\ell} \neq \emptyset$, for some $\ell \in \mathbb{N}$, then one can find a vector ${\bf y}$ in this intersection and obtain for $\Delta t$ small that
$\| W^{-1} \| = \sup_{\|{\bf x}\|=1}\frac{1}{\| W {\bf x} \|} = \frac{1}{\| W {\bf y} \|} = O(\Delta t^{-(p+\ell)})$.
Then, it is crucial to study the dependence of $W$ and $W^{-1}$ on $\Delta t$.

In the following subsections, we investigate the structure of the matrices $A_+$ and $A_-$ for different polynomial degrees and study the corresponding accuracy and consistency.

\subsubsection{Polynomial degree \texorpdfstring{$r = 1$}{r=1}}\label{sec:P1}
For $\mathbb{V}_n^1$ we define the basis functions 
\begin{equation}\label{basis_r1_minus}
\psi_{n,1}(t) = \frac{t_{n} - t}{\Delta t}, \quad \psi_{n,2}(t) = \frac{t - t_{n-1}}{\Delta t}, \quad t \in [t_{n-1},t_n].
\end{equation}
By substituting the expressions \eqref{basis_r1_minus} into \eqref{eq:a_m}-\eqref{eq:a_p} we obtain the matrices
%\textcolor{red}{qui si è diviso per un fattore $4\Delta t^2$ SEI SICURO? A ME SEMBRANO PROPRIO LORO, NO?}\IM{Se non ricordo male non ho diviso per niente in questa espressione, per non perderci fattori di scala che entravano nel codice.} OK!!
\begin{align*}
    A_- = \frac{1}{4\Delta t^2}\begin{bmatrix}
                  - 2s\lambda - 4 &
         -4\lambda\Delta t^2  + 2s\lambda + 4 \\
          2s\lambda + 4 &
         - 2s\lambda - 4 
    \end{bmatrix},
\end{align*}
\begin{align*}
    A_+ = \frac{1}{4\Delta t^2}\begin{bmatrix}
         2\lambda \Delta t^2 - 2s\lambda + 4 & 
         -2\lambda\Delta t^2 + 2s\lambda -4 \\
         2\lambda \Delta t^2  + 2s\lambda - 4 & 
         2\lambda\Delta t^2 - 2s\lambda + 4
         \end{bmatrix}.
\end{align*}
Thus, the matrix $G=-A_+^{-1}A_-$ is
\begin{equation}\label{eq:matrixG_P1}
    G = \begin{bmatrix}
         0 & 1 \\
         \frac{2+\lambda s}{\lambda s - \Delta t^2 \lambda -2} & 
         \frac{\Delta t^2 \lambda -4}{\lambda s - \Delta t^2 \lambda -2} \\
         \end{bmatrix}.
\end{equation}
Hence, \eqref{eq:simple_local_system2} implies that
\begin{equation}\label{eq:general_stepP1}
\begin{split}
    u_{n,2}&=u_{n+1,1}, \\
    (2+s\lambda) u_{n,1} + (\lambda \Delta t^2 - 4) u_{n,2} &= (s\lambda - \lambda \Delta t^2 -2) u_{n+1,2} .
\end{split}
\end{equation}
Notice that the equation $u_{n,2}=u_{n+1,1}$ represents the continuity of the DG solution between the adjacent slabs $I_n$ and $I_{n+1}$. 
The meaning of the second equation in \eqref{eq:general_stepP1} will be clarified in what follows (see Theorem \ref{thm:Newmark}).

Let us now, study consistency and accuracy.
A direct calculation using a Taylor expansion allows us to obtain
\begin{equation}\label{eq:thetaNewmarkdG}
    \bm \theta_{n} =
    \begin{bmatrix}
        0 \\
        \frac{48 \Delta t^2 (\ddot{u}(t_n)+\lambda u(t_n)) 
        -8 \Delta t^3 \lambda (\dddot{u}(t_n)s-3\dot{u}(t_n))
        -48 \dot{u}(t_n) \Delta t \lambda s + 4 \Delta t^4 (3\ddot{u}(t_n) \lambda+\dddot{u}(t_n))
        + O(\Delta t^5) }{24 \lambda \Delta t^2  - 24 (\lambda s -2)}
    \end{bmatrix}.
\end{equation}
Now, recalling \eqref{eq:modelProblem}, we have that $\ddot{u}(t_n)+\lambda u(t_n)=0$.
Thus, it follows that $\bm \theta_{n} = O(\Delta t)$ for $s = O(1)$,
$\bm \theta_{n} = O(\Delta t^2)$ for $s = O(\Delta t)$, 
and $\bm \theta_{n} = O(\Delta t^3)$ for $s = O(\Delta t^2)$.
Notice also that choosing $s$ with higher order does not increase the consistency order.
Thus, we focus on the case $s = O(\Delta t^2)$ and set $s=a \Delta t^2$.
Using $s=a \Delta t^2$ in \eqref{eq:general_stepP1}, we get
\begin{equation}\label{eq:general_stepP1_2}
\begin{split}
    u_{n,2}&=u_{n+1,1}, \\
    (2+a \Delta t^2\lambda) u_{n,1} + (\lambda \Delta t^2 - 4) u_{n,2} &= (a \Delta t^2\lambda - \lambda \Delta t^2 -2) u_{n+1,2} .
\end{split}
\end{equation}
Moreover, we notice the following two particular cases:
\begin{equation}\label{eq:cases}
    \| \bm \theta_{n} \| = 
    \begin{cases}
        \frac{|6 \Delta t^3 \lambda \dot{u}(t_n)
        + \Delta t^4 (3\ddot{u}(t_n) \lambda+\dddot{u}(t_n)) + O(\Delta t^5)|}{12+6 \lambda \Delta t^2} &\text{if $s=0$ ($a=0$)},\\
        \frac{|\Delta t^4 (3\ddot{u}(t_n) \lambda+\dddot{u}(t_n)) + O(\Delta t^6)|}{12+3 \lambda \Delta t^2} &\text{if $s=\frac{\Delta t^2}{2}$  ($a=\frac{1}{2}$)}.\\
    \end{cases}
\end{equation}
%\marginpar{\MG{Again need to align with Hairer/Wanner I think}}
This implies that the method is first order if $a=0$,
and second order if $a=\frac{1}{2}$.

Now, we show an equivalence result between $\mathbb{P}^1$-DG and the Newmark methods reviewed in section~\ref{sec:review:Newmark}.
For this purpose, we notice that choosing $f(u)=\lambda u$ the Newmark formula \eqref{eq:Newmark} can be written as
\begin{small}
$\begin{bmatrix}
    u_{n+1} \\ v_{n+1} \\
\end{bmatrix} 
= G_{\rm New}
\begin{bmatrix}
    u_{n} \\ v_{n} \\
\end{bmatrix}$
\end{small}, 
where $$G_{\rm New} = 
    \begin{bmatrix}
   \frac{-(1-2\beta) \Delta t^2 \lambda +2}{2\beta \Delta t^2 \lambda+2}
   &\frac{2\Delta t}{2\beta \Delta t^2 \lambda+2} \\
  \frac{(\gamma-2\beta)\Delta t^3 \lambda^2-2\Delta t \lambda}{2\beta \Delta t^2 \lambda+2}
  &\frac{((\gamma-2\beta) \Delta t^3 \lambda+(2\beta-2\gamma) \Delta t^2 \lambda+2}{2\beta \Delta t^2 \lambda+2} \\
  \end{bmatrix}.$$
We can now prove our first result.

\begin{theorem}[$\mathbb{P}^1$-DG is a Newmark scheme]\label{thm:Newmark}
Let $s=a \Delta t^2$. The scheme \eqref{eq:general_stepP1}, equivalently 
\eqref{eq:general_stepP1_2}, is a Newmark scheme with $\beta = (1-a)/2$
and $\gamma=1-a$.
In particular, the DG vector $(u_{n,1},u_{n,2})^\top$ can be mapped into the Newmark vector $(u_n,v_n)^\top$ by the matrix
\begin{equation*}
    T = \begin{bmatrix}
        0 & 1 \\
        \frac{-(a+(1-a)/2-1/2)\Delta t^2 \lambda-1}{\Delta t}
        &
        \frac{(a+(1-a)/2-1) \Delta t^2 \lambda+1}{\Delta t} \\
    \end{bmatrix},
\end{equation*}
and the similarity relation $G = T^{-1} G_{\rm New} T$ holds.
\end{theorem}

\begin{proof}
Because of the continuity condition $u_{n,2}=u_{n+1,1}$, we can introduce the notation $w_n=u_{n,2}=u_{n+1,1}$, $w_{n-1}=u_{n,1}$, and $w_{n+1}=u_{n+1,2}$. 
Therefore, by a simple manipulation, the second equation in \eqref{eq:general_stepP1_2} is equivalent to
\begin{equation}\label{eq:general_stepP1_3}
    w_{n+1}-2w_{n}+w_{n-1}+ \lambda \Delta t^2 \Bigl( \frac{1-a}{2}w_{n+1} +\frac{1}{2} w_{n} +\frac{a}{2} w_{n-1}\Bigr) =0.
\end{equation}
Now, we consider the Newmark scheme \eqref{eq:Newmark} for \eqref{eq:modelProblem}, and rewrite it as a single equation in only the $u_n$ variables. To do so, we first write the first equation in \eqref{eq:Newmark}  for two consecutive time steps as
\begin{equation*}
\begin{split}
u_{n+1} - u_{n} &= \Delta t \ v_n + \Delta t^2 \ \Bigl[ - \Bigl(\frac{1}{2}-\beta \Bigr) \lambda u_n - \beta \lambda u_{n+1} \Bigr], \\ 
u_{n} - u_{n-1} &= \Delta t \ v_{n-1} + \Delta t^2 \ \Bigl[ - \Bigl(\frac{1}{2}-\beta \Bigr) \lambda u_{n-1} - \beta \lambda u_n \Bigr]. \\ 
\end{split}
\end{equation*}
By subtracting these two equations we get
\begin{equation*}
    u_{n+1} - 2 u_{n} + u_{n-1} = \Delta t \ (v_n-v_{n-1})
    + \Delta t^2 \ 
    \Bigl[ -\beta \lambda u_{n+1} - \Bigl( \frac{1}{2} - 2\beta \Bigr) \lambda u_n + \Bigl( \frac{1}{2} - \beta \Bigr) \lambda u_{n-1} \Bigr].
\end{equation*}
Now, using the second equation in \eqref{eq:Newmark}, written as
$$v_{n} - v_{n-1} = \Delta t \ [-(1-\gamma) \lambda u_{n-1} - \gamma \lambda u_n],$$ 
we obtain
\begin{equation}\label{eq:Newmark_order2}
    u_{n+1} - 2 u_{n} + u_{n-1} + \lambda \Delta t^2 \ 
    \Bigl[ \beta u_{n+1} + \Bigl( \frac{1}{2} - 2\beta + \gamma \Bigr) u_n + \Bigl( \frac{1}{2} + \beta - \gamma \Bigr) u_{n-1} \Bigr] = 0.
\end{equation}
By comparing \eqref{eq:Newmark_order2} with \eqref{eq:general_stepP1_3}, we see that \eqref{eq:general_stepP1_3} is a Newmark scheme with $\beta = (1-a)/2$ and $\gamma=1-a$.
In particular, we proved that $u_{n} = w_{n} = u_{n,2}$ for all $n$. 
This corresponds to the first equation in $T (u_{n,1},u_{n,2})^\top=(u_n,v_n)^\top$.
To obtain the second equation of $T (u_{n,1},u_{n,2})^\top=(u_n,v_n)^\top$, one can solve the first equation in \eqref{eq:Newmark} for $v_n$, replace the result into the second equation and use the relations $u_{n} = u_{n,2}$ and $u_{n+1} = u_{n,1}$.
Finally, the similarity relation $G = T^{-1} G_{\rm New} T$ can be obtained by a direct calculation.
\end{proof}

According to the discussion provided in Section \ref{sec:review:Newmark}, the Newmark scheme is 
%\marginpar{\MG{Again align with Hairer/Wanner}}
%fourth-order consistent and 
second-order accurate if and only if $\gamma=\frac{1}{2}$, it is stable if $\gamma \geq \frac{1}{2}$, while the choice of $\beta=\frac{1}{4}$ provides conservation properties.
Using Theorem \ref{thm:Newmark}, we can immediately transfer this properties to our DG scheme. 
Since $\gamma=\frac{1}{2}$ is obtained only for $a=\frac{1}{2}$ ($s=\frac{\Delta t^2}{2}$), which also corresponds to $\beta=\frac{1}{4}$, we obtain 
%\marginpar{\MG{Again align with Hairer/Wanner}}
that our DG scheme is first-order accurate
%and third-order consistent 
for all values of $a$ different from $\frac{1}{2}$. 
This is also in agreement with \eqref{eq:cases}.
Thus, the choice $a=0$ (that is $s=0$) corresponds to a first-order convergent method.

Even though one can benefit from the analogy with the class of Newmark methods, we prove convergence of our DG scheme using the general results obtained in Section \ref{sec:algebraic}.
This will be useful to get a detailed analysis and compare the $\mathbb{P}^1$ and $\mathbb{P}^2$ cases.

\begin{lemma}[Eigen-decomposition of $G$ for $\mathbb{P}^1$]\label{lemma:stability}
    Consider the matrix $G$ defined in \eqref{eq:matrixG_P1}
    and let $s:=a \Delta t^2$.
    The corresponding eigen-decomposition is $G=WDW^{-1}$ with
    \begin{equation}
        W = \begin{bmatrix}
            1 & 1 \\
            \lambda_1  & \lambda_2 \\
        \end{bmatrix}
        \quad
        \text{and}
        \quad
        D = \begin{bmatrix}
            \lambda_1 & 0\\
            0 & \lambda_2 
        \end{bmatrix},
    \end{equation}
    with $\lambda_{1,2} = \frac{\pm\Delta t \sqrt{(4a^2-4a+1) \Delta t^2 \lambda^2-16\lambda}+\Delta t^2 \lambda-4}{2(a-1) \Delta t^2 \lambda-4}$.
    Moreover, it holds that
    \begin{itemize}\itemsep0em
        \item[\rm (a)] If $a=\frac{1}{2}$, then $\rho(G) = 1$ for all $\Delta t$ and $\lambda$.
        \item[\rm (b)] If $a=0$, then $\rho(G) < 1$ for all $\Delta t$ and $\lambda$.
        \item[\rm (c)] If $a\in \bigl(\frac{1}{2}-\frac{2}{\Delta t \sqrt{\lambda}},\frac{1}{2}\bigr)$, then $\rho(G) < 1$. If $a\in \bigl(\frac{1}{2},\frac{1}{2}+\frac{2}{\Delta t \sqrt{\lambda}}\bigr)$, then $\rho(G) > 1$.
        \item[\rm (d)] $W = W_0 + O(\Delta t) = O(1)$, with $W_0 = 
            \begin{bmatrix}
                1 & 1 \\
                1 & 1 \\
            \end{bmatrix}
       $.
    \end{itemize}
\end{lemma}

\begin{proof}
    The matrices $V$ and $D$ can be obtained by a direct calculation.
    Let us now prove point {\rm (a)}.
    If we set $a=\frac{1}{2}$, then $\lambda_{1,2} = \frac{4 - \lambda \Delta t^2 \pm i 4 \Delta t \sqrt{\lambda}}{4 + \lambda \Delta t^2}$, which have both unit modulus.
    %(as in Section \ref{sec:Newmark}, by the similarity of $G$ with $G_{\rm New}$).
    Let us now focus on point {\rm (b)}. By setting $a=0$ and using the change of variables $x= \Delta t \sqrt{\lambda}$, the eigenvalues of $G$ become $\lambda_{1,2} = \frac{\pm x \sqrt{ x^2-16}+x-4}{-2 x-4}$.
    We distinguish two cases: $x\leq 4$ and $x>4$.
    In the first case, the two eigenvalues are complex conjugate and thus have the same modulus. A direct calculation reveals that
    $|\lambda_{1,2}| < 1$ if and only if $-4 x^4-8 x^2 < 0$, which holds for all $x>0$ (and thus for all $\Delta t$ and $\lambda$).
    In the second case, one has $\rho(G) = \frac{|x \sqrt{ x^2-16}+x^2-4|}{|2 x^2+4|}$.
    Taking the square, one can compute that $\rho(G)^2=1/9<1$ for $x=4$ and that $\lim_{x \rightarrow \infty} \rho(G)^2=1$. Moreover, a direct calculation allows us to get
    \begin{equation*}
        \frac{d (\rho(G)^2)}{dx} = \frac{(x\sqrt{x^2-16}+x^2-4)^2}{(-2x^2-4)^2}
        \frac{\frac{2 x^4+4 x^2}{\sqrt{x^2-16}}+4x^3+\sqrt{x^2-16}(2x^2+4)+4x}{(x^2+2)},
    \end{equation*}
    which is clearly positive for all $x>4$.
    Thus, the spectral radius $\rho(G)$ is monotonically increasing and bounded by $1$.
    Let us now focus on point {\rm (c)}.
    We consider only the case $(4a^2-4a+1)x \leq 16$, which is equivalent to $a \in \bigl[\frac{1}{2}-\frac{2}{x},\frac{1}{2}+\frac{2}{x} \bigr]$. 
    In this case, the two eigenvalues are complex conjugate and have the same modulus, and a direct calculation reveals that $|\lambda_{1,2}| < 1$ if and only if
    $(-8 a^2+12 a-4) x^2+16 a-8 < 1$, which holds for all $a \in \bigl(\frac{1}{2}-\frac{2}{x},\frac{1}{2}\bigr)$. %(the values $a \geq \max\bigl\{\frac{1}{2}+\frac{2}{x},1+\frac{2}{x^2}\bigr\}$ are excluded)
    Thus, $\rho(G)< 1$ for $a \in \bigl(\frac{1}{2}-\frac{2}{x},\frac{1}{2}\bigr)$
    and $\rho(G)>1$ for $a \in \bigl(\frac{1}{2},\frac{1}{2}+\frac{2}{x} \bigr)$.
    Finally, point {\rm (d)} follows by a direct Taylor expansion of $W$.
\end{proof}

%\begin{remark}[Lemma~\ref{lemma:stability} and the Newmark framework]
%\begin{itemize}
%\item In view of the similarity of $G$ and $G_{\rm New}$, the results of Lemma \ref{lemma:stability} could be obtained from the stability analysis of the Newmark methods.
%In fact, in the conservative case Lemma \ref{lemma:stability} states that for the choice $a=1/2$ (corresponding to $\gamma=1/2$ and $\beta = 1/4$) the eigenvalues of $G$ have unit modulus, implying that Newmark is conservative.
%\item For $\lambda$ fixed and $\Delta t \rightarrow 0$ the result of point {\rm (c)} covers all values of $a \geq 0$. In the limiting case $\Delta t \rightarrow 0$, one gets that $\rho(G)\leq 1$ for $a \leq 1/2$ (equivalent to $\gamma \geq 1/2$), which is similar to the absolute stability result valid for the Newmark scheme discussed in Section \ref{sec:review:Newmark}.
%\end{itemize}
%\end{remark}

Now, we can finally come back to the consistency and accuracy of our DG method.

\begin{theorem}[Consistency and accuracy for $\mathbb{P}^1$]
    Consider the DG scheme \eqref{eq:error}-\eqref{eq:simple_local_system2} with $s=a \Delta t^2$. Then $\bm \theta_{n} =O(\Delta t^{p+1})$, $W^{-1}\bm \theta_{n}=O(\Delta t^p)$, and $\bm e_{n} =O(\Delta t^{p-1})$ with $p=2$ for $a\in \{0\} \cup \bigl(\frac{1}{2}-\frac{2}{\Delta t \sqrt{\lambda}},\frac{1}{2}\bigr)$ and $p=3$ for $a=\frac{1}{2}$.
\end{theorem}

\begin{proof}
    The result follows from Lemma \ref{lemma:general} and Lemma \ref{lemma:stability} once the estimates for the consistency errors are obtained.
    Using \eqref{eq:thetaNewmarkdG} with $s=a \Delta t^2$ (and recalling that $\ddot{u}(t_n)+\lambda u(t_n)=0$),
    one obtains
    \begin{equation*}
    \bm \theta_{n} =
    \begin{bmatrix}
        0 \\
        \frac{ 
        8 \Delta t^3 \lambda (3-6a)\dot{u}(t_n)
        + 4 \Delta t^4 (3\ddot{u}(t_n) \lambda+\dddot{u}(t_n))
        + O(\Delta t^5) }{24 \lambda \Delta t^2 (1-a) + 48}
    \end{bmatrix}.
    \end{equation*}
    Therefore, $\bm \theta_{n} = O(\Delta t^3)$ for all considered values of $a$, and $\bm \theta_{n} = O(\Delta t^4)$ for $a=\frac{1}{2}$.
    Similarly, a direct calculation using a Taylor expansion leads to
    \begin{equation*}
    W^{-1}\bm \theta_{n} =
    \begin{bmatrix}
        \frac{(2a-1)\dot{u}(t_n) \Delta t^2\sqrt{\lambda}}{-4 i} 
        + \frac{\Delta t^3}{192} 
        \bigl( \frac{24 \ddot{x}(t_n) \lambda + 8 \dddot{x}(t_n)}{i \sqrt{\lambda}} + C_1(a)\lambda^3 + C_2(a)\lambda^2 \bigr)
        + O(\Delta t^4) \\
        \frac{(2a-1)\dot{u}(t_n) \Delta t^2\sqrt{\lambda}}{4 i} 
        - \frac{\Delta t^3}{192}
        \bigl( \frac{24 \ddot{x}(t_n) \lambda + 8 \dddot{x}(t_n)}{i \sqrt{\lambda}} + C_1(a)\lambda^3 + C_2(a)\lambda^2 \bigr)
        + O(\Delta t^4) \\
    \end{bmatrix},
    \end{equation*}
    where $i$ is the imaginary unit, and
    $$
    C_1(a) = \frac{12u(t_n)(a- a^2) -3 u(t_n)}{i \lambda^{3/2}}
    \quad \text{and} \quad 
    C_2(a) = \frac{12\ddot{u}(t_n)(a- a^2) -3 \ddot{u}(t_n)}{i \lambda^{3/2}}.
    $$
    Thus, $W^{-1} \bm \theta_{n} = O(\Delta t^2)$ and $W^{-1} \bm \theta_{n} = O(\Delta t^3)$ for $a=\frac{1}{2}$.
\end{proof}

\subsubsection{Polynomial degree \texorpdfstring{$r=2$}{r=2}}\label{sec:P2}
For $\mathbb{V}_{n}^2$ we consider the basis functions 
\begin{align}\label{basis_r2_minus}
\psi_{n,1}(t) & = -\frac{2( t -  t_{n-\frac12})(t_n-t)}{\Delta t^2}, \quad \psi_{n,2}(t) = \frac{4(t - t_{n-1})(t_n-t)}{\Delta t^2}, \nonumber \\ \psi_{n,3}(t) & = \frac{2(t_{n-1}-t)( t_{n-\frac12} - t )}{\Delta t^2},   \quad t \in [t_{n-1},t_n].
\end{align}
%while for for $V_{n+1}^2$ we consider  
%\begin{align}\label{basis_r2_plus}
%\psi_1^+(t) & = -\frac{2(t- t_{n+\frac12})(t_{n+1}-t)}{\Delta t^2}, \quad \psi_2^+(t) = \frac{4(t - t_{n})(t_{n+1}-t)}{\Delta t^2}, \nonumber \\ \psi_3^+(t) & =  \frac{2(t_n-t)( t_{n+\frac12} -t)}{\Delta t^2},     \quad t \in [t_{n},t_{n+1}],
%\end{align}
%Proceeding as before, and substituting  \eqref{basis_r1_minus}-\eqref{basis_r1_plus} into \eqref{eq:a_m}-\eqref{eq:a_p} we get
%\begin{equation}\label{eq:system_small_r2}
%    \left(A_- | A_+ \right) \left( \begin{array}{c}
%    u_1^- \\
%    u_2^- \\
%    u_3^- \\
%    u_1^+ \\
%    u_2^+ \\
%    u_3^+ 
%    \end{array}
%    \right)  = \left( \begin{array}{c}
%    0 \\ 
%    \vdots \\
%    0 
%    \end{array}
%    \right), 
%\end{equation}
%where
Replacing \eqref{basis_r1_minus} into \eqref{eq:a_m}-\eqref{eq:a_p}, we get
\begin{align*}
    A_- = \frac{1}{12 \Delta t^2}
    \begin{bmatrix}
         3(6s\lambda+12)&  -12(6s\lambda +12) & -3(4\lambda\Delta t^2 - 18 s\lambda -36) \\
        -12(2s\lambda +4) & 48(2s\lambda+4)  & -36(2s\lambda+4) \\
        3(2s\lambda+4) & -12(2s\lambda +4) & 9(2s\lambda+4)
    \end{bmatrix},
\end{align*}
\begin{align*}
    A_+ = \frac{1}{12 \Delta t^2}
    \begin{bmatrix}
         3(2\lambda\Delta t^2 -18s\lambda+20)&  -4(2\lambda \Delta t^2-18s\lambda+12) & (2\lambda\Delta t^2 - 18 s\lambda -12) \\
        4(2\lambda\Delta t^2 +18 s\lambda -36) & -48(2s\lambda-4)  & -4(2\lambda\Delta t^2-6s\lambda+12) \\
        - (2\lambda\Delta t^2 + 18 s\lambda -84) & 4(2\lambda\Delta t^2 + 6 s\lambda -36) & 
        3(2\lambda\Delta t^2 -2 s\lambda +20)
    \end{bmatrix}.
\end{align*}
Thus, the matrix $G=-A_+^{-1}A_-$ is given by
\begin{equation*}
    G = 
    %\begin{small}
    \begin{bmatrix}
        0 & 0 & 1 \\
        -\frac{(3\Delta t^2 \lambda^2+36 \lambda)s+6\Delta t^2 \lambda+72}{2d} 
        & \frac{s+12 \Delta t^2 \lambda+144}{d} 
        & -\frac{36\lambda s-\Delta t^4 \lambda^2+12 \Delta t^2 \lambda+360}{2d}  \\
        \frac{(3 \Delta t^2 \lambda^2-36\lambda)s+6\Delta t^2\lambda-72}{d} 
        & -\frac{(12 \Delta t^2 \lambda^2-144 \lambda)s+24 \Delta t^2 \lambda-288}{d}
        & -\frac{72 \lambda s+\Delta t^4\lambda^2-48 \Delta t^2 \lambda+288}{d}\\
    \end{bmatrix},
    %\end{small},
\end{equation*}
where $d = (9 \Delta t^2 \lambda^2+36 \lambda)s-\Delta t^4 \lambda^2-6 \Delta t^2 \lambda-72$.
As for the $\mathbb{P}^1$ case, we notice that the first row of $G$ corresponds to the continuity condition ($u_{n,3}=u_{n+1,1}$).

Now, we proceed as in Section~\ref{sec:P1} and study $\bm \theta_{n}$ by a Taylor expansion of $u(t)$:
\begin{equation*}
\bm \theta_{n} =
\begin{bmatrix}
    0 \\
    \frac{ \dot{u}(t_n) s \lambda \Delta t }{s \lambda - 2} + O(\Delta t^3) \\
    \frac{2 \dot{u}(t_n) s \lambda \Delta t }{s \lambda - 2} + O(\Delta t^3) \\
\end{bmatrix}.
\end{equation*}
Thus, to get that $\bm \theta_{n}$ is at least of order $\Delta t^3$, one needs that $s=O(\Delta t^2)$. Hence, as for the $\mathbb{P}^1$ case, we set $s=a\Delta t^2$ and study the method with respect to the parameter $a$. In this case, the consistency error $\bm \theta_{n}$ becomes
\begin{equation}
\label{eq:theta_P2}
\bm \theta_{n} =
\begin{bmatrix}
    0 \\
    -\frac{ (8\dot{u}(t_n) a \lambda - \dddot{u}(t_n) )\Delta t^3 }{16} 
    + \frac{(4 u(t_n) \lambda^2-5 \dddot{u}(t_n) )\Delta t^4}{384}+ O(\Delta t^5) \\
    -\frac{ ((12a-2)\dot{u}(t_n)  \lambda - 3\dddot{u}(t_n) )\Delta t^3 }{16} 
    + \frac{(4 u(t_n) \lambda^2- \dddot{u}(t_n) )\Delta t^4}{96}+ O(\Delta t^5) \\
\end{bmatrix},
\end{equation}
which shows that there is no value of $a$ eliminating simultaneously the third-order terms of $\bm \theta_{n}$.
In particular, in view of the relation $\frac{d}{dt}(u(t_n)  \lambda + 3\ddot{u}(t_n)) =0$, the choice $a=-1/8$ cancels the third-order term in the second component of $\bm \theta_{n}$, but not the one in the third component. 
Similarly, the choice $a=-1/12$ cancels the third-order term in the third component of $\bm \theta_{n}$, but not the one in the second component. 
Therefore, the analysis of $\bm \theta_{n}$ does not provide any information on the choice of $a$: for all values of $a$ one gets that $\bm \theta_{n} = O(\Delta t^3)$.

To better understand the role of the parameter $a$, we recall the framework of Lemma~\ref{lemma:general} (used already in Section~\ref{sec:P1}) and study $W^{-1}\bm \theta_{n}$.
To do so, we first need to compute the eigen-decomposition $G=WDW^{-1}$ obtained by the matrices
\begin{equation}\label{eq:eigendecompositionP2}
    D = \begin{bmatrix}
        0 & 0 & 0 \\
        0 & \lambda_1 & 0 \\
        0 & 0 & \lambda_2 \\
    \end{bmatrix}
    \quad \text{and} \quad
    W = \begin{bmatrix}
        1 & 1 & 1 \\
        \frac{1}{4} & \widehat{\lambda}_1 & \widehat{\lambda}_2 \\
        0 & \lambda_1 & \lambda_2 \\
    \end{bmatrix}
\end{equation}
with
\begin{equation}\label{eq:eigenvaluesP2}
\begin{split}
    \lambda_{1,2} &= 
    \frac{(6a-1)\Delta t^4\lambda^2+60\Delta t^2\lambda-144
    }{(18a-2)\Delta t^4\lambda^2+(72a-12)\Delta t^2\lambda-144} \\
    & \qquad \qquad \pm
    \frac{\Delta t \sqrt{a_4\Delta t^6\lambda^4+a_3\Delta t^4\lambda^3+a_2\Delta t^2\lambda^2+a_1\lambda}
    }{(18a-2)\Delta t^4\lambda^2+(72a-12)\Delta t^2\lambda-144}, \\
    \widehat{\lambda}_{1,2} &= 
    \frac{(1-6a)\Delta t^6\lambda^3-(288a+24)\Delta t^4 \lambda^2-(864a+432)\Delta t^2\lambda-3456
    }{(36a-4)\Delta t^6 \lambda^3+(24-288a)\Delta t^4\lambda^2-1728a\Delta t^2\lambda+3456} \\
    &\qquad \qquad \mp
    \frac{(\Delta t^3 \lambda+12\Delta t)\sqrt{{a}_4\Delta t^6\lambda^4+{a}_3\Delta t^4\lambda^3+{a}_2\Delta t^2 \lambda^2 + {a}_1\lambda}
    }{(36a-4)\Delta t^6 \lambda^3+(24-288a)\Delta t^4\lambda^2-1728a\Delta t^2\lambda+3456}, \\
\end{split}
\end{equation}
where 
\begin{equation*}
    a_4 = 144a^2-24a+1, \, a_3 = 1728a^2+720a-144, \,a_2 = 5184a^2+3456, \, a_1 = -20736.
\end{equation*}
%and
%\begin{equation*}
%    \widehat{a}_4 = 144a^2-24a+1, \, \widehat{a}_3 = 1728a^2+720a-144, \, \widehat{a}_2 = 5184 a^2+3456, \, \widehat{a}_1 = -20736.
%\end{equation*}
Now, using again a Taylor expansion, we obtain by a direct calculation that
\begin{equation}\label{eq:Vtheta_P2}
    W^{-1}\bm \theta_{n} =
    \begin{bmatrix}
        \frac{(4 \dot{u}(t_n)+3\dddot{u}(t_n))\Delta t^3}{12} + O(\Delta t^4) \\
        -\frac{i\sqrt{\lambda}((2a-1)\lambda\dot{u}(t_n)-\dddot{u}(t_n) )\Delta t^2}{4 \lambda} 
        - \frac{C_1(a)\Delta t^3}{48\lambda} 
        + O(\Delta t^4) \\
        -\frac{i\sqrt{\lambda}((-2a+1)\lambda\dot{u}(t_n)+\dddot{u}(t_n) )\Delta t^2}{4 \lambda}
        - \frac{C_2(a)\Delta t^3}{48\lambda} 
        + O(\Delta t^4) \\
    \end{bmatrix},
    \end{equation}
where
\begin{small}
\begin{equation*}
\begin{split}
    C_1(a) &= [8\dot{u}(t_n)+(-3\sqrt{\lambda}ia^2+6\sqrt{\lambda} i a+4\sqrt{\lambda}i)u(t_n)]\lambda^2 \\
    &+[(-3\sqrt{\lambda}ia^2+6\sqrt{\lambda}ia)\ddot{u}(t_n)+6\dddot{u}(t_n)]\lambda-2\sqrt{\lambda}i\ddddot{u}(t_n), \\
    C_2(a) &= [8\dot{u}(t_n)-(-3\sqrt{\lambda}ia^2+6\sqrt{\lambda} i a+4\sqrt{\lambda}i)u(t_n)]\lambda^2 \\
    &+[(3\sqrt{\lambda}ia^2-6\sqrt{\lambda}ia)\ddot{u}(t_n)+6\dddot{u}(t_n)]\lambda+2\sqrt{\lambda}i\ddddot{u}(t_n). \\
\end{split}
\end{equation*}
\end{small}

\noindent
In view of the relation $\frac{d}{dt}(u(t_n)  \lambda + 3\ddot{u}(t_n)) =0$, we observe from \eqref{eq:Vtheta_P2} that $a=0$ is the only value that can make $W^{-1}\bm \theta_{n}$ of order $\Delta t^3$.
This is in contrast with the fact that some nonzero values of $a$ can improve the order of some components of $\bm \theta_{n}$.
This observation shows very clearly that, to study our DG methods the key quantity to look at is $W^{-1}\bm \theta_{n}$, according to the framework of Lemma~\ref{lemma:general}.
A study of $\bm \theta_{n}$ only can not be sufficient for fully understanding the behavior of the method.
We next summarize our findings, where we focus on three different choices of $a$:
$a=0$, making $W^{-1}\bm \theta_{n}$ of order $\Delta t^3$, and $a=-1/8$ and $a=-1/12$, making some components of $\bm \theta_{n}$ of order $\Delta t^4$.

\begin{lemma}[Eigen-decomposition of $G$ for $\mathbb{P}^2$]\label{lemma:stability:P2}
The eigen-decomposition of $G$ is $G=WDW^{-1}$ obtained by the matrices given in \eqref{eq:eigendecompositionP2}.
For all values $a=0,-1/8,-1/12$ it holds that $\rho(G)\leq 1$ for all small enough values of $\sqrt{\lambda}\Delta t$, and $W=O(1)$.
\end{lemma}
\begin{proof}
    The eigen-decomposition $G=WDW^{-1}$ and the fact that $W=O(1)$ can be obtained by direct calculations (and a Taylor expansion).
    Let us now study the eigenvalues of $a$.
    For $a=0$, and introducing the (positive) variable $x=\Delta t \sqrt{\lambda}$, we have
    \begin{equation*}%\label{eq:eigenvaluesP2a0}
\begin{split}
    \lambda_{1,2} &= 
    \frac{-x^4+60x^2-144
    \pm x \sqrt{x^6-144x^4 + 3456x^2-20736}
    }{-2x^4-12x^2-144}.
\end{split}
\end{equation*}
Now, for $x$ small enough, it holds that
$x^6-144x^4 + 3456x^2-20736 \leq 0$.
In this case, the eigenvalues $\lambda_1$ and $\lambda_2$ are complex conjugates and have the same modulus. A direct calculation shows that $|\lambda_{1,2}| \leq 1$ if and only if
$-4 x^8-24 x^6-288 x^4\leq 0$, which holds for all $x \geq 0$.
The same arguments can be used to obtain the result for $a=-1/12$ and $a=-1/8$.
\end{proof}

\begin{remark}[Spectral radius of $\rho(G)$]
    The result of Lemma~\ref{lemma:stability:P2} on the spectral radius $\rho(G)$ could be improved.
    In particular, one has that $\rho(G) \leq 1$ for $a\leq0$ and any values of $\Delta t$ and $\lambda$,
    as shown in Figure~\ref{fig:rho_P2} (left).
    \begin{figure}
        \centering
        \includegraphics[scale=0.4]{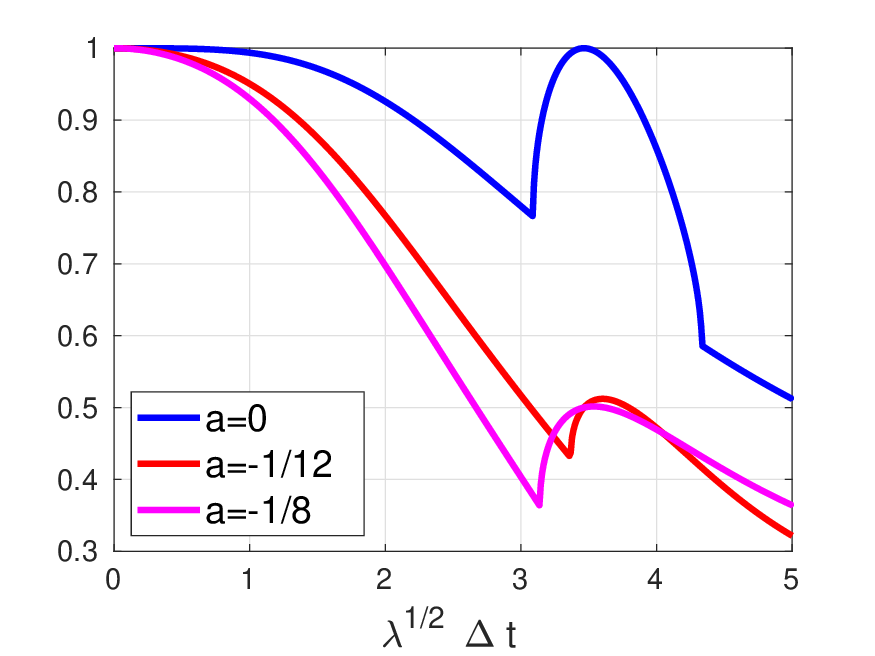}
        \includegraphics[scale=0.4]{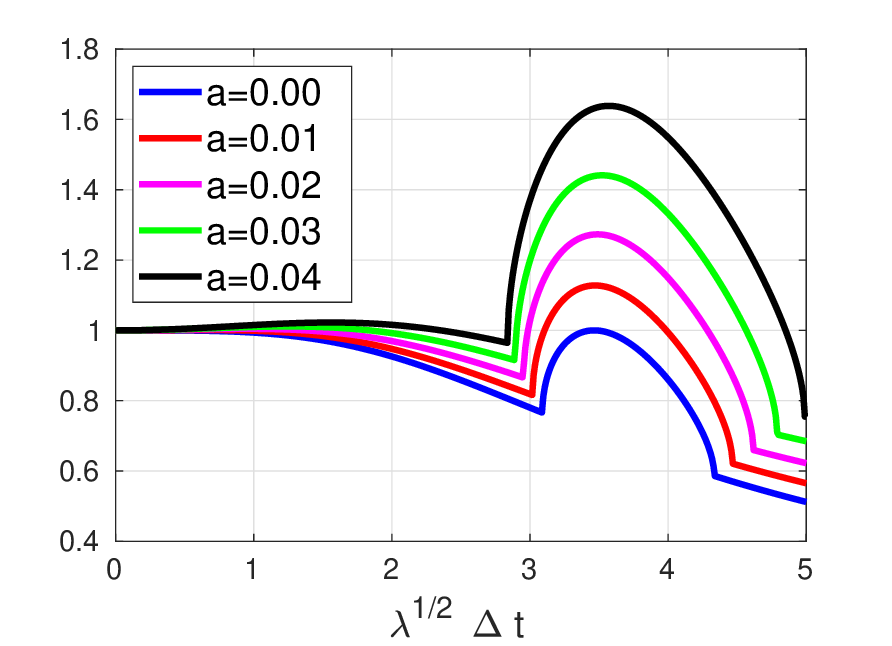}
        \caption{Spectral radius of $G$ as a function of $\sqrt{\lambda}\Delta t$ for different values of $a$.}
        \label{fig:rho_P2}
    \end{figure}
    Instead, if $a >0$ the spectral radius $\rho(G)$ becomes positive, meaning that the method is unstable.
    However, proving these results requires the study of third-order polynomials, like $x^6-144x^4 + 3456x^2-20736$ shown in the proof of Lemma~\ref{lemma:stability:P2}. In this case, for example, the corresponding roots can be obtained by using Cardano's formula. However, this approach leads to complicated formulas and their study is beyond the scope of this work.
\end{remark}

\begin{theorem}[Consistency and accuracy for $\mathbb{P}^2$]
Consider the DG scheme \eqref{eq:error}-\eqref{eq:simple_local_system2} with $s=a \Delta t^2$. Then, under the assumptions of Lemma~\ref{lemma:stability:P2}, it holds that
$\bm \theta_{n} =O(\Delta t^{3})$ for any value of $a \in \mathbb{R}$, $W^{-1}\bm \theta_{n}=O(\Delta t^2)$ for any $a \in \mathbb{R}$ and $W^{-1}\bm \theta_{n}=O(\Delta t^3)$ for $a=0$. Correspondingly, $\bm e_{n} =O(\Delta t^{p})$ with $p=1$ for $a = -1/12$ and $a=-1/8$
and $p=2$ for $a=0$.
\end{theorem}

\begin{proof}
    The proof follows from Lemma~\ref{lemma:stability:P2}, Lemma~\ref{lemma:general} and the discussion above.
\end{proof}

From the results obtained in this section, we can clearly state that, if one considers a second-order problem of the form \eqref{eq:modelProblem}, it is more appropriate to use a DG method with $\mathbb{P}^1$ elements, rather than $\mathbb{P}^2$.
In fact, an appropriate choice of the parameter $a$ leads the $\mathbb{P}^1$-DG to second-order accuracy, which is the best accuracy that one can obtain with $\mathbb{P}^2$.

\subsubsection{Polynomial degree \texorpdfstring{$r=3$}{r=3}}\label{sec:P3}
For $\mathbb{V}_{n}^3$ we consider the set of basis functions 
\begin{align}\label{basis_r3_minus}
\psi_{n,1}(t) & = \frac{(t_{n}-3t+2t_{n-1})(2t_{n}-3t+t_{n-1})(t_{n}-t)}{2\Delta t^3}, \nonumber \\ 
\psi_{n,2}(t) & = 9\frac{(t - t_{n-1})(2t_{n}-3t+t_{n-1})(t_{n}-t)}{2\Delta t^3}, \nonumber \\ 
\psi_{n,3}(t) & = - 9\frac{(t - t_{n-1})(t_{n}-3t+2t_{n-1})(t_{n}-t)}{2\Delta t^3}, \nonumber \\
\psi_{n,4}(t) & = \frac{(t-t_{n-1})(t_{n}-3t+2t_{n-1})(2t_{n}-3t+t_{n-1})}{2\Delta t^3}, \quad t \in [t_{n-1},t_{n}]. \nonumber
\end{align}
%and the following set  
%\begin{align}\label{basis_r3_plus}
%\psi_1^+(t) & = \frac{(t_{n+1}-3t+2t_n)(2t_{n+1}-3t+t_n)(t_{n+1}-t)}{2\Delta t^3}, \nonumber \\ 
%\psi_2^+(t) & = 9\frac{(t - t_n)(2t_{n+1}-3t+t_n)(t_{n+1}-t)}{2\Delta t^3}, \nonumber \\ 
%\psi_3^+(t) & = - 9\frac{(t - t_n)(t_{n+1}-3t+2t_n)(t_{n+1}-t)}{2\Delta t^3}, \nonumber \\
%\psi_4^+(t) & = \frac{(t-t_n)(t_{n+1}-3t+2t_n)(2t_{n+1}-3t+t_n)}{2\Delta t^3}, \quad t \in [t_{n},t_{n+1}], \nonumber
%\end{align}
%for $V_{n+1}^3$. 
In this case, one can proceed as in the previous sections to estimate $\bm \theta_n$, $W^{-1} \bm \theta_n$, and $\bm e_n$. 
This leads to an estimate $\bm \theta_n = O(\Delta t^4)$, $W^{-1} \bm \theta_n = O(\Delta t^4$) and thus by Lemma~\ref{lemma:general}, to $\bm e_n=O(\Delta t^3)$. If one chooses $s=a\Delta t^2$ or $s=a\Delta t^3$, these orders cannot be improved by an appropriate choice of $a$, exactly as in the $\mathbb{P}^2$ case.

\subsection{DG2 as GLMs}\label{sec:dG:second:GLM}
In this section, we discuss the relation between the DG method introduced at the beginning of section~\ref{sec:dG:second} and the GLMs of section~\ref{sec:review:GLM}.
In particular, we have seen in section~\ref{sec:algebraic} that DG methods for the second-order equation \eqref{eq:modelProblem} can be written in the form ${\bf u}_{n+1}=G{\bf u}_{n}$, where $G=-A_+^{-1}A_-$. Now, the matrix $G$ generally has a special structure (see sections \ref{sec:P1}, \ref{sec:P2} and \ref{sec:P3}), which is used to prove that the main result of this section is given by the following theorem.
\begin{theorem}[DG and GLM]
    Assume that the determinant of $A_+$ has the form $d=\alpha_0 + \sum_{k=1}^N \lambda^k \Delta t^{2k} \alpha_k$, for some coefficients $\alpha_k \in \mathbb{R}$ independent of $\Delta t$ and $\lambda$. Moreover, assume that $G \in \mathbb{R}^{(r+1) \times (r+1)}$ has the form 
    \begin{equation}\label{eq:GLM:G}
        G=\frac{1}{d} \Bigl( G_0 - \sum_{k=1}^N \lambda^k \Delta t^{2k} G_k \Bigr),
    \end{equation}
    where the $G_k$ are independent of $\Delta t$ and $\lambda$.
    Then the iteration ${\bf u}_{n+1}=G{\bf u}_{n}$ is equivalent to a GLM \eqref{eq:GLM} with ${\bf y}^{[n]} \in \mathbb{R}^{2(r+1)}$ such that $({\bf y}^{[n]})_{1:r+1} = {\bf u}_{n}$, and 
    \begin{equation}\label{eq:dG:GLM:matrices1}
        A= \begin{bmatrix}
            \frac{\alpha_1}{\alpha_0}I_{r+1} & \frac{\alpha_2}{\alpha_0}I_{r+1} & \cdots & \frac{\alpha_{N-1}}{\alpha_0}I_{r+1} & \frac{\alpha_N}{\alpha_0}I_{r+1} \\
            -I_{r+1} & 0 & \cdots & 0 & 0 \\
            0 & -I_{r+1} & \cdots & 0 & 0 \\
            \vdots &  & \ddots & & \vdots \\
            0 & 0 & \cdots & -I_{r+1} & 0 \\
        \end{bmatrix},
        \quad
        U = \begin{bmatrix}
            \frac{1}{\alpha_0}G_0 & I_{r+1} \\
             0 & 0 \\
             \vdots & \vdots \\
             0 & 0 \\
        \end{bmatrix}
    \end{equation}
    \begin{equation}\label{eq:dG:GLM:matrices2}
        B= \begin{bmatrix}
            \frac{\alpha_1}{\alpha_0}I & \frac{\alpha_2}{\alpha_0}I_{r+1} & \cdots & \frac{\alpha_{N-1}}{\alpha_0}I_{r+1} & \frac{\alpha_N}{\alpha_0}I \\
            \frac{1}{\alpha_0}G_1 & \frac{1}{\alpha_0}G_2 & \cdots & \frac{1}{\alpha_0}G_{N-1} & \frac{1}{\alpha_0}G_N \\
        \end{bmatrix},
        \quad
        V = \begin{bmatrix}
            \frac{1}{\alpha_0}G_0 & I_{r+1} \\
             0 & 0 \\
        \end{bmatrix},
    \end{equation}
    where $I_{r+1} \in  \mathbb{R}^{(r+1) \times (r+1)}$ is the identity matrix.
\end{theorem}
\begin{proof}
    Using the expression of $d$ and \eqref{eq:GLM:G}, the formula ${\bf u}_{n+1}=G{\bf u}_{n}$ becomes
    \begin{equation}\label{eq:GLM:dG:1}
        {\bf u}_{n+1} 
        + \sum_{k=1}^N \lambda^k \Delta t^{2k} \frac{\alpha_k}{\alpha_0} {\bf u}_{n+1} 
        = \frac{1}{\alpha_0} G_0 {\bf u}_{n}
        - \sum_{k=1}^N \lambda^k \Delta t^{2k} \frac{1}{\alpha_0} G_k{\bf u}_{n}.
    \end{equation}
    Now to write \eqref{eq:GLM:dG:1} in the form \eqref{eq:GLM}, we introduce the (explicit) variables ${\bf y}^{[n]} \in \mathbb{R}^{2(r+1)}$,
    \begin{equation}\label{eq:GLM:dG:2}
        \begin{split}
        ({\bf y}^{[n]})_{1:r+1} &:= {\bf u}_{n}, \\
        ({\bf y}^{[n]})_{r+1:2(r+1)} &:= - \sum_{k=1}^N \frac{\lambda^k \Delta t^{2k}}{\alpha_0} G_k {\bf u}_{n},
        \end{split}
    \end{equation}
    and define the (implicit) variable $Y_1$ via the equation
    \begin{equation}\label{eq:GLM:dG:3}
        Y_1
        + \sum_{k=1}^N \lambda^k \Delta t^{2k} \frac{\alpha_k}{\alpha_0} Y_1 
        = \frac{1}{\alpha_0} G_0 ({\bf y}^{[n]})_{1:r+1}
        + ({\bf y}^{[n]})_{r+1:2(r+1)}.
    \end{equation}
    Next, we introduce the variables $Y_k$, $k=2,\dots,N$, as
    \begin{equation}\label{eq:GLM:dG:4}
            Y_k := \lambda \Delta t^2 Y_{k-1}, \text{ for $k=2,\dots,N$},
    \end{equation}
    which allows us to rewrite \eqref{eq:GLM:dG:3} as
    \begin{equation}\label{eq:GLM:dG:5}
        Y_1
        + \lambda \Delta t^{2} \sum_{k=1}^N \frac{\alpha_k}{\alpha_0} Y_k 
        = \frac{1}{\alpha_0} G_0 ({\bf y}^{[n]})_{1:r+1}
        + ({\bf y}^{[n]})_{r+1:2(r+1)}.
    \end{equation}
    Now, we notice that \eqref{eq:GLM:dG:1}, \eqref{eq:GLM:dG:3}, and \eqref{eq:GLM:dG:5} are equivalent and hence $Y_1 = {\bf u}_{n+1} = ({\bf y}^{[n+1]})_{1:r+1}$.
    This allows us to formulate \eqref{eq:GLM:dG:5} as
    \begin{equation}\label{eq:GLM:dG:6}
        ({\bf y}^{[n+1]})_{1:r+1}
        + \lambda \Delta t^{2} \sum_{k=1}^N \frac{\alpha_k}{\alpha_0} Y_k 
        = \frac{1}{\alpha_0} G_0 ({\bf y}^{[n]})_{1:r+1}
        + ({\bf y}^{[n]})_{r+1:2(r+1)}.
    \end{equation}
    Now, using the second line of \eqref{eq:GLM:dG:2} and \eqref{eq:GLM:dG:4}, we can write
    \begin{equation}\label{eq:GLM:dG:7}
        ({\bf y}^{[n+1]})_{r+1:2(r+1)} 
        = - \sum_{k=1}^N \frac{\lambda^k \Delta t^{2k}}{\alpha_0} G_k {\bf u}_{n+1}
        = - \sum_{k=1}^N \frac{\lambda \Delta t^{2}}{\alpha_0} G_k Y_{k}.
    \end{equation}
    The result follows by grouping together 
    \eqref{eq:GLM:dG:5}, \eqref{eq:GLM:dG:4}, \eqref{eq:GLM:dG:6}, and \eqref{eq:GLM:dG:7}.
\end{proof}

An important remark comes from the proof of the previous theorem. 
From \eqref{eq:GLM:dG:4} we get $Y_k = \lambda^k \Delta t^{2k} Y_1$. Thus, recalling that $Y_1 = {\bf u}_{n+1}$ is an approximation to $u(t)$ at the interpolation nodes $t_n+c_i \Delta t$, we see that $$(Y_k)_i = \lambda^k \Delta t^{2k} (Y_1)_i \approx  \lambda^k \Delta t^{2k} u(t_n+c_i \Delta t) = \Delta t^{2k} u^{(2k)}(t_n+c_i \Delta t).$$
Therefore, one cannot consider the analysis approach reviewed in Section \ref{sec:review:GLM} (where $Y \approx u(t_{n+1})$; see \eqref{eq:GLM:output2}).
For this reason, we present an approach that is a little different from the one considered in the literature (see, e.g., \cite{DAmbrosio2012}), but that can be applied to our DG framework. The first step is to rewrite \eqref{eq:GLM} in the equivalent form
\begin{equation}\label{eq:GLM2}
\begin{split}
    Y^{n+1} + \Delta t^2 \lambda A Y^{n+1} &= U {\bf y}^{[n]},\\
{\bf y}^{[n+1]} &= G_{\rm GLM} {\bf y}^{[n]},
\end{split}
\end{equation}
where $G_{\rm GLM} := V - \Delta t^2 \lambda B [I+\Delta t^2 \lambda A]^{-1}U$ and we introduced the superscript $n+1$ in $Y^{n+1}$ to make the dependence of $Y$ on the time step explicit.
We consider that $Y^n$ and ${\bf y}^{[n]}$ are approximations to the solution vectors $\widehat{Y}^n$ and $\widehat{{\bf y}}^{[n]}$. Now, it is natural to define the truncation errors $T^n$ and ${\bm \theta}^{[n]}$ as the residuals of the two (block) equations in \eqref{eq:GLM2} when $Y^n$ and ${\bf y}^{[n]}$ are replaced by $\widehat{Y}^n$ and $\widehat{{\bf y}}^{[n]}$,
\begin{equation}\label{eq:GLM3}
\begin{split}
    \widehat{Y}^{n+1} + \Delta t^2 \lambda A \widehat{Y}^{n+1} &= U \widehat{{\bf y}}^{[n]} + T^n,\\
\widehat{{\bf y}}^{[n+1]} &= G_{\rm GLM} \widehat{{\bf y}}^{[n]} + {\bm \theta}^{[n]}.
\end{split}
\end{equation}
Now, neglecting the issue of a starting procedure (see \cite{Butcher2006,DAmbrosio2012}), namely $Y=\widehat{Y}^0$ and ${\bf y}^{[0]}=\widehat{{\bf y}}^{[0]}$, and introducing the errors
\begin{equation*}
    E^n := \widehat{Y}^{n} - Y^{n} \text{ and } {\bf e}^{[n]} := \widehat{{\bf y}}^{[n]} - {\bf y}^{[n]},
\end{equation*}
we obtain
\begin{equation}\label{eq:GLM4}
\begin{split}
    E^{n+1} + \Delta t^2 \lambda A E^{n+1} &= U {\bf e}^{[n]} + T^n,\\
{\bf e}^{[n+1]} &= G_{\rm GLM} {\bf e}^{[n]} + {\bm \theta}^{[n]},
\end{split}
\end{equation}
and the following result.
\begin{theorem}[convergence of GLM]\label{thm:GLM_DG}
Consider \eqref{eq:GLM4} for $n=0,\dots,N-1$. 
Let $G_{\rm GLM}=WDW^{-1}$ be the eigen-decomposition of $G_{\rm GLM}$, with $D$ denoting the diagonal matrix whose entries are the eigenvalues of $G_{\rm GLM}$, and assume that $\rho(G_{\rm GLM})\leq 1$. 
If ${\bf e}^{[0]}={\bf 0}$, $E^0 = {\bf 0}$, $W=O(\Delta t^q)$, and $W^{-1} \bm \theta_n = O(\Delta t^{p+1})$ and $T^n = O(\Delta t^{a+1})$, for some $a,p \in \mathbb{N}^+$, then ${\bf e}^{[n]} = O(\Delta t^{q+p})$ and $E^n = O(\Delta t^{\min(a+1,q+p)})$.
\end{theorem}
\begin{proof}
    By the Neumann series and a simple manipulation of the first equation in \eqref{eq:GLM4} one gets
\begin{equation*}
    E^{n+1} = [I + \Delta t^2 \lambda A ]^{-1} ( U {\bf e}^{[n]} + T^n)  = U {\bf e}^{[n]} + T^n + O(\Delta t^2).\\
\end{equation*}
Hence, if $T^n = O(\Delta t^{a+1})$ and ${\bf e}^{[n]} = O(\Delta t^{b+1})$, for some $a,b\in\mathbb{N}_+$, then $E^{n} = O(\Delta t^{\min(a+1,b+1)})$.
Now, the second equation in \eqref{eq:GLM4} has exactly the form of \eqref{eq:error}, and the result follows from Lemma \ref{lemma:general}. 
\end{proof}

\iffalse
Two points need to be further investigated to use Theorem~\ref{thm:GLM_DG}: the properties of $G_{\rm GLM}$ and the truncation errors $T^n$ and $\theta^n$.
By exploiting the structure of $G_{\rm GLM}$, that is
$$G_{\rm GLM} := V - \Delta t^2 \lambda B [I+\Delta t^2 \lambda A]^{-1}U,$$
we can obtain the following result.
\begin{theorem}[on the matrix $G_{\rm GLM}$]\label{thm:GLM_DG_2}
Consider the matrices $G_{\rm GLM}$, $V$, and $F = V - \Delta t^2 \lambda B U$. Assume that $F$ is diagonalizable, namely $F = W D W^{-1}$ for a diagonal $D$ and an invertible $W$.
If $\rho(V) \leq 1$, $\rho(F) < 1$, and $\kappa(W)_2 \leq C$ for a constant $C$ independent of $\Delta t$, then $\rho(G_{\rm GLM})<1$.
\end{theorem}
\begin{proof}
Recalling the definition of $G_{\rm GLM}$ and using the Neumann series on $[I+\Delta t^2 \lambda A]^{-1}$, we can write
$$G_{\rm GLM} = V - \Delta t^2 \lambda BU + \Delta t^4 M(\Delta t),$$
for some matrix $M(\Delta t)$. Notice that $M(\Delta t) \rightarrow \lambda^2 B A U$ for $\Delta t \rightarrow 0$.
Hence, there exists a constant $c^*$ (independent of $\Delta t$) such that $\|M(\Delta t)\|_2 \leq c^*$.
By the Bauer-Fike theorem, let $\mu$ be an eigenvalue of $G_{\rm GLM}$, there exist an eigenvalue $\gamma$ of $V - \Delta t^2 \lambda BU$ such that
$$|\mu-\gamma| \leq \Delta t^4 \| M(\Delta t) \|_2 \kappa_2(W) \leq c^* C \Delta t^4 ,$$
where we used the assumption on $\kappa_2(W)$. The result follows from the fact that $\rho(V) \leq 1$ and $\rho(F) < 1$.
\end{proof}
\fi

Let us now turn our attention to the truncation errors $T^n$ and $\theta^n$. For this purpose, we proceed as in Section \ref{sec:review:GLM}, consider the exact solution $\widehat{Y}_{k}^n = \Delta t^{(2k-2)} y^{(2k-2)}(t_n+ {\bf c} \Delta t)$ (namely, $\widehat{Y}_{k,i}^n = \Delta t^{2k-2} y^{(2k-2)}(t_n+c_i \Delta t)$) and assume that the exact $\widehat{{\bf y}}^{[n]}$ has the form
$$\widehat{{\bf y}}^{[n]} = \sum_{j=0}^\infty y^{(j)}(t_n) \Delta t^j {\bf q}_j,$$
for some vectors ${\bf q}_j$, $j=0,1,2,\dots$
Replacing them into the first line of \eqref{eq:GLM3},
and denoting by $\widehat{U} := \Bigl[\frac{1}{\alpha_0}G_0 \; \: I_{r+1} \Bigr]$, we can compute from the first block
\begin{equation}\label{eq:GLM:trunc:T}
\begin{split}
T_1^n &= \widehat{Y}_1^{n+1} + \Delta t^2 \lambda \sum_{k=1}^N \frac{\alpha_k}{\alpha_0} \widehat{Y}_k^{n+1} - \widehat{U} \widehat{{\bf y}}^{[n]} \\
    &= y(t_n+{\bf c} \Delta t) + \Delta t^2 \lambda \sum_{k=1}^N \frac{\alpha_k}{\alpha_0} \Delta t^{2k-2} y^{(2k-2)}(t_n+{\bf c} \Delta t) - \sum_{j=0}^\infty y^{(j)}(t_n) \Delta t^j \widehat{U} {\bf q}_j \\
    &= y(t_n+{\bf c} \Delta t) - \sum_{k=1}^N \frac{\alpha_k}{\alpha_0} \Delta t^{2k} y^{(2k)}(t_n+{\bf c} \Delta t) - \sum_{j=0}^\infty y^{(j)}(t_n) \Delta t^j \widehat{U} {\bf q}_j \\
    &= \sum_{j=0}^\infty y^{(j)}(t_n) \frac{\Delta t^j}{j!} {\bf c}^j - \sum_{k=1}^N \frac{\alpha_k}{\alpha_0} \Delta t^{2k} \sum_{j=0}^\infty y^{(2k+j)}(t_n) \frac{\Delta t^j}{j!} {\bf c}^j - \sum_{j=0}^\infty y^{(j)}(t_n) \Delta t^j \widehat{U} {\bf q}_j \\
    &= \sum_{j=0}^\infty y^{(j)}(t_n) \Delta t^j \Bigl[ \frac{1}{j!} {\bf c}^j - \widehat{U} {\bf q}_j \Bigr] - \sum_{j=0}^\infty \sum_{k=1}^N \frac{\alpha_k}{\alpha_0} y^{(2k+j)}(t_n) \frac{\Delta t^{j+2k}}{j!} {\bf c}^j,
\end{split}
\end{equation}
where ${\bf c}^j = [c_1^j,\dots,c_N^j]^\top$. Now, notice that
\begin{equation*}
    \sum_{j=0}^\infty \sum_{k=1}^N \frac{\alpha_k}{\alpha_0} y^{(2k+j)}(t_n) \frac{\Delta t^{j+2k}}{j!} {\bf c}^j
    =
    \sum_{j=2}^\infty \sum_{k=1}^{\lfloor j/2 \rfloor} \frac{\alpha_k}{\alpha_0} y^{(j)}(t_n) \frac{\Delta t^{j}}{(j-2k)!} {\bf c}^{j-2k},
\end{equation*}
which replaced into \eqref{eq:GLM:trunc:T} leads to
\begin{equation}\label{eq:GLM:trunc:T2}
    T_1^n
    =
    \sum_{j=0}^\infty y^{(j)}(t_n) \Delta t^j
    \Biggl[\frac{1}{j!}{\bf c}^{j} - \widehat{U} {\bf q}_j -  \frac{1}{\alpha_0} \sum_{k=1}^{\lfloor j/2 \rfloor} \frac{\alpha_k}{(j-2k)!}{\bf c}^{j-2k} \Biggr].
\end{equation}
Thus, the truncation error $T_1^n$ is of order $O(\Delta t^{a+1})$, for $a \in \mathbb{N}$, if
\begin{equation}\label{eq:GLM:trunc:T3}
    \widehat{U} {\bf q}_j + \frac{1}{\alpha_0} \sum_{k=1}^{\lfloor j/2 \rfloor} \frac{\alpha_k}{(j-2k)!}{\bf c}^{j-2k} = \frac{1}{j!}{\bf c}^{j},
    \; \text{ for $j=0,1,\dots,a$}.
\end{equation}
Notice that, by means of \eqref{eq:GLM:dG:4}, $T_1^n = O(\Delta t^{a+1})$ implies that $T_k^n = O(\Delta t^{a+1+2(k-1)})$ for $k=2,\dots,N$.

Now, replacing the exact solution $\widehat{{\bf y}}^{[n]}$ into the second block of \eqref{eq:GLM3}, recalling the definition of $G_{\rm GLM}$, and using Neumann series and Taylor expansion, we get
\begin{equation}\label{eq:trunc:theta}
\begin{split}
{\bm \theta}^{[n]} &= \widehat{{\bf y}}^{[n+1]} - G_{\rm GLM} \widehat{{\bf y}}^{[n]} \\
&= \sum_{j=0}^\infty y^{(j)}(t_{n+1}) \Delta t^j {\bf q}_j 
- \Biggl[ V - \Delta t^2 \lambda B [I+\Delta t^2 \lambda A]^{-1}U \Biggr]
\sum_{j=0}^\infty y^{(j)}(t_n) \Delta t^j {\bf q}_j \\
&= \sum_{j=0}^\infty y^{(j)}(t_{n+1}) \Delta t^j {\bf q}_j 
- \Biggl[ V + B \biggl[ \sum_{\ell=0}^\infty (-1)^\ell \Delta t^{2 \ell +2} \lambda^{\ell+1} A^k \biggr] U \Biggr]
\sum_{j=0}^\infty y^{(j)}(t_n) \Delta t^j {\bf q}_j \\
&= \sum_{j=0}^\infty \sum_{\ell=0}^\infty y^{(j+\ell)}(t_n) \frac{1}{\ell !}\Delta t^{j+\ell} {\bf q}_j 
- \sum_{j=0}^\infty y^{(j)}(t_n) \Delta t^j V {\bf q}_j \\
&\quad - \sum_{j=0}^\infty \sum_{\ell=0}^\infty (-1)^\ell \Delta t^{j+2 \ell +2}
y^{(j+2\ell +2)}(t_n) B A^k U{\bf q}_j \\
&=\sum_{j=0}^\infty \sum_{\ell=0}^j y^{(j)}(t_n) \frac{1}{\ell !}\Delta t^{j} {\bf q}_{j-\ell} 
- \sum_{j=0}^\infty y^{(j)}(t_n) \Delta t^j V {\bf q}_j \\
&\quad - \sum_{j=2}^\infty \sum_{\ell=0}^{\lfloor \frac{j-2}{2} \rfloor} (-1)^\ell \Delta t^{j}
y^{(j)}(t_n) B A^\ell U{\bf q}_{j-2\ell-2} \\
&= y(t_n) \bigl[ {\bf q}_0 - V {\bf q}_0 \bigr] 
+ y^{(1)}(t_n) \Delta t \bigl[ {\bf q}_0 + {\bf q}_1 - V {\bf q}_1 \bigr] \\
&\quad +\sum_{j=2}^\infty y^{(j)}(t_n) \Delta t^j \biggl[ 
\sum_{\ell=0}^j \frac{1}{\ell !} {\bf q}_{j-\ell} 
- \sum_{\ell=0}^{\lfloor \frac{j-2}{2} \rfloor} (-1)^\ell B A^\ell U{\bf q}_{j-2\ell-2}
- V {\bf q}_j
\biggr].
\end{split}
\end{equation}
Thus, the truncation error ${\bm \theta}^{[n]}$ is of order $O(\Delta t^{p+1})$, for $p \in \mathbb{N}$, if
\begin{equation}\label{eq:GLM:trunc:theta2}
\begin{split}
   V {\bf q}_0 &= {\bf q}_0, \\
   V {\bf q}_1 &= {\bf q}_0 + {\bf q}_1, \\
   V {\bf q}_j &= \sum_{\ell=0}^j \frac{1}{\ell !} {\bf q}_{j-\ell} 
-\sum_{\ell=0}^{\lfloor \frac{j-2}{2} \rfloor} (-1)^\ell B A^\ell U{\bf q}_{j-2\ell-2}, \; \text{ for $j=2,\dots,p$}.
\end{split}
\end{equation}
We summarize our findings in the following theorem.
%\marginpar{\MG{The inconsistency with Hairer and Wanner leads now to a method with two orders, reformulate?}}
\begin{theorem}[Consistency of DG2 as GLM - order conditions]
    Let ${\bf q}_{j}$, $j=0,1,\dots$, be vectors such that the order conditions \eqref{eq:GLM:trunc:T3} and \eqref{eq:GLM:trunc:theta2} hold for some $p \in \mathbb{N}$ and $a \in \mathbb{N}$, respectively.
    Then, a DG method in the form \eqref{eq:GLM} satisfies $T^n = O(\Delta t^{a+1})$ and ${\bm \theta}^{[n]} = O(\Delta t^{p+1})$.
\end{theorem}

In light of Theorem \ref{thm:GLM_DG}, the above result may not be sufficient to get an accurate estimate of the order of the method, since the role of $W^{-1}$ applied to ${\bm \theta}^{[n]}$ may have a non-negligible impact.
To estimate it, we still rely on the order conditions \eqref{eq:GLM:trunc:theta2}. To do so, we first consider an expansion $W^{-1} = \sum_{j=0}^\infty \Delta t^j W_j$, for some matrices $W_j$. Combining this expansion with \eqref{eq:GLM:trunc:theta2}, the relations 
\begin{equation}\label{eq:GLM:trunc:theta3}
\begin{split}
   W_{\ell} V {\bf q}_0 &= W_{\ell}{\bf q}_0, \\
   W_{\ell}V {\bf q}_1 &= W_{\ell}({\bf q}_0 + {\bf q}_1), \\
   W_{\ell}V {\bf q}_j &= W_{\ell}\biggl[\sum_{\ell=0}^j \frac{1}{\ell !} {\bf q}_{j-\ell} 
-\sum_{\ell=0}^{\lfloor \frac{j-2}{2} \rfloor} (-1)^\ell B A^\ell U{\bf q}_{j-2\ell-2} \biggr], \; \text{ for $j=2,\dots,p_\ell$},
\end{split}
\end{equation}
arise naturally and the following result can be proved by a direct estimate of the term $W^{-1}{\bm \theta}^{[n]}$.
\begin{theorem}[Consistency of DG2 as GLM - more order conditions]\label{thm:more}
%A DG method (written in the form \eqref{eq:GLM}) is $W$-consistent of orders $a$ and $p$, namely $T^n = O(\Delta t^{a+1})$ and $W^{-1}{\bm \theta}^{[n]} = O(\Delta t^{p+1})$ if there exist vectors ${\bf q}_{j}$, $j=0,1,\dots$, such that \eqref{eq:GLM:trunc:T3} holds and \eqref{eq:GLM:trunc:theta3} holds for.
    Assume that there exist vectors ${\bf q}_\ell$, $\ell=1,\dots,p$, such that \eqref{eq:GLM:trunc:theta2} holds.
    Further, assume that \eqref{eq:GLM:trunc:theta3} holds for some $p_\ell \in \mathbb{N}$, for all $\ell=0,1,\dots$
    Then, $W^{-1}{\bm \theta}^{[n]} = O(\Delta t^{p_{\widetilde{\ell}}+\widetilde{\ell}+1})$ with $\widetilde{\ell} = \arg \min_{\ell} p_{\ell}$. Clearly, $p_{\widetilde{\ell}}= \min_{\ell} p_{\ell} \geq p$.
\end{theorem}
We now apply the theory developed in this section to the case $\mathbb{P}^1$. The reader may appreciate how complicated the analysis can become in the GLM formalisms.
For this reason, here we do not report the cases $\mathbb{P}^2$ and $\mathbb{P}^3$.

Let us recall the matrix $G=-A_+^{-1}A_-$ (for $s=\Delta t^2/2$) computed in \eqref{eq:matrixG_P1} and notice that
$d = -2 - \lambda\Delta t^2/2$, with $\alpha_0=-2$ and $\alpha_1 = -1/2$, and 
$G_0 = \begin{small} \begin{bmatrix}
0 & -2 \\
2 & -4 \\
\end{bmatrix} \end{small}$
and
$G_1 = \begin{small} \begin{bmatrix}
0 & 1/2 \\
-1/2 & -1 \\
\end{bmatrix} \end{small}$.
Thus, the GLM matrices are
\begin{equation*}
\begin{split}
    A &= \frac{\alpha_1}{\alpha_0} I_2, \quad U=\begin{bmatrix}
        \frac{1}{\alpha_0 G_0} & I_2 \\
    \end{bmatrix},  \quad
    B = \begin{bmatrix}
        \frac{\alpha_1}{\alpha_0} I_2 \\
        \frac{1}{\alpha_0} G_1
    \end{bmatrix}, \quad V= \begin{bmatrix}
        \frac{1}{\alpha_0}G_0 & I_2 \\
        0 & 0 \\
    \end{bmatrix}.\\
\end{split}    
\end{equation*}
Now, direct calculations reveal that
$$
G_{\rm GLM} = \begin{bmatrix}
0&        1-\frac{\lambda \Delta t^2}{4(\lambda \Delta t^2/4 + 1)}&    1-\frac{\lambda \Delta t^2}{4(\lambda \Delta t^2/4 + 1)}&    0\\
\frac{\lambda \Delta t^2}{4(\lambda \Delta t^2/4 + 1)}-1&        2-\frac{\lambda \Delta t^2}{2(\lambda \Delta t^2/4 + 1)}&   0&   1-\frac{\lambda \Delta t^2}{4(\lambda \Delta t^2/4 + 1)}\\
-\frac{\lambda \Delta t^2}{4(\lambda \Delta t^2/4 + 1)}&        \frac{\lambda \Delta t^2}{2(\lambda \Delta t^2/4 + 1)}&    0&    -\frac{\lambda \Delta t^2}{4(\lambda \Delta t^2/4 + 1)}\\
\frac{\lambda \Delta t^2}{2(\lambda \Delta t^2/4 + 1)} & -\frac{5\lambda \Delta t^2}{4(\lambda \Delta t^2/4 + 1)} & -\frac{\lambda \Delta t^2}{4(\lambda \Delta t^2/4 + 1)} & -\frac{\lambda \Delta t^2}{2(\lambda \Delta t^2/4 + 1)}
\end{bmatrix},
$$
and that $G_{\rm GLM} =W^{-1} D W$ with
$$
D ={\rm diag}(0,0,\gamma_1,\gamma_2),
$$
where 
$\gamma_{1}=\frac{(2+i\Delta t \sqrt{\lambda})^2}{\Delta t^2 \lambda+4}$ and $\gamma_{2}=-\frac{(2i+\Delta t \sqrt{\lambda})^2}{\Delta t^2 \lambda+4}$,
and
$$
W = \begin{bmatrix}
-2& 1&  \frac{4(\Delta t \sqrt{\lambda} + 2i)}{\Delta t^2 \lambda(\Delta t \sqrt{\lambda} - 6i)}&  \frac{4(\Delta t \sqrt{\lambda} - 2i)}{\Delta t^2 \lambda(\Delta t \sqrt{\lambda} + 6i)} \\
-1& 0& -\frac{4(\Delta t \sqrt{\lambda} - 2i)}{\Delta t^2 \lambda(\Delta t \sqrt{\lambda} - 6i)}& -\frac{4(\Delta t \sqrt{\lambda} + 2i)}{\Delta t^2 \lambda(\Delta t \sqrt{\lambda} + 6i)} \\ 
 1& 0&        -\frac{\Delta t^2 \lambda + 12 + \Delta t \sqrt{\lambda}4i}{\Delta t^2 \lambda + 36}&         -\frac{\Delta t \sqrt{\lambda} + 2i}{\Delta t \sqrt{\lambda} + 6i} \\
 0& 1&                                      1&                                      1 \\
\end{bmatrix}.
$$
We see that $W=O(\Delta t^{-2})$ and $\rho(G_{\rm GLM}) = 1$.
Moreover, it is possible to compute that
$W^{-1} = W_0 + \Delta t W_1 + O(\Delta t^2)$, where
$$
W_0 = \begin{bmatrix}
    0 & 0 & 1 & 0 \\
    0 & 0 & 0 & 1 \\
    0 & 0 & 0 & 0 \\
    0 & 0 & 0 & 0 \\
\end{bmatrix}, \quad
W_1 = \frac{3i}{8}\begin{bmatrix}
    0 & 0 & 0 & 0 \\
    0 & 0 & 0 & 0 \\
    -1 & 1 & -1 & 1 \\
    1 & -1 & 1 & -1 \\
\end{bmatrix}. 
%\quad
%W_2 = \begin{bmatrix}
%0& -1/4& -1/4&    0\\
%1/4&  1/2&    1& -1/4\\
%-1/8& -1/4& -1/2&  1/8\\
%-1/8& -1/4& -1/2&  1/8\\
%\end{bmatrix}.
$$
Now, direct calculations reveal that the vectors
\begin{equation*}
    {\bf q}_0 = \begin{bmatrix}
        1 \\ 1 \\ 0 \\ 0 \\
    \end{bmatrix}, \quad
    {\bf q}_1 = \begin{bmatrix}
        -1 \\ 0 \\ 0 \\ 0 \\
    \end{bmatrix}, \quad
    {\bf q}_2 = \begin{bmatrix}
        1/2 \\ 0 \\ -1/4 \\ 3/4 \\
    \end{bmatrix}, 
\end{equation*}
satisfy the order conditions \eqref{eq:GLM:trunc:T3} and \eqref{eq:GLM:trunc:theta2} for $a=p=2$, thus ${\bm \theta}^{[n]} = O(\Delta t^3)$ and $T_n = O(\Delta t^3)$.
However, it is not possible to find a vector ${\bf q}_3$ satisfying \eqref{eq:GLM:trunc:T3} and \eqref{eq:GLM:trunc:theta2} for $a=p=3$. We thus consider the effect of $W^{-1}$.
Since ${\bf q}_0$, ${\bf q}_1$, and ${\bf q}_2$ satisfy \eqref{eq:GLM:trunc:theta2} for $a=p=2$, they will clearly satisfy \eqref{eq:GLM:trunc:theta3} for all $\ell=0,1,2,\dots$ and $p_\ell = 2$, implying that $W^{-1} {\bm \theta}^{[n]}$ has at least order $O(\Delta t^2)$.
Now, one can verify that the vectors 
\begin{equation*}
    {\bf q}_3 = \begin{bmatrix}
        x \\ y \\ 0 \\ -1/4 \\
    \end{bmatrix}, \quad
    {\bf q}_4 = \begin{bmatrix}
        x_0 \\ y_0 \\ 1/8 \\ -1/4 \\
    \end{bmatrix}, \quad
    {\bf q}_5 = \begin{bmatrix}
        x_1 \\ y_1 \\ \frac{1}{24} + \frac{x}{4}-\frac{y}{2}\\ \frac{5y}{4} - \frac{x}{2} \\
    \end{bmatrix}, 
\end{equation*}
together with ${\bf q}_0$, ${\bf q}_1$, and ${\bf q}_2$,
satisfy the conditions \eqref{eq:GLM:trunc:theta3} for 
$\ell = 0,1$ with $p_0 = 5$ and $p_1=3$ (for any values of $x,y,x_0,y_0,x_1,y_1$). Thus, we obtained that $\widetilde{\ell} = \arg \min_{\ell} p_{\ell} = 1$, with
$p_{\widetilde{\ell}}=p_1=3$, and hence
$W^{-1} {\bm \theta}^{[n]} = O(\Delta t^{5})$ by Theorem \ref{thm:more}.
Hence, recalling that $T_n = O(\Delta t^3)$ and $W = O(\Delta t^{-2})$, Theorem \ref{thm:GLM_DG} implies that ${\bf e}^{[n]} = O(\Delta t^{2})$ and $E^n = O(\Delta t^{2})$. 
Notice that by appropriate choices of $x$ and $y$,  ${\bf q}_3$ can satisfy \eqref{eq:GLM:trunc:T3} for $a=3$. However, by Theorem \ref{thm:GLM_DG}, this does not improve the order of $E^n$. 
We thus conclude that DG2 with $\mathbb{P}^1$ and $s=\Delta t^2/2$ is second order convergent. Moreover, the reader can verify also in this GLM framework that $s \neq \Delta t^2/2$ does not lead to a second-order convergent method, in agreement with the results of Section \ref{sec:P1}.

Finally, we remark that, while the equivalence with GLM is elegant and puts DG2 very clearly into the context of numerical methods for differential equations, it does not lead to a formalism that is simpler to analyze.

%%%%%%%%%%%%%%%%%%%%%%%%%%%%%%%%%%%%%%%%
\section{Time DG methods for first-order systems: DG1}\label{sec:dG:first}
%%%%%%%%%%%%%%%%%%%%%%%%%%%%%%%%%%%%%%%%
For $T>0$, we rewrite \eqref{eq:modelProblem} as a system of first-order differential equations,
\begin{equation}
	\label{eq:modelProblem_system}
	\begin{cases}
            \dot{u}(t) - w(t) = 0 & \forall\, t \in (0,T], \\
		\dot{w}(t) + \lambda u(t) = 0 & \forall\, t \in (0,T],\\
		u(0) = \widehat{u}_0, \\
		w(0) = \widehat{u}_1,
	\end{cases}
\end{equation}
or in compact form as 
\begin{equation}\label{eq:modelP_system_matrix}
\begin{cases}
     \dot{\bm z}(t) = L \bm z(t), & \forall\, t \in (0,T], \\
    \bm z(0) = \widehat{\bm z}_0,
\end{cases}    
\end{equation}
where $\bm z(t) := [ u(t), w(t)]^\top$,  $L := \begin{bmatrix}
        0 & 1 \\ -\lambda & 0 
    \end{bmatrix}$,
and  $\widehat{\bm z}_0 := [ \widehat{u}_0, \widehat{u}_1]^\top$. In order to construct a DG method for \eqref{eq:modelP_system_matrix},
we focus as for the second-order case in Section~\ref{sec:dG:second} on one time slab $I_n$. 
Multiplying  \eqref{eq:modelP_system_matrix} by a (regular enough) test function $\bm v(t)$ and integrating in time over $I_n$ we obtain
\begin{equation}
	\label{Eq:Weak1}
	(\dot{\bm z},\bm v)_{I_n} - (L\bm z,\bm v)_{I_n}  = 0.
\end{equation}
Noting that $\bm z \in \bm H^1(0,T)$, since $u \in H^2(0,T)$ and $w = \dot{u}$, we add to \eqref{Eq:Weak1} the vanishing term $[\bm z]_{n-1}\cdot\bm v(t_{n-1}^+)$ and get
\begin{equation}
	\label{Eq:Weak2}
	(\dot{\bm z},\bm v)_{I_n} - (L\bm z,\bm v)_{I_n} + [\bm z]_{n-1}\cdot\bm v_{n-1}^+ = 0.
\end{equation}
Summing over all time slabs we obtain the problem: find $\bm z_{DG} \in [\mathbb{V}_{DG}]^2$ such that 
\begin{equation}
	\label{Eq:BilinearForm}
	\mathcal{B}(\bm z_{DG},\bm v) = \mathcal{G}(\bm v) \quad \forall \, \bm v \in [\mathbb{V}_{DG}]^2,  
\end{equation}
where $\mathcal{B}(\cdot,\cdot): [\mathbb{V}_{DG}]^2 \times [\mathbb{V}_{DG}]^2 \rightarrow \mathbb{R}$ is defined by 
\begin{align}
 \mathcal{B}(\bm z,\bm v) := \sum_{n=1}^N \big( (\dot{\bm z},\bm v)_{I_n} - (L\bm z,\bm v)_{I_n} \big) + \sum_{n=1}^{N-1} [\bm z]_n\cdot\bm v(t_n^+) +  \bm z(0^+)\cdot\bm v_0^+,  
\end{align}
and the linear functional $\mathcal{G}(\cdot): [\mathbb{V}_{DG}]^2 \rightarrow \mathbb{R}$ is 
\begin{equation}\label{Eq:LinearFunctional}
	\mathcal{G}(\bm v) := \widehat{\bm z}_0\cdot\bm z_{0}^+.
\end{equation}
Note that the bilinear form $\mathcal{B}(\cdot, \cdot)$ is well defined whenever its arguments are, at least, $\bm H^1(I_n)$ functions for any $n = 1, \dots, N$. The existence and uniqueness of the discrete solution and stability bounds in a suitable mesh-dependent norm can be found in \cite{AntoniettiMiglioriniMazzieri2021}.

\subsection{Algebraic formulation for DG discretization in first-order form: DG1}\label{sec:algebraic_1st}

As in Section~\ref{sec:algebraic} we focus 
on a generic time slab $I_{n}$, where a local polynomial degree $r$ is used, and we fix a basis $\{[\psi_{n,j}(t),0]^T,[0,\psi_{n,j}(t)]^T\}_{j=1,\dots,r}$ for the polynomial space $[\mathbb{P}^r(I_{n})]^2$, $2(r+1)$ being the dimension of the local finite-dimensional space $[\mathbb{V}_{n}^r]^2$. We write the trial function $\bm z$ as a linear combination of the basis functions, i.e.,
\begin{align*}
\bm z(t) = 
\sum \limits_{j=1}^{r+1} u_{n,j} \begin{bmatrix}
    \psi_{n,j}(t) \\ 0
\end{bmatrix}
+ \sum \limits_{j=1}^{r+1} w_{n,j} \begin{bmatrix}
     0 \\ \psi_{n,j}(t)
\end{bmatrix}, \text{\qquad for $t \in I_{n}$}. 
\end{align*} 
By choosing $\bm v$ in 
\eqref{Eq:Weak2} as the $i-$th basis function we obtain the system

\begin{align}\label{eq:simple_local_system_1st}
B \widehat{\bm z} = \bm c,
\end{align}
where
\begin{equation*}
    B := \begin{bmatrix}
        I_{2(r+1)} \\
        B_- & B_+ \\
        & B_- & B_+ \\
        & & \ddots & \ddots \\
        & & & B_- & B_+ \\
    \end{bmatrix},
    \quad
    \widehat{\bm z} := \begin{bmatrix}
        \widehat{\bm z}_0 \\
        \widehat{\bm z}_1 \\
        \widehat{\bm z}_2 \\
        \vdots \\
        \widehat{\bm z}_N \\
    \end{bmatrix},
    \quad
    \bm c := \begin{bmatrix}
        \bm c_0 \\
        \bm 0 \\
        \bm 0 \\
        \vdots \\
        \bm 0 \\
    \end{bmatrix}.
\end{equation*}
Here, $\widehat{\bm z}_n \in \mathbb{R}^{2(r+1)(N+1)}$ is the vector containing the coefficients $$\widehat{\bm z}_n := [
u_{n,1},w_{n,1},\dots,u_{n,r+1},w_{n,r+1}]^\top,  \quad n=1,\dots,N,$$ 
and $\widehat{\bm z}_0$ and $\bm c_0$ contain the initial conditions, and the local matrices $B_-$ and $B_+$ are given by
 \begin{align}
     B_{-} & := N^0 \otimes I_2, \label{eq:b_m}\\
     B_{+} & := (N^1 + N^3) \otimes I_2 - N^2 \otimes L, \label{eq:b_p} 
 \end{align} 
 where  $I_2 \in \mathbb{R}^2$ is the identity matrix and
 \begin{align*}
     N^0_{ij}  & :=   - \langle \psi_{n,j}, {\psi}_{n+1,i} \rangle_{t_n}, \quad  N^1_{ij} := (\dot{ \psi}_{n+1,j},{\psi}_{n+1,i})_{I_n}, \\
      N^2_{ij} & := (\psi_{n+1,j}, {\psi}_{n+1,i} )_{I_n}, \quad N^3_{ij}  :=   \langle \psi_{n+1,j}, \psi_{n+1,i} \rangle_{t_n},
 \end{align*}
for $i,j=1,\dots,r+1$. 
Clearly, \eqref{eq:simple_local_system_1st} can be written as
\begin{equation}\label{eq:simple_local_system2_1st}
    B_+ \widehat{\bm z}_{n+1} + B_- \widehat{\bm z}_{n} = \bm 0 \, \text{ for $n=0,1,\dots,N-1$},
\end{equation}
however, to derive the analogy presented in the following section, we rewrite  \eqref{eq:simple_local_system2_1st} as
\begin{equation}\label{eq:simple_local_system2_1st_rw}
    \widehat{B}_+ \bm z_{n+1} = - \widehat{B}_- \bm z_{n} \, \text{ for $n=0,1,\dots,N-1$},
\end{equation}
or equivalently as 
\begin{equation}\label{eq:dg_kroneker}
(I_{r+1} \otimes I_2 - N^4 \otimes L ) \bm z_{n+1} = - (N^5 \otimes I_2) \bm z_n, 
\end{equation}
where  $N^4 = (N^1 +N^3)^{-1} N^2$, $N^5 = (N^1 +N^3)^{-1} N^0$ and $I_{r+1}$ is the $(r+1)\times (r+1)$ identity matrix. Note that the system matrix in \eqref{eq:dg_kroneker} is invertible because of the well-posedness of problem \eqref{Eq:BilinearForm}. 
To see this directly at the algebraic level one has to prove that ${\rm det}(I_{r+1} \otimes I_2 - N^4 \otimes L ) = {\rm det}(I_{r+1}+\lambda (N^4)^2) \neq 0$.
Thus, one needs to study the eigenvalues of $I_{r+1}+\lambda (N^4)^2$, which are equal to $1+\lambda (\sigma(N^4))^2$, where $\sigma(N^4)$ are the eigenvalues of $N^4$.
In general, since the eigenvalues of $N^4$ are complex, the invertibility depends on their specific values. In Sections \ref{sec:P1_1st}, \ref{sec:P2_1st} and \ref{sec:P3_1st}, we will show that for polynomial degree $r=1,2,3$ the system matrix \eqref{eq:dg_kroneker} is invertible and in particular that $N_4 = \Delta t A$, cf. \eqref{eq:rk_kroneker},  with $A$ defined by the specific Butcher tableau \eqref{eq:Butcher_tableau}.

\subsection{DG1 as Implicit Runge-Kutta scheme}\label{sec:implicit_RK}
In this section, we start presenting the general analogy between the DG discretization \eqref{Eq:Weak2} and Implicit Runge Kutta (IRK) schemes by adapting the original proof in \cite{lesaint1974finite} to problem 
\eqref{eq:modelProblem_system}.
Next, we give an equivalent result by considering the matrix formulation \eqref{eq:dg_kroneker}.

We consider the problem \eqref{Eq:Weak2}, i.e., for any $I_n$, $n=1,...,N$, find $\bm z \in [\mathbb{P}^r(I_n)]^2$ such that
\begin{equation}
	\label{eq:DG1}
	(\dot{\bm z} - L\bm z, \bm v)_{I_n} + [\bm z]_{n-1}\cdot {\bm v}_{n-1}^+ = 0, \quad \forall \bm v \in [\mathbb{P}^r(I_n)]^2,
\end{equation}
with ${\bm z}_0^- = \widehat{\bm z}_0$,
and replace the first integral in  \eqref{eq:DG1} by an interpolatory quadrature formula,
\begin{equation}\label{eq:quad_formula}
\int_{t_{n-1}}^{t_n} (\dot{\bm z} - L\bm z) \, ds = \Delta t \sum_{i=1}^{r+1} b_i (\dot{\bm z} - L\bm z)(t_{n-1}^i) + \mathcal{O}(\Delta t^{p+1}), \quad r+1 \leq p \leq 2r+1,
\end{equation}
where $t_{n-1}^i = t_{n-1} + c_i \Delta t$, for $1\leq i \leq r+1$, $c_1=0$, and where $b_i$ and $c_i$ are weigths and nodes in $[0,1]$.
Using \eqref{eq:quad_formula} we can rewrite \eqref{eq:DG1}
for all $I_n$, $n=1,...,N$: we look for $\bm z \in [\mathbb{P}^r(I_n)]^2$ that satisfies 
\begin{equation}\label{eq:dg_1st}
\Delta t \sum_{i=1}^{r+1} b_i  
(\dot{\bm z} - L\bm z)(t_{n-1}^i)\cdot \bm v(t_{n-1}^i) + [\bm z]_{n-1}\cdot {\bm v}_{n-1}^+ = 0 \quad \forall \bm v \in [\mathbb{P}^r(I_n)]^2.
\end{equation}
Next, we define
\begin{equation*}
\begin{cases}
\bm z_{n-1} := \bm z(t_{n-1}^-) & \\
\bm z_{n-1}^1  := \bm z(t_{n-1}^+) = \bm z(t_{n-1}^1), & \\ 
\bm z_{n-1}^i  := \bm z(t_{n-1}^i), & 2 \leq i \leq r+1,
\end{cases}
\end{equation*}
and introduce the Lagrange polynomials
\begin{equation*}
\ell_i(t) := \prod_{j=2,j\neq i}^{r+1} \dfrac{t-c_j}{c_i-c_j}, \quad 2 \leq i \leq r+1.
\end{equation*}
Following \cite{lesaint1974finite} one can prove that \textit{scheme \eqref{eq:dg_1st} is equivalent to the implicit Runge-Kutta method}
\begin{equation}\label{eq:general_RK_dgmethod}
\begin{cases}
\bm z_{n-1}^i & = \bm z_{n-1} + \Delta t \sum_{j=1}^{r+1} a_{ij} L \bm z_{n-1}^j \quad 1 \leq i \leq r+1, \\
\bm z_{n} & = \bm z_{n-1} + \Delta t \sum_{j=1}^{r+1} b_{j} L \bm z_{n-1}^j,
\end{cases}
\end{equation}
where $a_{i1} = b_1$  and $a_{ij} = \int_{0}^{c_i} \ell_j(s) ds - b_1 \ell_j(c_1)$, for  $1\leq i \leq r+1$, $2\leq j\leq r+1$.
To do so, we consider the basis functions $\{ \bm \psi_i \}_{i=1}^{r+1} = \{[\psi_{i}(t),0]^T,[0,\psi_{i}(t)]^T\}$ for $i=1,\dots,r+1$ for $[\mathbb{P}^r(I_n)]^2$ such that $\psi_{i}(t_{n-1}^j) = \delta_{ij}$ for $1\leq i,j \leq r+1$. 
Next, we replace $\bm v$ by $\bm \psi_i$ in \eqref{eq:dg_1st} to get
\begin{equation}\label{eq:dg_1st_1step}
\begin{cases}
 \bm z(t_{n-1}^+) - \bm z(t_{n-1}^-) + \Delta t b_1 \Big( \dot{\bm z}(t_{n-1}^1)-L\bm z(t_{n-1}^1)\Big) = \bm 0, & \\ \dot{\bm z}(t_{n-1}^i)-L\bm z(t_{n-1}^i) = \bm 0, & \quad 2 \leq i \leq r+1,
 \end{cases}  
\end{equation}
and use the second equation above to obtain
\begin{equation*}
\dot{\bm z}_h(t) = \sum_{j=2}^{r+1} \ell_j\left(\frac{t-t_{n-1}}{\Delta t}\right)\dot{\bm z}_h(t_{n-1}^i)  =  \sum_{j=2}^{r+1} \ell_j\left(\frac{t-t_{n-1}}{\Delta t}\right) L \bm z(t_{n-1}^j).
\end{equation*}
Taking $t = t_{n-1} = t_{n-1}^1$ and using the first equation in \eqref{eq:dg_1st_1step} yields 
\begin{equation}\label{eq:first _equationRK}
\bm  z_{n-1}^1 = \bm z_{n-1} + \Delta t b_1 \Big(
L \bm z_{n-1}^1 - \sum_{j=2}^{r+1} \ell_j(c_1) L \bm z_{n-1}^j \Big).
\end{equation}
On the other hand, for $2\leq i \leq r+1$, we have
\begin{equation*}
\bm z_{n-1}^i = \bm z_{n-1}^1 + \int_{t_{n-1}^1}^{t_{n-1}^i} \dot{\bm z}_h(s) \, ds,
\end{equation*}
and then 
\begin{equation*}
\bm z_{n-1}^i = \bm z_{n-1}^1 + \Delta t \left( b_1 L \bm z_{n-1}^1 + \sum_{j=2}^{r+1} \left( \int_0^{c_i} \ell_j(s) \, ds - b_1\ell_j(c_1)\right)L \bm z_{n-1}^j \right).
\end{equation*}
Similarly for $t_{n}^-$  we get 
\begin{equation*}
\bm z_{n} = \bm z_{n-1} + \Delta t \left( b_1 L \bm z_{n-1}^1 + \sum_{j=2}^{r+1} \left( \int_0^{1} \ell_j(s) \, ds - b_1\ell_j(c_1)\right)L \bm z_{n-1}^j \right),
\end{equation*}
and noticing that $\int_0^1 \ell_j(s) \, ds = b_1 \ell(\xi_1) + b_j$ we obtain
\begin{equation}\label{eq:second_equationRK}
\bm z_{n} = \bm z_{n-1} + \Delta t \sum_{j=1}^{r+1}  b_j L \bm z_{n-1}^j.
\end{equation}
Equations \eqref{eq:first _equationRK}-\eqref{eq:second_equationRK} are identical to the equations in \eqref{eq:general_RK_dgmethod}, which means that the discontinuous Galerkin method leads to the one-step method  \eqref{eq:general_RK_dgmethod}. On the other hand, system \eqref{eq:general_RK_dgmethod} can be viewed as a discontinuous Galerkin method.

Next, with a different point of view,  we prove that schemes of the form \eqref{eq:simple_local_system2_1st_rw} are IRK methods with $(r + 1)$-stages, cf. \cite[Remark 1]{Southworth2021}.
Moreover, for polynomial degree $r=1,2,3$, we show that the DG scheme is equivalent to an IRK Lobatto-IIIC method, by computing the entries of the time matrices $N_4$ and $N_5$ through a Gauss-Legendre-Lobatto (GLL) quadrature formula having $r+1$ points and weights, and consider the basis functions $\{\psi_\ell\}_{\ell=1}^{r+1}$ as the characteristic polynomials associated with these points.
We show, under specific assumptions, that the Runge-Kutta method \eqref{eq:RK_extended_0} (and \eqref{eq:rk_kroneker}) is equivalent to the DG scheme \eqref{eq:dg_kroneker}. 

\begin{theorem}[equivalence between DG and RK]\label{thm:equiv}
    Assume that $s=r+1$, $N^4=\Delta t A$, $b_i = a_{s,i}$ for $i=1,\dots,s$, that the matrix $N^5 \in \mathbb{R}^{(r+1)\times (r+1)}$ is zero up to the last column having all entries equal to $-1$, and that the matrix $(I_s \otimes I_2 - N^4 \otimes L)$ is invertible, cf. \eqref{eq:dg_kroneker}. 
    If $[\widehat{\bm z}_0]_{2r+1,2r+2} = \bm z_0$, then 
    \begin{equation}\label{eq:equivalence}
        [\widehat{\bm z}_{n}]_{2r+1,2r+2} = \bm z_{n}
        \quad \text{and} \quad
        \widehat{\bm z}_n = (I_{r+1} \otimes L^{-1}) \bm k_{n-1}
        \quad \text{for $n=1,\dots,N$}.
    \end{equation}
\end{theorem}
\begin{proof}
    The proof is by induction. 
    Since the relation $[\widehat{\bm z}_0]_{2r+1,2r+2} = \bm z_0$ holds by assumption, we assume that \eqref{eq:equivalence} holds for $n$ and we prove it for $n+1$.
    To prove the second equation in \eqref{eq:equivalence}, a direct calculation using the structure of $N^5$ and the induction hypothesis $[\widehat{\bm z}_{n}]_{2r+1,2r+2} = \bm z_{n}$ allow us to obtain that $\bm 1 \otimes  \bm z_n = - (N^5 \otimes I_2) \widehat{\bm z}_n$, where $\bm 1 \in \mathbb{R}^{s}$ is a vector with all components equal to $1$.
    Now, noticing that the matrices $I_{s} \otimes L^{-1}$ and $(I_s \otimes I_2 - N^4\otimes L)$ commute\footnote{This follows using the property of the Kronecker product $(A\otimes B)(C\otimes D) = (AC) \otimes (BD)$.}, and using \eqref{eq:rk_kroneker}, we can compute
    \begin{equation*}
    \begin{split}
        - (N^5 \otimes I_2) \widehat{\bm z}_n 
        &= \bm 1 \otimes \bm z_n 
        = (I_s \otimes L^{-1}) (\bm 1 \otimes L \bm z_n) 
        = (I_s \otimes L^{-1}) \bm f\\
        &= (I_s \otimes L^{-1}) (I_s \otimes I_2 - \Delta t A \otimes L) \bm k_n \\
        &= (I_s \otimes L^{-1}) (I_s \otimes I_2 - N^4 \otimes L) \bm k_n \\
        &= (I_s \otimes I_2 - N^4 \otimes L) (I_s \otimes L^{-1}) \bm k_n,
    \end{split}
    \end{equation*}
    which shows that the vector $(I_s \times L^{-1}) \bm k_n$ satisfies \eqref{eq:dg_kroneker}. Thus, by the invertibility of $(I_s \otimes I_2 - N^4 \otimes L)$  it follows that $\widehat{\bm z}_{n+1} = (I_s \otimes L^{-1}) \bm k_n$.
    To prove the first equation in \eqref{eq:equivalence},
    since $b_i=a_{s,i}$ for all $i$, we obtain from \eqref{eq:RK_extended_0} that
    \begin{equation*}
        \bm k_{s,n} 
        = L\Bigl(\bm z_n + \Delta t \sum_{j=1}^s a_{s,j} \bm k_{j,n}\Bigr) 
        = L\Bigl(\bm z_n + \Delta t \sum_{j=1}^s b_{j} \bm k_{j,n}\Bigr) 
        = L \bm z_{n+1}.
    \end{equation*}
    Using this relation, we can compute
    \begin{equation*}
        \widehat{\bm z}_{n+1} = (I_{s} \otimes L^{-1}) \bm k_n
        = \begin{bmatrix}
            L^{-1} \\
            & \ddots \\
            & & L^{-1} \\
            & & & L^{-1} \\
        \end{bmatrix}
        \begin{bmatrix}
            \bm k_{1,n} \\
            \vdots \\
            \bm k_{s-1,n} \\
            L \bm z_{n+1} \\
        \end{bmatrix}
        =
        \begin{bmatrix}
            L^{-1}\bm k_{1,n} \\
            \vdots \\
            L^{-1}\bm k_{s-1,n} \\
            \bm z_{n+1} \\
        \end{bmatrix},
    \end{equation*}
    which shows that $[\widehat{\bm z}_{n+1}]_{2r+1,2r+2} = \bm z_{n+1}$, and the claim follows.
\end{proof}

\subsection{Polynomial degree \texorpdfstring{$r=1$}{r=1}}\label{sec:P1_1st}
For $\mathbb{V}_{n}^1$ we consider the same basis functions as in Section~\ref{sec:P1}, and use the GLL nodes and weights, i.e., 
\begin{equation*}
    \{x_1, x_2\} = \{-1,1\}, \quad \{w_1, w_2 \} = \{1,1\}. 
\end{equation*}
The matrices $N_4$ and $N_5$ are then given by
\begin{equation}\label{eq:mat:p1}
    N_4 = \Delta t \begin{bmatrix}
     1/2 &  -1/2 \\ 1/2 & \phantom{-}1/2  
    \end{bmatrix}, \quad 
    N_5 =  \begin{bmatrix}
        0 & -1 \\ 0 & -1
    \end{bmatrix},
\end{equation}
which corresponds to the Butcher tableau
\begin{equation*}
\renewcommand\arraystretch{1.2}
\begin{array}{c|cc}
0 & 1/2 & -1/2\\
1 & 1/2 & \phantom{-}1/2 \\
\hline
& 1/2 & \phantom{-}1/2
\end{array}
\end{equation*}
for \eqref{eq:rk_kroneker}, i.e., the $2$-stage Lobatto IIIC method, see \cite{Jay2015}.

\begin{corollary}[DG is a RK Lobatto IIIC - $\mathbb{P}^1$]\label{coro:RKLIIIC_r1}
    Let $r=1$. Denote by $\{ (\bm k_n , \bm z_n) \}_n$, with $\bm k_n = [\bm k_{1,n},\bm k_{2,n}]^\top$, the sequence produced by the Lobatto IIIC scheme \eqref{eq:RK_extended_0}-\eqref{eq:RK_extended} starting from $\bm z_0$, and by $\{ \widehat{\bm z}_n \}_n$ the sequence generated by the DG scheme \eqref{eq:dg_kroneker} starting from $\widehat{\bm z}_0$.
    If $[\widehat{\bm z}_0]_{3,4} = \bm z_0$, then 
    \begin{equation}\label{eq:equivalenceP1}
        [\widehat{\bm z}_{n}]_{3,4} = \bm z_{n}
        \quad \text{and} \quad
        \widehat{\bm z}_n = (I_{2} \otimes L^{-1}) \bm k_{n-1}
        \quad \text{for $n=1,\dots,N$}.
    \end{equation}
\end{corollary}

\begin{proof}
    The result follows by Theorem~\ref{thm:equiv}. In fact, the coefficients $b_j$ and $a_{j,i}$ and the matrix $N^5$ (given in \eqref{eq:mat:p1}) clearly satisfy the hypotheses of Theorem~\ref{thm:equiv}. Moreover, a direct calculation reveals that $K:= (I_s \otimes I_2 - \Delta t A \otimes L)$ is
    \begin{equation*}
        K = \begin{bmatrix}
        1 & -\frac{\Delta t}{2} & 0 & \frac{\Delta t}{2}\\
        \frac{\Delta t \lambda}{2} & 1 & -\frac{\Delta t \lambda}{2} & 0\\
        0 & -\frac{\Delta t}{2} & 1 & -\frac{\Delta t}{2}\\
        \frac{\Delta t \lambda}{2} & 0 & \frac{\Delta t \lambda}{2} & 1
        \end{bmatrix},
    \end{equation*}
    and ${\rm det} \, K = \frac{4+\lambda^2 \Delta t^4}{4}>0$.
\end{proof}

\subsection{Polynomial degree \texorpdfstring{$r=2$}{r=2}}\label{sec:P2_1st}
For $\mathbb{V}_{n}^2$ we consider the same basis functions as in Section~\ref{sec:P2}, and use again the GLL nodes and weights, i.e., 
\begin{equation*}
    \{x_1, x_2, x_3\} = \{-1,0,1\}, \quad \{w_1, w_2, w_3\} = \{1/3,4/3,1/3\} 
\end{equation*}
The matrices $N_4$ and $N_5$ then become
\begin{equation}\label{eq:mat:p2}
    N_4 = \Delta t \begin{bmatrix}
    1/6 & -1/3  & \phantom{-}1/6 \\ 1/6 &  \phantom{-}5/12  & -1/12  \\   
       1/6 & \phantom{-}2/3  & \phantom{-}1/6    
    \end{bmatrix}, \quad 
    N_5 = \begin{bmatrix}
        0 & 0 & -1 \\ 0 & 0 & -1 \\ 0 & 0 & -1
    \end{bmatrix},
\end{equation}
which corresponds to the Butcher tableau
\begin{equation*}
\renewcommand\arraystretch{1.2}
\begin{array}{c|ccc}
0 & 1/6 & -1/3  & \phantom{-}1/6 \\
1/2 & 1/6 &  \phantom{-}5/12  & -1/12\\
1 & 1/6 & \phantom{-}2/3  & \phantom{-}1/6 \\
\hline
& 1/6 & \phantom{-}2/3 & \phantom{-}1/6 
\end{array}
\end{equation*}
for \eqref{eq:rk_kroneker}, i.e., the $3$-stage Lobatto IIIC method, see \cite{Jay2015}.

\begin{corollary}[DG is a RK Lobatto IIIC - $\mathbb{P}^2$]\label{coro:RKLIIIC_r2}
    Let $r=2$. Denote by $\{ (\bm k_n , \bm z_n) \}_n$, with $\bm k_n = [\bm k_{1,n},\bm k_{2,n}]^\top$, the sequence produced by the Lobatto IIIC scheme \eqref{eq:RK_extended_0}-\eqref{eq:RK_extended} starting from $\bm z_0$, and by $\{ \widehat{\bm z}_n \}_n$ the sequence generated by the DG scheme \eqref{eq:dg_kroneker} starting from $\widehat{\bm z}_0$.
    If $[\widehat{\bm z}_0]_{5,6} = \bm z_0$, then 
    \begin{equation}\label{eq:equivalenceP2}
        [\widehat{\bm z}_{n}]_{5,6} = \bm z_{n}
        \quad \text{and} \quad
        \widehat{\bm z}_n = (I_{3} \otimes L^{-1}) \bm k_{n-1}
        \quad \text{for $n=1,\dots,N$}.
    \end{equation}
\end{corollary}

\begin{proof}
    The result follows again by Theorem~\ref{thm:equiv}. In fact, the coefficients $b_j$ and $a_{j,i}$ and the matrix $N^5$ (given in \eqref{eq:mat:p2}) clearly fulfill the hypotheses of Theorem~\ref{thm:equiv}, and a direct calculation reveals that $K:= (I_s \otimes I_2 - \Delta t A \otimes L)$ is
    \begin{equation*}
        K = \begin{bmatrix}
        1 & -\frac{\Delta t}{6} & 0 & \frac{\Delta t}{3} & 0 & -\frac{\Delta t}{6}\\
\frac{\Delta t \lambda}{6} & 1 & -\frac{\Delta t \lambda}{3} & 0 & \frac{\Delta t \lambda}{6} & 0\\
0 & -\frac{\Delta t}{6} & 1 & -\frac{5 \Delta t}{12} & 0 & \frac{\Delta t}{12}\\
\frac{\Delta t \lambda}{6} & 0 & \frac{5 \Delta t \lambda}{12} & 1 & -\frac{\Delta t \lambda}{12} & 0\\
0 & -\frac{\Delta t}{6} & 0 & -\frac{2 \Delta t}{3} & 1 & -\frac{\Delta t}{6}\\
\frac{\Delta t \lambda}{6} & 0 & \frac{2 \Delta t \lambda}{3} & 0 & \frac{\Delta t \lambda}{6} & 1\\
        \end{bmatrix},
    \end{equation*}
    and ${\rm det} \, K = \frac{576+36\lambda \Delta t^2+\lambda^3 \Delta t^6}{576}>0$.
\end{proof}

\subsection{Polynomial degree \texorpdfstring{$r=3$}{r=3}}\label{sec:P3_1st}
For $\mathbb{V}_{n}^3$ we consider the same basis functions as in Section~\ref{sec:P2}, and use the GLL nodes and weights
\begin{equation*}
    \{x_1, x_2, x_3, x_4\} = \{-1,-1/\sqrt{5},1/\sqrt{5},1\}, \quad \{w_1, w_2, w_3, w_4\} = \{1/6,5/6,5/6,1/6\}. 
\end{equation*}
The matrices $N_4$ and $N_5$ become
\begin{equation}\label{eq:mat:p3}
    N_4 = \Delta t \begin{bmatrix}
    1/12 & -\sqrt{5}/12  & \sqrt{5}/12 & -1/12 \\ 1/12 & \phantom{-}1/4  & (10-7\sqrt{5})/60 & \phantom{-}\sqrt{5}/60 \\ 
1/12 & (10+7\sqrt{5})/60  & 1/4 & -\sqrt{5}/60 \\
1/12 & \phantom{-}5/12 & 5/12 & \phantom{-}1/12 
\end{bmatrix}, \quad 
    N_5 = \begin{bmatrix}
        0 & 0 & 0 & -1 \\ 0 & 0 & 0 & -1 \\ 0 &  0 & 0 & -1 \\0 & 0 & 0 & -1   
        \end{bmatrix},
\end{equation}
corresponding to the Butcher tableau
\begin{equation*}
\renewcommand\arraystretch{1.2}
\begin{array}{c|cccc}
 0 &  1/12 & -\sqrt{5}/12  & \sqrt{5}/12 & -1/12 \\ 1/2 - \sqrt{5}/10 & 1/12 & \phantom{-}1/4  & (10-7\sqrt{5})/60 & \phantom{-}\sqrt{5}/60 \\ 
1/2 + \sqrt{5}/10 & 1/12 & (10+7\sqrt{5})/60  & 1/4 & -\sqrt{5}/60 \\
1 & 1/12 & \phantom{-}5/12 & 5/12 & \phantom{-}1/12 \\
\hline
& 1/12 & 5/12 & 5/12 & 1/12  
\end{array}
\end{equation*}
for \eqref{eq:rk_kroneker}, i.e., the $4$-stage Lobatto IIIC method, see \cite{Jay2015}.

\begin{corollary}[DG is a RK Lobatto IIIC - $\mathbb{P}^3$]\label{coro:RKLIIIC_r3}
    Let $r=3$. Denote by $\{ (\bm k_n , \bm z_n) \}_n$, with $\bm k_n = [\bm k_{1,n},\bm k_{2,n}]^\top$, the sequence produced by the Lobatto IIIC scheme \eqref{eq:RK_extended_0}-\eqref{eq:RK_extended} starting from $\bm z_0$, and by $\{ \widehat{\bm z}_n \}_n$ the sequence generated by the DG scheme \eqref{eq:dg_kroneker} starting from $\widehat{\bm z}_0$.
    If $[\widehat{\bm z}_0]_{7,8} = \bm z_0$, then 
    \begin{equation}\label{eq:equivalenceP3}
        [\widehat{\bm z}_{n}]_{7,8} = \bm z_{n}
        \quad \text{and} \quad
        \widehat{\bm z}_n = (I_{4} \otimes L^{-1}) \bm k_{n-1}
        \quad \text{for $n=1,\dots,N$}.
    \end{equation}
\end{corollary}

\begin{proof}
    The result follows again by Theorem~\ref{thm:equiv}: the coefficients $b_j$ and $a_{j,i}$ and the matrix $N^5$ (given in \eqref{eq:mat:p3}) satisfy the hypotheses of Theorem~\ref{thm:equiv}, and a direct calculation reveals that $K:= (I_s \otimes I_4 - \Delta t A \otimes L)$ is
    \begin{equation*}
        K = \begin{bmatrix}
             1& -\frac{\Delta t}{12}&                          0&            \frac{\sqrt{5} \Delta t}{12}&                           0&          -\frac{\sqrt{5} \Delta t}{12}&                 0&      \frac{\Delta t}{12} \\
\frac{\Delta t \lambda}{12}&      1&          -\frac{\sqrt{5} \Delta t \lambda}{12}&                    0&            \frac{\sqrt{5} \Delta t \lambda}{12}&                   0&   -\frac{\Delta t \lambda}{12}&          0\\
             0& -\frac{\Delta t}{12}&                          1&                -\frac{\Delta t}{4}&                           0& a_-&                 0& -\frac{\sqrt{5} \Delta t}{60} \\
\frac{\Delta t \lambda}{12}&      0&      \frac{\Delta t \lambda}{4}&                    1& -a_-&                   0&  \frac{\sqrt{5} \Delta t \lambda}{60}&          0\\
             0& -\frac{\Delta t}{12}&                          0& -a_+&                           1&               -\frac{\Delta t}{4}&                 0&  \frac{\sqrt{5} \Delta t}{60}\\
\frac{\Delta t \lambda}{12}&      0& a_+&                    0&               \frac{\Delta t \lambda}{4}&                   1& -\frac{ \Delta t \lambda}{60}&          0\\
             0& -\frac{\Delta t}{12}&                          0&           -\frac{5 \Delta t}{12}&                           0&          -\frac{5 \Delta t}{12}&                 1&     -\frac{\Delta t}{12}\\
\frac{\Delta t \lambda}{12}&      0&           \frac{5 \Delta t \lambda}{12}&                    0&            \frac{5 \Delta t \lambda}{12}&                   0&      \frac{\Delta t \lambda}{12}&          1\\
        \end{bmatrix},
    \end{equation*}
    where $a_- = \Delta t \lambda (7 \sqrt{5}-10)$, $a_+ = \Delta t \lambda (7 \sqrt{5}+10)$
    and 
$    {\rm det} \, K = \frac{\Delta t^8 \lambda^4}{129600} %\Delta t^6 \lambda^3(11/34560  - 19/51840 + 1/20736) 
    %+ \Delta t^4 \lambda^2 (60025/12960000 - 605/51840 + 169/20736)
    + \frac{\Delta t^4 \lambda^2}{900} 
   + \frac{2 \Delta t^2 \lambda}{45} + 1 > 0 
$
\end{proof}

%%%%%%%%%%%%%%%%%%%%%%%%%%%%%%%%%%%%%%%%
\section{Numerical results: DG2 vs DG1}\label{sec:comparison}
%%%%%%%%%%%%%%%%%%%%%%%%%%%%%%%%%%%%%%%%
We now compare the DG2 method proposed in Section \ref{sec:dG:second} and the DG1 in Section \ref{sec:dG:first} numerically. For the model problem \eqref{eq:modelProblem}, we consider the regular solution $u_{ex} = \cos(t) + \sin(t)$ in the time interval $(0,20]$ 
and set $\lambda=1$, $\widehat{u}_0 = 1$ and $\widehat{u}_1 = 1$. For the DG2 scheme, we show in Figure \ref{fig:conv_test_dt} 
\begin{figure}[t]
\begin{tikzpicture}

\begin{axis}[%
width=0.4\textwidth,
height=0.4\textwidth,
scale only axis,
xmode=log,
xmin=0.05,
xmax=1,
xminorticks=true,
xlabel style={font=\color{white!15!black}},
xlabel={$\Delta  t $},
ymode=log,
ymin=2.e-07,
ymax=50,
yminorticks=true,
ylabel style={font=\color{white!15!black}},
title={$\displaystyle{\max_{n=0,...,N-1}} \|\bm e_{n} \|, \; s = 0$},
axis background/.style={fill=white},
xmajorgrids,
xminorgrids,
ymajorgrids,
yminorgrids,
legend style={at={(0.6,0.018)}, anchor=south west, legend cell align=left, align=left, draw=white!15!black}
]
\addplot [color=blue, line width=2.0pt, mark=o, mark options={solid, blue}]
  table[row sep=crcr]{%
1	1.31200868671537\\
0.5	0.978485408201263\\
0.4	0.86077520619985\\
0.2	0.529683457867475\\
0.1	0.296121038227356\\
0.05	0.156774071752195\\
};
\addlegendentry{$r=1$}

\addplot [color=green, line width=2.0pt, mark=+, mark options={solid, green}]
  table[row sep=crcr]{%
1	0.0920055457793678\\
0.5	0.0161025813597016\\
0.4	0.00991662178863334\\
0.2	0.0028920463851152\\
0.1	0.000777457570079521\\
0.05	0.000201288491896823\\
};
\addlegendentry{$r=2$}

\addplot [color=black, line width=2.0pt, mark=square, mark options={solid, black}]
  table[row sep=crcr]{%
1	0.0112846429614378\\
0.5	0.0013283103927112\\
0.4	0.000657750333485363\\
0.2	7.51586142326843e-05\\
0.1	8.88712910340761e-06\\
0.05	1.07728720610822e-06\\
};
\addlegendentry{$r=3$}

%% grado 3
\addplot [color=black, forget plot]
  table[row sep=crcr]{%
0.1	   5.e-06    \\
0.05   6.2500e-07\\
};

\addplot [color=black, forget plot]
  table[row sep=crcr]{%
0.1	   6.2500e-07    \\
0.05   6.2500e-07\\
};
\addplot [color=black, forget plot]
  table[row sep=crcr]{%
0.1	  5.e-06    \\
0.1   6.2500e-07\\
};

%% grado 2
\addplot [color=black, forget plot]
  table[row sep=crcr]{%
0.1	   0.0005    \\
0.05   1.2500e-04\\
};

\addplot [color=black, forget plot]
  table[row sep=crcr]{%
0.1	   1.2500e-04    \\
0.05   1.2500e-04\\
};
\addplot [color=black, forget plot]
  table[row sep=crcr]{%
0.1	  0.0005    \\
0.1   1.2500e-04\\
};

%% grado 1 
\addplot [color=black, forget plot]
  table[row sep=crcr]{%
0.1	   0.2    \\
0.05   0.1\\
};

\addplot [color=black, forget plot]
  table[row sep=crcr]{%
0.1	   0.1 \\
0.05   0.1 \\
};
\addplot [color=black, forget plot]
  table[row sep=crcr]{%
0.1	  0.1    \\
0.1   0.2\\
};

\node[right, align=left, text=black, font=\normalsize]
at (axis cs:0.1,0.15) {$1$};

\node[right, align=left, text=black, font=\normalsize]
at (axis cs:0.1,0.0003) {$2$};

\node[right, align=left, text=black, font=\normalsize]
at (axis cs:0.1,2.e-6) {$3$};

\end{axis}
\end{tikzpicture}% 
\begin{tikzpicture}

\begin{axis}[%
width=0.4\textwidth,
height=0.4\textwidth,
scale only axis,
xmode=log,
xmin=0.05,
xmax=1,
xminorticks=true,
xlabel style={font=\color{white!15!black}},
xlabel={$\Delta  t $},
ymode=log,
ymin=2.e-7,
ymax=50,
yminorticks=true,
ylabel style={font=\color{white!15!black}},
title={$\displaystyle{\max_{n=0,...,N-1}} \|\bm e_{n} \|, \; s=\frac{\Delta t^2}{2}$},
axis background/.style={fill=white},
xmajorgrids,
xminorgrids,
ymajorgrids,
yminorgrids,
legend style={at={(0.6,0.018)}, anchor=south west, legend cell align=left, align=left, draw=white!15!black}
]

\addplot [color=blue, line width=2.0pt, mark=o, mark options={solid, blue}]
  table[row sep=crcr]{%
1	0.888366981793343\\
0.5	0.250950790332436\\
0.4	0.165992571785768\\
0.2	0.0423876109919296\\
0.1	0.0107201590458514\\
0.05	0.0026890783661675\\
};
\addlegendentry{$r=1$ }

\addplot [color=green, line width=2.0pt, mark=+, mark options={solid, green}]
  table[row sep=crcr]{%
     1.000000000000000e+00     1.618875764418156e+01\\
     5.000000000000000e-01     3.160268132374196e+00\\
     4.000000000000000e-01     2.193440855772271e+00\\
     2.000000000000000e-01     8.440587526510224e-01\\
     1.000000000000000e-01     3.734893742548486e-01\\
     5.000000000000000e-02     1.760084136358375e-01\\
     2.000000000000000e-02     6.798497951541171e-02\\
     1.000000000000000e-02     3.360186200305582e-02\\
};
\addlegendentry{$r=2$}

\addplot [color=black, line width=2.0pt, mark=square, mark options={solid, black}]
  table[row sep=crcr]{%
     1.000000000000000e+00     1.365703033386544e+01\\
     5.000000000000000e-01     3.079518981422390e+00\\
     4.000000000000000e-01     2.157962360538667e+00\\
     2.000000000000000e-01     8.399030910906271e-01\\
     1.000000000000000e-01     3.727926926168081e-01\\
     5.000000000000000e-02     1.758672437085544e-01\\
     2.000000000000000e-02     6.796511149399009e-02\\
     1.000000000000000e-02     3.359712515910118e-02\\
};
\addlegendentry{$r=3$}

%% grado 1 corretto
\addplot [color=black, forget plot]
  table[row sep=crcr]{%
0.1	   0.008    \\
0.05   0.0020\\
};

\addplot [color=black, forget plot]
  table[row sep=crcr]{%
0.1	   0.0020 \\
0.05   0.0020 \\
};
\addplot [color=black, forget plot]
  table[row sep=crcr]{%
0.1	  0.008    \\
0.1   0.0020\\
};

%% grado 1 
\addplot [color=black, forget plot]
  table[row sep=crcr]{%
0.1	   0.2    \\
0.05   0.1\\
};

\addplot [color=black, forget plot]
  table[row sep=crcr]{%
0.1	   0.1 \\
0.05   0.1 \\
};
\addplot [color=black, forget plot]
  table[row sep=crcr]{%
0.1	  0.1    \\
0.1   0.2\\
};

\node[right, align=left, text=black, font=\normalsize]
at (axis cs:0.1,0.15) {$1$};

\node[right, align=left, text=black, font=\normalsize]
at (axis cs:0.1,0.005) {$2$};

%\node[right, align=left, text=black, font=\normalsize]
%at (axis cs:0.1,0.0003) {$2$};

%\node[right, align=left, text=black, font=\normalsize]
%at (axis cs:0.1,2.e-6) {$3$};

\end{axis}
\end{tikzpicture}%      
\caption{DG2: computed convergence errors $\| \bm e_n \|$ as a function of the time step $\Delta t$ for $r=1,2,3$ with $s=0$ (left) and $s=\frac{\Delta t^2}{2}$ (right).}\label{fig:conv_test_dt}
\end{figure}
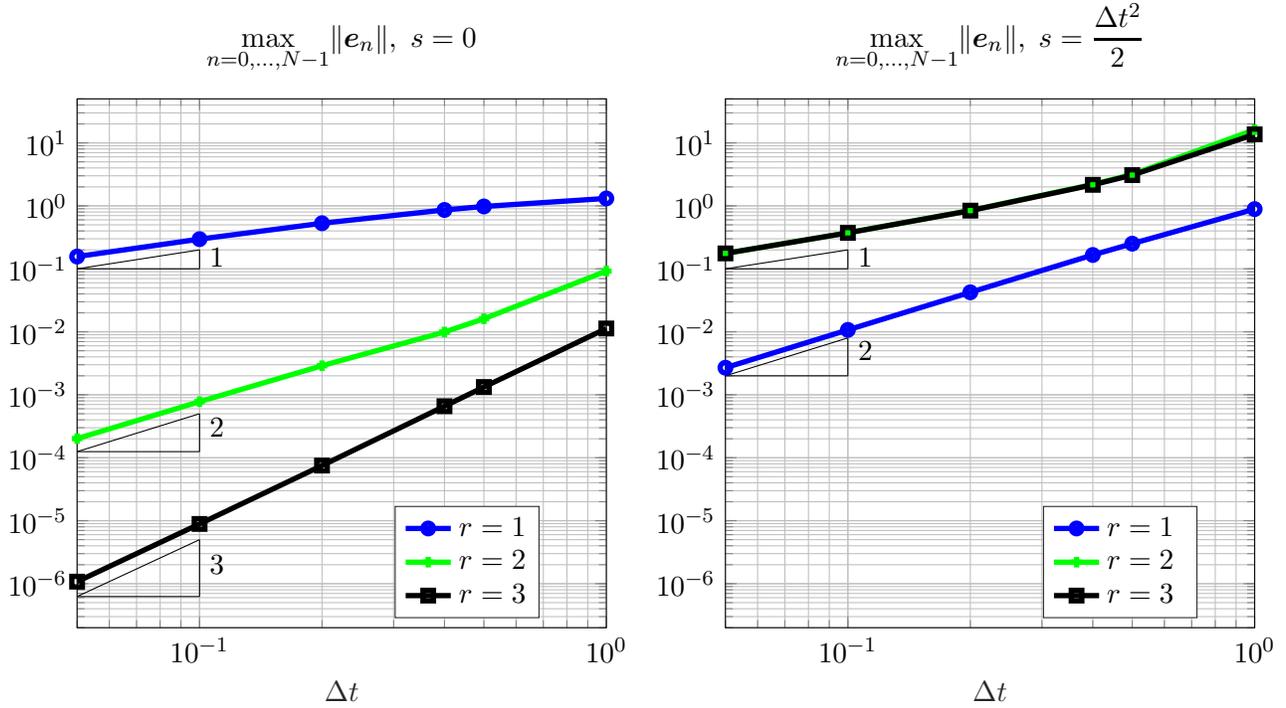
the computed errors $\| \bm e_n \| = \| \bm u_n^{ex} - \bm u_n \|$ as a function of the time step $\Delta t$, by varying the polynomial degree $r=1,2,3$. In Figure \ref{fig:conv_test_dt} we select $s=0$ (left) and $s=\frac{\Delta t^2}{2}$ (right) in \eqref{eq:intStrong}.
The numerical results confirm the theoretical ones in Lemma \ref{lemma:general}. In particular, for $s=0$ we can observe an order of convergence of $\mathcal{O}(\Delta t^r)$, cf. Figure \ref{fig:conv_test_dt}-left, while for $s=\frac{\Delta t^2}{2}$ the order of convergence is $\mathcal{O}(\Delta t^2)$ for $r=1$ and $\mathcal{O}(\Delta t)$ for $r=2,3$. This agrees with the findings in Section \ref{sec:dG:second}. 
Next, to further illustrate these findings, we plot in Figure \ref{fig:cons_test_dt} 
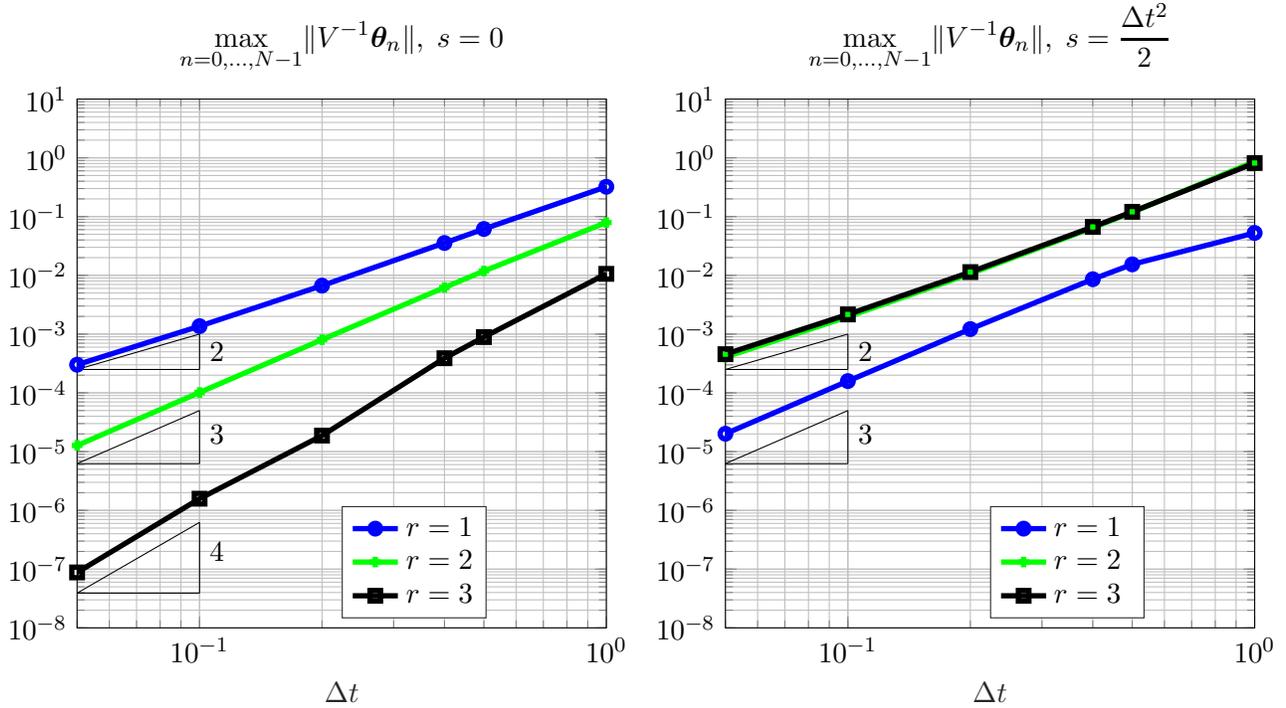
\begin{figure}[t]
\begin{tikzpicture}

\begin{axis}[%
width=0.4\textwidth,
height=0.4\textwidth,
scale only axis,
xmode=log,
xmin=0.05,
xmax=1,
xminorticks=true,
xlabel style={font=\color{white!15!black}},
xlabel={$\Delta  t $},
ymode=log,
ymin=1.e-08,
ymax=10,
yminorticks=true,
ylabel style={font=\color{white!15!black}},
title={$\displaystyle{\max_{n=0,...,N-1}} \|V^{-1}\bm \theta_{n} \|, \; s=0$},
axis background/.style={fill=white},
xmajorgrids,
xminorgrids,
ymajorgrids,
yminorgrids,
legend style={at={(0.5,0.018)}, anchor=south west, legend cell align=left, align=left, draw=white!15!black}
]
\addplot [color=blue, line width=2.0pt, mark=o, mark options={solid, blue}]
  table[row sep=crcr]{%
     1.000000000000000e+00     3.215087214463894e-01\\
     5.000000000000000e-01     6.148337551915015e-02\\
     4.000000000000000e-01     3.554159256042001e-02\\
     2.000000000000000e-01     6.680318042508684e-03\\
     1.000000000000000e-01     1.363775463528934e-03\\
     5.000000000000000e-02     3.012322237661073e-04\\
     2.000000000000000e-02     4.432699555301482e-05\\
     1.000000000000000e-02     1.075719540393559e-05\\
};
\addlegendentry{$r=1$}

\addplot [color=green, line width=2.0pt, mark=+, mark options={solid, green}]
  table[row sep=crcr]{%
     1.000000000000000e+00     7.930495295384118e-02\\
     5.000000000000000e-01     1.183119285702790e-02\\
     4.000000000000000e-01     6.205310096972965e-03\\
     2.000000000000000e-01     8.032384422272670e-04\\
     1.000000000000000e-01     1.014708198191698e-04\\
     5.000000000000000e-02     1.272833374343739e-05\\
     2.000000000000000e-02     8.154292144582024e-07\\
     1.000000000000000e-02     1.020450510583757e-07\\
};
\addlegendentry{$r=2$}

\addplot [color=black, line width=2.0pt, mark=square, mark options={solid, black}]
  table[row sep=crcr]{%
     1.000000000000000e+00     1.059838802624262e-02\\
     5.000000000000000e-01     8.901966270799954e-04\\
     4.000000000000000e-01     3.886298317492631e-04\\
     2.000000000000000e-01     1.873190052496465e-05\\
     1.000000000000000e-01     1.579491525944934e-06\\
     5.000000000000000e-02     8.786975412908725e-08\\
     2.000000000000000e-02     2.890761817163110e-09\\
     1.000000000000000e-02     9.363188880974796e-10\\
};
\addlegendentry{$r=3$}

%% grado 3
\addplot [color=black, forget plot]
  table[row sep=crcr]{%
0.1	   6.2500e-07     \\
0.05   3.906250000000000e-08\\
};

\addplot [color=black, forget plot]
  table[row sep=crcr]{%
0.1	   3.906250000000000e-08  \\
0.05   3.906250000000000e-08\\
};
\addplot [color=black, forget plot]
  table[row sep=crcr]{%
0.1	  6.2500e-07    \\
0.1   3.906250000000000e-08\\
};

%% grado 2
\addplot [color=black, forget plot]
  table[row sep=crcr]{%
0.1	   5e-5    \\
0.05    6.250000000000000e-06\\
};

\addplot [color=black, forget plot]
  table[row sep=crcr]{%
0.1	    6.250000000000000e-06    \\
0.05    6.250000000000000e-06\\
};
\addplot [color=black, forget plot]
  table[row sep=crcr]{%
0.1	  5e-5     \\
0.1    6.250000000000000e-06\\
};

\addplot [color=black, forget plot]
  table[row sep=crcr]{%
0.1	   1e-3    \\
0.05   2.500000000000000e-04\\
};

\addplot [color=black, forget plot]
  table[row sep=crcr]{%
0.1	   2.500000000000000e-04    \\
0.05   2.500000000000000e-04\\
};
\addplot [color=black, forget plot]
  table[row sep=crcr]{%
0.1	  1e-3     \\
0.1    2.500000000000000e-04\\
};

\node[right, align=left, text=black, font=\normalsize]
at (axis cs:0.1,5.e-4) {$2$};

\node[right, align=left, text=black, font=\normalsize]
at (axis cs:0.1,2.e-5) {$3$};

\node[right, align=left, text=black, font=\normalsize]
at (axis cs:0.1,2.e-7) {$4$};

\end{axis}
\end{tikzpicture}% 
\begin{tikzpicture}

\begin{axis}[%
width=0.4\textwidth,
height=0.4\textwidth,
scale only axis,
xmode=log,
xmin=0.05,
xmax=1,
xminorticks=true,
xlabel style={font=\color{white!15!black}},
xlabel={$\Delta  t $},
ymode=log,
ymin=1.e-08,
ymax=10,
yminorticks=true,
ylabel style={font=\color{white!15!black}},
title={$\displaystyle{\max_{n=0,...,N-1}} \|V^{-1}\bm \theta_{n} \|, \; s = \frac{\Delta t^2}{2}$},
axis background/.style={fill=white},
xmajorgrids,
xminorgrids,
ymajorgrids,
yminorgrids,
legend style={at={(0.5,0.018)}, anchor=south west, legend cell align=left, align=left, draw=white!15!black}
]

\addplot [color=blue, line width=2.0pt, mark=o, mark options={solid, blue}]
  table[row sep=crcr]{%
     1.000000000000000e+00     5.266151662454171e-02\\
     5.000000000000000e-01     1.537278580340113e-02\\
     4.000000000000000e-01     8.591173995249972e-03\\
     2.000000000000000e-01     1.217443187621042e-03\\
     1.000000000000000e-01     1.585414264712073e-04\\
     5.000000000000000e-02     2.012672949790195e-05\\
     2.000000000000000e-02     1.298113960809742e-06\\
     1.000000000000000e-02     1.626473532775534e-07\\
};
\addlegendentry{$r=1$ }

\addplot [color=green, line width=2.0pt, mark=+, mark options={solid, green}]
  table[row sep=crcr]{%
     1.000000000000000e+00     8.672988047149610e-01\\
     5.000000000000000e-01     1.183006753595093e-01\\
     4.000000000000000e-01     6.461379101647341e-02\\
     2.000000000000000e-01     1.065152306434117e-02\\
     1.000000000000000e-01     1.972898147236422e-03\\
     5.000000000000000e-02     4.065155544300030e-04\\
     2.000000000000000e-02     5.668853712505187e-05\\
     1.000000000000000e-02     1.347769177539223e-05\\
     };
\addlegendentry{$r=2$}

\addplot [color=black, line width=2.0pt, mark=square, mark options={solid, black}]
  table[row sep=crcr]{%
     1.000000000000000e+00     8.152986930836548e-01\\
     5.000000000000000e-01     1.204810351228196e-01\\
     4.000000000000000e-01     6.666512074651516e-02\\
     2.000000000000000e-01     1.135656291111717e-02\\
     1.000000000000000e-01     2.161978862547223e-03\\
     5.000000000000000e-02     4.549551572251609e-04\\
     2.000000000000000e-02     6.453450305135409e-05\\
     1.000000000000000e-02     1.537109423556985e-05\\
};
\addlegendentry{$r=3$}

%% grado 2
\addplot [color=black, forget plot]
  table[row sep=crcr]{%
0.1	   5e-5    \\
0.05    6.250000000000000e-06\\
};

\addplot [color=black, forget plot]
  table[row sep=crcr]{%
0.1	    6.250000000000000e-06    \\
0.05    6.250000000000000e-06\\
};
\addplot [color=black, forget plot]
  table[row sep=crcr]{%
0.1	  5e-5     \\
0.1    6.250000000000000e-06\\
};

\addplot [color=black, forget plot]
  table[row sep=crcr]{%
0.1	   1e-3    \\
0.05   2.500000000000000e-04\\
};

\addplot [color=black, forget plot]
  table[row sep=crcr]{%
0.1	   2.500000000000000e-04    \\
0.05   2.500000000000000e-04\\
};
\addplot [color=black, forget plot]
  table[row sep=crcr]{%
0.1	  1e-3     \\
0.1    2.500000000000000e-04\\
};

\node[right, align=left, text=black, font=\normalsize]
at (axis cs:0.1,5.e-4) {$2$};

\node[right, align=left, text=black, font=\normalsize]
at (axis cs:0.1,2.e-5) {$3$};

\end{axis}
\end{tikzpicture}% 
\caption{DG2: computed consistency errors $\| V^{-1} \bm \theta_n \|$ as a function of the time step $\Delta t$ for $r=1,2,3$ with $s=0$ (left) and $s=\frac{\Delta t^2}{2}$ (right). }\label{fig:cons_test_dt}
\end{figure}
the consistency error $\| V^{-1}\bm \theta_n \| $ as a function of the time step $\Delta t$ by varying the polynomial degree $r=1,2,3$, with $s=0$ (left) and with $s=\frac{\Delta t^2}2$ (right). We see that for $s=0$ the asymptotic trend is $\mathcal{O}(\Delta t^{r+1})$  while for 
$ s = \frac{\Delta t^2}{2}$ we obtain $\mathcal{O}(\Delta t^2)$ for $r=2,3$ and $\mathcal{O}(\Delta t^3)$ for $r=1$. Again, this is aligned with the  theoretical results of Section \ref{sec:dG:second}.

For the DG1 method of Section \ref{sec:dG:first} we show the asymptotic rate of convergence with respect to the time step $\Delta t$ in Figure \ref{fig:convDG2_test_dt}. 
\begin{figure}[t]    
\begin{tikzpicture}

\begin{axis}[%
width=0.4\textwidth,
height=0.4\textwidth,
scale only axis,
xmode=log,
xmin=0.05,
xmax=1,
xminorticks=true,
xlabel style={font=\color{white!15!black}},
xlabel={$\Delta  t $},
ymode=log,
ymin=2.e-13,
ymax=5,
yminorticks=true,
ylabel style={font=\color{white!15!black}},
title={$\displaystyle{\max_{n=0,...,N-1}} \|\bm e_{n} \|$},
axis background/.style={fill=white},
xmajorgrids,
xminorgrids,
ymajorgrids,
yminorgrids,
legend style={at={(0.536,0.058)}, anchor=south west, legend cell align=left, align=left, draw=white!15!black}
]
\addplot [color=blue, line width=2.0pt, mark=o, mark options={solid, blue}]
  table[row sep=crcr]{%
   1.000000000000000   1.020748362436374\\
   0.500000000000000   0.677745419461844\\
   0.400000000000000   0.477124880806738\\
   0.200000000000000   0.127677834212889\\
   0.100000000000000   0.032094394287500\\
   0.050000000000000   0.008028301389498\\
};
\addlegendentry{$r=1$}

\addplot [color=green, line width=2.0pt, mark=+, mark options={solid, green}]
  table[row sep=crcr]{%
     1.000000000000000e+00     4.674991142266605e-02\\
     5.000000000000000e-01     4.807563900536966e-03\\
     4.000000000000000e-01     2.534635619727171e-03\\
     2.000000000000000e-01     3.291770208722622e-04\\
     1.000000000000000e-01     4.153653959085707e-05\\
     5.000000000000000e-02     5.204264958653255e-06\\
     };
\addlegendentry{$r=2$ }

\addplot [color=black, line width=2.0pt, mark=square, mark options={solid, black}]
  table[row sep=crcr]{%
     1.000000000000000e+00     2.700254959893789e-03\\
     5.000000000000000e-01     1.722951154617958e-04\\
     4.000000000000000e-01     7.079682577704194e-05\\
     2.000000000000000e-01     4.439530768340916e-06\\
     1.000000000000000e-01     2.777128599307233e-07\\
     5.000000000000000e-02     1.735990218243444e-08\\
};
\addlegendentry{$r=3$}

%% grado 3
\addplot [color=black, forget plot]
  table[row sep=crcr]{%
0.1	   1.e-07    \\
0.05   6.250e-09\\
};

\addplot [color=black, forget plot]
  table[row sep=crcr]{%
0.1	   6.250e-09    \\
0.05   6.250e-09\\
};
\addplot [color=black, forget plot]
  table[row sep=crcr]{%
0.1	  1.e-07    \\
0.1   6.250e-09\\
};

%% grado 2
\addplot [color=black, forget plot]
  table[row sep=crcr]{%
0.1	   2.e-5    \\
0.05   2.5e-06\\
};

\addplot [color=black, forget plot]
  table[row sep=crcr]{%
0.1	   2.5e-06   \\
0.05   2.5e-06\\
};
\addplot [color=black, forget plot]
  table[row sep=crcr]{%
0.1	  2.e-5    \\
0.1   2.5e-06\\
};

%% grado 1 
\addplot [color=black, forget plot]
  table[row sep=crcr]{%
0.1	   0.02    \\
0.05   0.005\\
};

\addplot [color=black, forget plot]
  table[row sep=crcr]{%
0.1	   0.005 \\
0.05   0.005 \\
};
\addplot [color=black, forget plot]
  table[row sep=crcr]{%
0.1	  0.02    \\
0.1   0.005\\
};

\node[right, align=left, text=black, font=\normalsize]
at (axis cs:0.1,0.009) {$2$};

\node[right, align=left, text=black, font=\normalsize]
at (axis cs:0.1,7e-6) {$3$};

\node[right, align=left, text=black, font=\normalsize]
at (axis cs:0.1,3.e-8) {$4$};

\end{axis}
\end{tikzpicture}% 
\begin{tikzpicture}

\begin{axis}[%
width=0.4\textwidth,
height=0.4\textwidth,
scale only axis,
xmode=log,
xmin=0.05,
xmax=1,
xminorticks=true,
xlabel style={font=\color{white!15!black}},
xlabel={$\Delta  t $},
ymode=log,
ymin=2.e-13,
ymax=5,
yminorticks=true,
ylabel style={font=\color{white!15!black}},
title={$ |\bm e_{N-1} |$},
axis background/.style={fill=white},
xmajorgrids,
xminorgrids,
ymajorgrids,
yminorgrids,
legend style={at={(0.536,0.058)}, anchor=south west, legend cell align=left, align=left, draw=white!15!black}
]
\addplot [color=blue, line width=2.0pt, mark=o, mark options={solid, blue}]
  table[row sep=crcr]{%
     1.000000000000000e+00     2.172848292893698e-01\\
     5.000000000000000e-01     1.566049869127559e-01\\
     4.000000000000000e-01     1.363546264660780e-01\\
     2.000000000000000e-01     4.808860237075863e-02\\
     1.000000000000000e-01     1.316976322284125e-02\\
     5.000000000000000e-02     3.403239858250862e-03   \\
};
\addlegendentry{$r=1$}

\addplot [color=green, line width=2.0pt, mark=+, mark options={solid, green}]
  table[row sep=crcr]{%
     1.000000000000000e+00     1.057765818854750e-02\\
     5.000000000000000e-01     9.102935292245595e-04\\
     4.000000000000000e-01     3.899826633072534e-04\\
     2.000000000000000e-01     2.631166646460681e-05\\
     1.000000000000000e-01     1.698529524807313e-06\\
     5.000000000000000e-02     1.077420093142933e-07\\
     };
\addlegendentry{$r=2$ }

\addplot [color=black, line width=2.0pt, mark=square, mark options={solid, black}]
  table[row sep=crcr]{%
     1.000000000000000e+00     8.248532143906306e-05\\
     5.000000000000000e-01     1.538518779709008e-06\\
     4.000000000000000e-01     4.148292453987068e-07\\
     2.000000000000000e-01     6.813004937988865e-09\\
     1.000000000000000e-01     1.088127365989067e-10\\
     5.000000000000000e-02     1.736832899723595e-12\\
};
\addlegendentry{$r=3$}

%% grado 3
\addplot [color=black, forget plot]
  table[row sep=crcr]{%
0.1	   5.e-011    \\
0.05   7.8125e-13\\
};

\addplot [color=black, forget plot]
  table[row sep=crcr]{%
0.1	   7.8125e-13    \\
0.05   7.8125e-13\\
};
\addplot [color=black, forget plot]
  table[row sep=crcr]{%
0.1	  5.e-011    \\
0.1   7.8125e-13\\
};

%% grado 2
\addplot [color=black, forget plot]
  table[row sep=crcr]{%
0.1	   8.e-7    \\
0.05   5.0000e-08\\
};

\addplot [color=black, forget plot]
  table[row sep=crcr]{%
0.1	   5.0000e-08   \\
0.05    5.0000e-08\\
};
\addplot [color=black, forget plot]
  table[row sep=crcr]{%
0.1	  8.e-7    \\
0.1   5.0000e-08\\
};

%% grado 1 
\addplot [color=black, forget plot]
  table[row sep=crcr]{%
0.1	   0.005    \\
0.05   1.250000000000000e-03\\
};

\addplot [color=black, forget plot]
  table[row sep=crcr]{%
0.1	   1.250000000000000e-03 \\
0.05   1.250000000000000e-03 \\
};
\addplot [color=black, forget plot]
  table[row sep=crcr]{%
0.1	  1.250000000000000e-03    \\
0.1   0.005\\
};

\node[right, align=left, text=black, font=\normalsize]
at (axis cs:0.1,0.003) {$2$};

\node[right, align=left, text=black, font=\normalsize]
at (axis cs:0.1,2e-7) {$4$};

\node[right, align=left, text=black, font=\normalsize]
at (axis cs:0.1,5.e-12) {$6$};

\end{axis}
\end{tikzpicture}% 
\caption{
DG1: computed convergence errors $\| \bm e_n \|$ (left) and $| \bm e_{N-1}|$ (right) as a function of the time step $\Delta t$ for $r=1,2,3$.}\label{fig:convDG2_test_dt}
\end{figure}
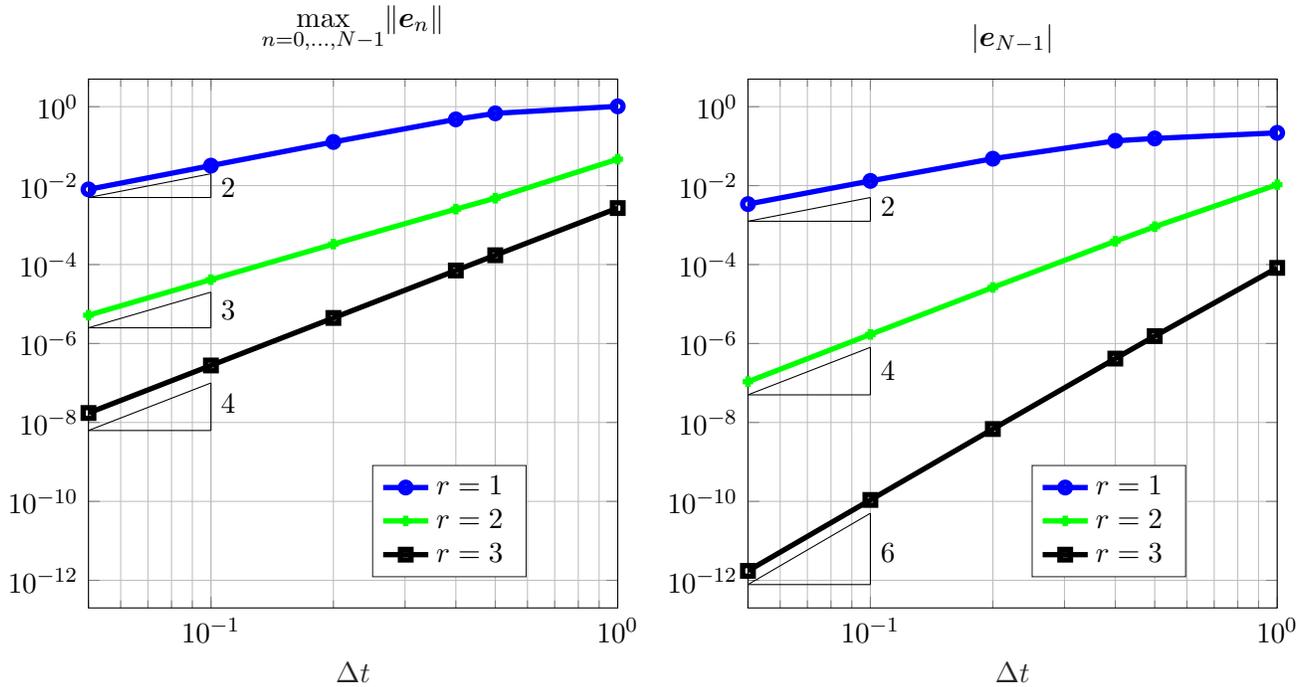
As shown is Section \ref{sec:review:RK} we obtain the rate of convergence $\mathcal{O}(\Delta t^{r+1})$ for $r=1,2,3$ when computing $\| \bm e_n \|$, cf. Figure \ref{fig:convDG2_test_dt}-left, and $\mathcal{O}(\Delta t^{2(r+1)-2})$ when considering the error at the final time $T$, cf. Figure \ref{fig:convDG2_test_dt}-right.
As a general remark, we can say that DG1 outperforms DG2 by retaining a higher order of accuracy for the same polynomial degree $r$. 
Finally, we compute the condition number of the system matrix $A^+$ in \eqref{eq:simple_local_system2} and $B^+$ in \eqref{eq:simple_local_system2_1st} and plot the results in Figure \ref{fig:cond_test_dt}. 
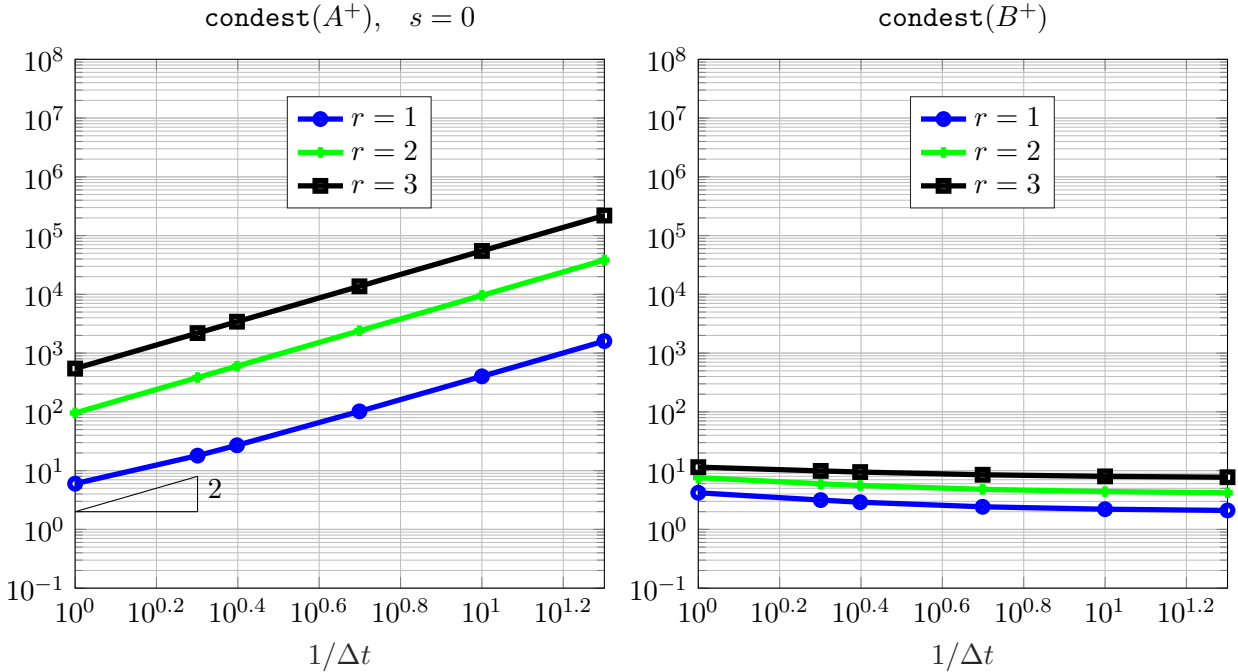
\begin{figure}[t]
\begin{tikzpicture}

\begin{axis}[%
width=0.4\textwidth,
height=0.4\textwidth,
scale only axis,
xmode=log,
xmin=1,
xmax=20,
xminorticks=true,
xlabel style={font=\color{white!15!black}},
xlabel={$1/\Delta  t $},
ymode=log,
ymin=0.1,
ymax=1e8,
yminorticks=true,
ylabel style={font=\color{white!15!black}},
title={\texttt{condest}($A^+$), \; $s=0$},
axis background/.style={fill=white},
xmajorgrids,
xminorgrids,
ymajorgrids,
yminorgrids,
legend style={at={(0.4,0.718)}, anchor=south west, legend cell align=left, align=left, draw=white!15!black}
]
\addplot [color=blue, line width=2.0pt, mark=o, mark options={solid, blue}]
  table[row sep=crcr]{%
     1     6.000000000000000e+00\\
     2     1.800000000000000e+01\\
     2.5     2.700000000000000e+01\\
     5     1.020000000000000e+02\\
     10     4.019999999999999e+02\\
     20     1.602000000000000e+03\\
};
\addlegendentry{$r=1$}

%\addplot [color=blue, dashed, line width=2.0pt, mark=o, mark options={solid, blue}]
%  table[row sep=crcr]{%
%     1.000000000000000e+00     5.500000000000000e+00\\
%     5.000000000000000e-01     1.750000000000000e+01\\
%     4.000000000000000e-01     2.650000000000000e+01\\
%     2.000000000000000e-01     1.015000000000000e+02\\
%     1.000000000000000e-01     4.014999999999999e+02\\
%     5.000000000000000e-02     1.601500000000000e+03\\
%     2.000000000000000e-02     1.000150000000000e+04\\
%     1.000000000000000e-02     4.000150000000000e+04\\
%};
%\addlegendentry{$r=1, a=\frac12$ }

\addplot [color=green, line width=2.0pt, mark=+, mark options={solid, green}]
  table[row sep=crcr]{%
     1.000000000000000e+00     9.600000000000018e+01\\
     2     3.840000000000035e+02\\
     2.5     5.999999999999912e+02\\
     5     2.399999999999863e+03\\
     10     9.599999999997266e+03\\
     20     3.839999999995633e+04\\
};
\addlegendentry{$r=2$}

\addplot [color=black, line width=2.0pt, mark=square, mark options={solid, black}]
  table[row sep=crcr]{%
     1.000000000000000e+00     5.462680440041535e+02\\
     2     2.191252515904163e+03\\
     2.5     3.424990869828469e+03\\
     5     1.370614381919647e+04\\
     10     5.483075561662664e+04\\
     20     2.193292028060948e+05\\
};
\addlegendentry{$r=3$}

%% grado 3
\addplot [color=black, forget plot]
  table[row sep=crcr]{%
1	   2    \\
2      8\\
};

\addplot [color=black, forget plot]
  table[row sep=crcr]{%
2	   2    \\
2      8\\
};

\addplot [color=black, forget plot]
  table[row sep=crcr]{%
1	  2    \\
2     2\\
};

\node[right, align=left, text=black, font=\normalsize]
at (axis cs:2,5) {$2$};

\end{axis}
\end{tikzpicture}% 
\begin{tikzpicture}

\begin{axis}[%
width=0.4\textwidth,
height=0.4\textwidth,
scale only axis,
xmode=log,
xmin=1,
xmax=20,
xminorticks=true,
xlabel style={font=\color{white!15!black}},
xlabel={$1/\Delta  t $},
ymode=log,
ymin=0.1,
ymax=1e8,
yminorticks=true,
ylabel style={font=\color{white!15!black}},
title={\texttt{condest}($B^+$)},
axis background/.style={fill=white},
xmajorgrids,
xminorgrids,
ymajorgrids,
yminorgrids,
legend style={at={(0.4,0.718)}, anchor=south west, legend cell align=left, align=left, draw=white!15!black}
]
\addplot [color=blue, line width=2.0pt, mark=o, mark options={solid, blue}]
  table[row sep=crcr]{%
     1     4.199999999999999e+00\\
     2     3.153846153846154e+00\\
     2.5     2.899841017488077e+00\\
     5     2.423430627748901e+00\\
     10     2.205469863253419e+00\\
     20     2.101310779201908e+00\\
};
\addlegendentry{$r=1$}

\addplot [color=green, line width=2.0pt, mark=+, mark options={solid, green}]
  table[row sep=crcr]{%
     1     7.595432300163133e+00\\
     2     5.968323495633129e+00\\
     2.5     5.588913785425494e+00\\
     5     4.801907586377658e+00\\
     10     4.401070219305521e+00\\
     20     4.200342049196949e+00\\
};
\addlegendentry{$r=2$}

\addplot [color=black, line width=2.0pt, mark=square, mark options={solid, black}]
  table[row sep=crcr]{%
     1     1.146541810017126e+01\\
     2     9.891216696153943e+00\\
     2.5     9.457009704373533e+00\\
     5     8.499820231402406e+00\\
     10     7.985562922643399e+00\\
     20     7.721492795719434e+00\\
};
\addlegendentry{$r=3$}

\end{axis}
\end{tikzpicture}% 
\caption{ 
 Computed condition number for the system matrix $A^+$ in \eqref{eq:simple_local_system2} (left) and $B^+$ in \eqref{eq:simple_local_system2_1st} (right).  }\label{fig:cond_test_dt}
\end{figure}
The results for $\Delta t$ tending to zero show that for the DG2 method, the condition number of $A^+$ behaves like $\mathcal{O}(\Delta t^{-2})$ while for the DG1 scheme the condition number of $B^+$ is almost constant. 
This result still favors the DG1 method, which is therefore preferable to the DG2 scheme.

%%%%%%%%%%%%%%%%%%%%%%%%%%%%%%%%%%%%%%%%
\section{Conclusions}\label{sec:conclusions}

We examined discontinuous Galerkin (DG) discretizations as time integrators for second-order differential systems, exploring two formulations: one in the original problem's second-order form, and another one in its first-order formulation. We presented a new convergence analysis framework for the second-order formulation, studying both consistency and accuracy of the resulting schemes. We also showed the equivalence between the DG formulation and other classical time integrator schemes such as the Newmark scheme and General Linear Methods. We also provided an in-depth review of the first-order formulation, complemented by a new algebraic proof that shows the equivalence of the DG method with the Lobatto IIIC schemes.
We compared the two different formulations in terms of accuracy, consistency, and computational cost. From the results obtained, we conclude that the DG formulation for first-order systems is to be preferred to that for the second-order equation, from every point of view.

%%%%%%%%%%%%%%%%%%%%%%%%%%%%%%%%%%%%%%%%

\section*{Acknowledgment}
The work of GC and IM has been partially supported by the PRIN2022 grant ASTICE - CUP: D53D23005710006. The present research is part of the activities of “Dipartimento di Eccelllenza 2023-2027”. GC and IM are members of INdAM-GNCS group.

\end{document}